\documentclass[10pt]{article}
\usepackage{enumerate}
\usepackage{amssymb,a4wide,latexsym,makeidx,epsfig,fleqn}
\usepackage{amsthm}
\usepackage{amsmath}
\usepackage{enumerate}
\usepackage{float}
\newtheorem{theorem}{Theorem}[section]

\newtheorem{lemma}[theorem]{Lemma}

\begin{document}
\textwidth 150mm \textheight 225mm
\title{The unicyclic graphs with the second smallest normalized Laplacian eigenvalue no less than $1-\frac{\sqrt{6}}{3}$
\thanks{ Supported by the National Natural Science
Foundation of China and the Seed Foundation of Innovation
and Creation for Graduate Students in Northwestern Polytechnical University (Z2017190).}}
\author{{Weige Xi, Ligong Wang\footnote{Corresponding author.}, Xiangxiang Liu, Xihe Li, Xiaoguo Tian}\\
{\small Department of Applied Mathematics, School of Natural and Applied Science,}\\
{\small Northwestern Polytechnical University, Xi'an, Shaanxi 710072, P.R.China}
 \\{\small E-mail: xiyanxwg@163.com, lgwangmath@163.com, xxliumath@163.com}
 \\ {\small lxhdhr@163.com, xiaoguotianwycm@163.com }\\}
\date{}
\maketitle
\begin{center}
\begin{minipage}{120mm}
\vskip 0.3cm
\begin{center}
{\small {\bf Abstract}}
\end{center}
{\small Let $\lambda_{2}(G)$ be the second smallest normalized Laplacian eigenvalue of a graph $G$. In this paper, we determine all unicyclic graphs of order $n\geq21$ with $\lambda_{2}(G)\geq 1-\frac{\sqrt{6}}{3}$. Moreover, the unicyclic graphs with
$\lambda_{2}(G)=1-\frac{\sqrt{6}}{3}$ are also determined.

}
\end{minipage}
\end{center}

\section{Introduction}
\label{sec:ch6-introduction}
All graphs considered in this paper are connected, undirected and simple. Let $G$ be  a graph with vertex set $V(G)$ and edge set $E(G)$.
Let $d(v)=d_{G}(v)$ be the degree of vertex $v$ in $G$. For $v\in V(G)$ ($e\in E(G)$, resp.), we use $G-v$ ($G-e$, resp.) to denote the graph obtained by deleting $v$ ($e$, resp.) from $G$. Let $I$ denote the identity matrix, ${\bf j}$ denote the vector consisting of all ones.
 We use the notation $C_n$ for the cycle of order $n$. Meanwhile we denote by $\Phi(B)=\Phi(B;\lambda)=det(\lambda I-B)$ the characteristic polynomial of the square matrix $B$.

Let $A(G)$ and $D(G)$ be the adjacency matrix and the diagonal matrix of vertex degrees of $G$, respectively. The Laplacian and normalized Laplacian matrices of $G$ are defined as $L(G)=D(G)-A(G)$ and $\mathcal{L}(G)=D^{-\frac{1}{2}}(G)L(G)D^{-\frac{1}{2}}(G)$, respectively. For $v\in V(G)$,
let $\mathcal{L}_v(G)$ be the principal submatrix of $\mathcal{L}(G)$ obtained by deleting the row and the column corresponding to the vertex $v$. When only one graph $G$ is under consideration, we sometimes use $A$, $D$, $L$, $\mathcal{L}$ and $\mathcal{L}_v$ instead of $A(G)$, $D(G)$, $L(G)$, $\mathcal{L}(G)$ and $\mathcal{L}_v(G)$ respectively. It is easy to see that $\mathcal{L}(G)$ is a symmetric positive semidefinite matrix and $D^{\frac{1}{2}}(G){\bf j}$ is an eigenvector of $\mathcal{L}(G)$ corresponding to eigenvalue 0. Thus, the eigenvalues $\lambda_{i}(G)$ of $\mathcal{L}(G)$ satisfy
$$\lambda_{n}(G)\geq\cdots\geq\lambda_{2}(G)\geq\lambda_{1}(G)=0.$$
Some of them may be repeated according to their multiplicities. $\lambda_{k}(G)$ is the $k$-th smallest normalized Laplacian eigenvalue of $G$. When only one graph is under consideration, we may use $\lambda_{k}$ instead of $\lambda_{k}(G)$, for $1\leq k \leq n$.

The normalized Laplacian matrix of a graph is pointed out by Chung \cite{Chung}. Chung showed that the normalized
Laplacian matrix has eigenvalues always lying in the range between 0 and 2 inclusive. One advantage to this is that it makes it easier to compare the distribution of
the eigenvalues for two different graphs. The eigenvalues of the normalized Laplacian matrix of a graph have a good relationship with other graph invariants for general graphs in a way that other definitions (such as the eigenvalues of adjacency matrix) fail to do. From these eigenvalues the so-called ¡°Randi\'{c} energy¡± can be calculated, which are believed to have some applications in chemistry. With regard to the results
of Randi\'{c} energ can be found in \cite{LG, LW}. The advantages of this definition on the normalized Laplacian eigenvalues
depend on the fact that it is the same with the eigenvalues in spectral geometry and in stochastic processes. We refer the
reader to \cite{Chung} for the detail.

In terms of $\lambda_{2}(G)$, Chung \cite{Chung} showed that $\lambda_{2}(G)$ is $0$ if and only if $G$ is disconnected. This result is closely related to the second smallest eigenvalue of its Laplacian matrix \cite{Butler}. Thus $\lambda_{2}(G)$ is popularly known as a good parameter to
measure how well a graph is connected. In fact, Chung \cite{Chung} also showed that $\lambda_{2}(G)$ is closely related to the discrete Cheeger¡¯s
constant, isoperimetric problems, etc.  H.H. Li et al. \cite{HHYZ} studied the behavior of $\lambda_{2}$ when the graph is perturbed by grafting
an edge. They determined the path with minimum $\lambda_{2}(G)$ among all trees of order $n$. H.H. Li et al. \cite{HH} studied the effect on the second smallest normalized Laplacian eigenvalue by grafting some pendant paths.

Restrictions on the second smallest normalized Laplacian eigenvalue of graphs force these graphs
to have a very special structure. Determining these graphs is an interesting problem. Many results about this problem
have been obtained. J.X. Li et al. \cite{Jian} studied the variation of $\lambda_{2}(G)$ when the graph is perturbed by separating an edge. They determined all trees and unicyclic graphs with $\lambda_{2}(G)\geq1-\frac{\sqrt{2}}{2}$. J.X. Li et al. \cite{Jianxi}
determine all trees with $\lambda_{2}(G)\geq1-\frac{\sqrt{6}}{3}$. They classified such trees into six classes, and the values of the second smallest normalized Laplacian eigenvalue for the six classes of the trees are provided, respectively. X.G. Tian et al \cite{Tian} determined all trees with $\lambda_{2}(G)\geq1-\frac{\sqrt{3}}{2}$.

In this paper, we further characterize all unicyclic graphs of order $n\geq21$ with $\lambda_{2}(G)\geq 1-\frac{\sqrt{6}}{3}$.
Moreover, the unicyclic graphs with
$\lambda_{2}(G)=1-\frac{\sqrt{6}}{3}$ are also determined. In Section 2, we give some known lemmas and preliminary
results. In Section $8-i$ ($3 \leq i \leq5$), we determine all trees of diameter $8-i$ with $\lambda_{2}(G)\geq 1-\frac{\sqrt{6}}{3}$.

\section{Preliminaries}

In this section, we recall some properties of the eigenvalues and eigenfunctions of the normalized Laplacian matrix $\mathcal{L}(G)$ of a graph $G$. Let $g$ be a eigenvector of $\mathcal{L}(G)$. Then we can view $g$ as a function which assigns to each vertex $v$ of $G$ a real value $g(v)$, the coordinate of $g$ according to $v$ (All the vectors in this paper are dealt in this way). By letting $g=D^{1/2}f$, we have $$\frac{g^{T}\mathcal{L}g}{g^{T}g}=\frac{f^{T}D^{1/2}\mathcal{L}D^{1/2}f}{(D^{1/2}f)^{T}D^{1/2}f}=\frac{f^{T}Lf}{f^{T}Df}=\frac{\sum_{uv\in E(G)}(f(u)-f(v))^{2}}{\sum_{v\in V(G)}d(v)(f(v))^{2}}.$$
Thus, we can obtain the following formulas for $\lambda_{2}(G)$.
\begin{equation}\label{eq:c1}\lambda_{2}(G)=\inf_{f\perp D{\bf j}}\frac{f^{T}Lf}{f^{T}Df}=\inf_{f\bot D{\bf j}}\frac{\sum_{uv\in E(G)}(f(u)-f(v))^{2}}{\sum_{v\in V(G)}d(v)(f(v))^{2}}.
\end{equation}
A nonzero vector that satisfies equality in (\ref{eq:c1}) is called a
harmonic eigenfunction associated with $\lambda_{2}(G)$.

The following inequalities are known as the Cauchy's inequalities and the whole theorem is also known as the interlacing theorem.

\noindent\begin{lemma}\label{le:1}  (\cite{Cvetk}) Let $A$ be a Hermitian matrix with eigenvalues $x_{1}\leq\cdots\leq x_{n}$ and $B$ be a principal submatrix of $A$. Let $B$ have eigenvalues $y_{1}\leq\cdots\leq y_{m}$ $(m\leq n)$. Then the inequalities
$x_{i}\leq y_{i}\leq x_{n-m+i}$ hold for $i=1,2,\cdots,m$. In particular, $x_{1}\leq y_{1}\leq x_{2}\leq y_{2}\leq\cdots\leq y_{n-1}\leq x_{n}$ for $m=n-1$.
\end{lemma}

\noindent\begin{lemma}\label{le:2}  (\cite{Jian}) Let $G$ be a connected graph and $v$ be a pendent vertex of $G$. Then $\lambda_{2}(G)\leq\lambda_{2}(G-v)$.
\end{lemma}

Let $e=uv$ be an edge of a graph $G$. Let $G'$ be the graph obtained from $G$ by contracting the edge $e$ into a new vertex $u_{e}$ and adding a new pendent edge $u_{e}v_{e}$, where $v_{e}$ is a new pendent vertex. We call that $G'$ is obtained from $G$ by separating an edge $uv$ (see Figure \ref{Fig.1.}).

\begin{figure}[htbp]
  \centering
  \includegraphics[scale=0.6]{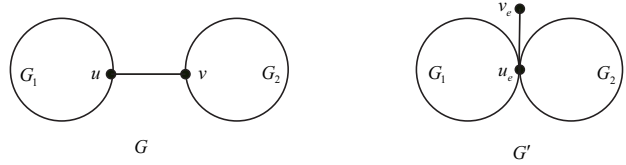}
  \caption{Separating an edge $uv$}\label{Fig.1.}
\end{figure}

\noindent\begin{lemma}\label{le:3} (\cite{Jian}) Let $e=uv$ be a cut edge of a connected graph $G$. Suppose that $G-uv=G_{1} \cup G_{2}$ $(|V(G_{1})|, |V(G_{2})|\geq2)$, where $G_{1}$ and $G_{2}$ are two components of $G-uv$, $u\in V(G_{1})$ and $v\in V(G_{2})$. Let $G'$ be the graph obtained from $G$ by separating the edge $uv$. Then $\lambda_{2}(G)\leq \lambda_{2}(G')$, and the inequality is strict if $f(v_{e})\neq0$, where $f$ is a harmonic eigenfunction associated with $\lambda_{2}(G')$.
\end{lemma}

\noindent\begin{lemma}\label{le:4} (\cite{Butler}) For $n \geq 3$, $\lambda_{2}(C_n) =1-cos(\frac{2\pi}{n})$.
\end{lemma}

Let $\mathcal{U}_n$ be the set of unicyclic graphs of order $n$, and $\mathcal{U}_n^g$ be the
set of unicyclic graphs of order $n$ with girth $g$ ($3 \leq g \leq n$). Let $C^1_{n-1}$ be the unicyclic graph of order
$n$ obtained by attaching a pendent edge to the cycle $C_{n-1}$. Clearly, if $U\in \mathcal{U}_n^n$, then $U\cong C_n$;
if $U\in \mathcal{U}_n^{n-1}$, then $U\cong C^1_{n-1}$.

For each $U\in\mathcal{U}_n^g$, $U$ consists of the (unique) cycle (say $C_g$ ) of length $g$ and a certain number of trees attached at
vertices of $C_g$ having (in total) $n-g$ edges. We assume that the vertices of $C_g$ are $v_1, v_2,\cdots, v_g$ (ordered in a natural
way around $C_g$, say in the clockwise direction). Then $U$ can be written as $C_g (T_1,\cdots, T_g )$, which is obtained from the cycle
$C_g$ on vertices $v_1,v_2,\cdots,v_g$ by identifying $v_i$ with the root of a rooted tree $T_i$ of order $n_i$ for each $i=1,\cdots,g,$ where
$n_i\geq1$ and $\sum\limits_{i=1}^g n_i =n$. If $T_i$, for each $i$, is a star of order $n_i$, whose root is a vertex of maximum degree,
then we write $U=S_g (n_1,\cdots,n_g)$.

\noindent\begin{lemma}\label{le:5}  (\cite{Jian}) Let $U$ be $C_g (T_1,\cdots,T_g)$, where $|V(T_i)| = n_i\geq1$ for $i = 1,\cdots,g$ and $\sum\limits_{i=1}^g n_i = n$. Then $\lambda_{2}(U) \leq \lambda_{2}(S_g (n_1,\cdots,n_g))$.
\end{lemma}

\noindent\begin{lemma}\label{le:6} Suppose that $U\in\mathcal{U}_n$, with girth $g$ and $n\geq21$. If $\lambda_{2}(U)\geq1-\frac{\sqrt{6}}{3}$, then $g\leq5$.
\end{lemma}

\begin{proof} We shall prove the contrapositive of the theorem. Let $C_g$ be the unique cycle in $U$, where $g\geq6$. We consider the following six cases.
\\{ \bf Case 1.} $g\geq11$.

Then Lemmas \ref{le:5}, \ref{le:2} and \ref{le:4} imply that $\lambda_{2}(U)\leq\lambda_{2}(C_{11})=1-cos(\frac{2\pi}{11})=0.15875<1-\frac{\sqrt{6}}{3}$.
\\{ \bf Case 2.} $g=10$.

Then Lemmas \ref{le:5} and \ref{le:2} imply that $\lambda_{2}(U)\leq\lambda_{2}(C_{10}^1)\doteq 0.15633<1-\frac{\sqrt{6}}{3}$.
\\{ \bf Case 3.} $g=9$.

We rewrite $U$ in the form $C_9 (T_1,T_2\cdots, T_9 )$, where $|V(T_i)| = n_i\geq1$ for $i = 1,2,\cdots,9$ and $\sum\limits_{i=1}^9 n_i = n$. Then we can know that there exists some $j\in\{1,2,\cdots,9\}$ such that $n_j\geq3$ since $n\geq21$. Then Lemmas \ref{le:5} and \ref{le:2} imply that $\lambda_{2}(U)\leq\lambda_{2}(S_9 (3,1,1,1,1,1,1,1,1))\doteq0.15875<1-\frac{\sqrt{6}}{3}$.
\\{ \bf Case 4.} $g=8$.

We rewrite $U$ in the form $C_8 (T_1,T_2,\cdots, T_8 )$, where $|V(T_i)| = n_i\geq1$ for $i = 1,2,\cdots,8$ and $\sum\limits_{i=1}^8 n_i = n$. Let $N=\{i\mid n_i\geq2, i=1,2,\cdots,8\}$.
\begin{figure}[htbp]
  \centering
  \includegraphics[scale=0.6]{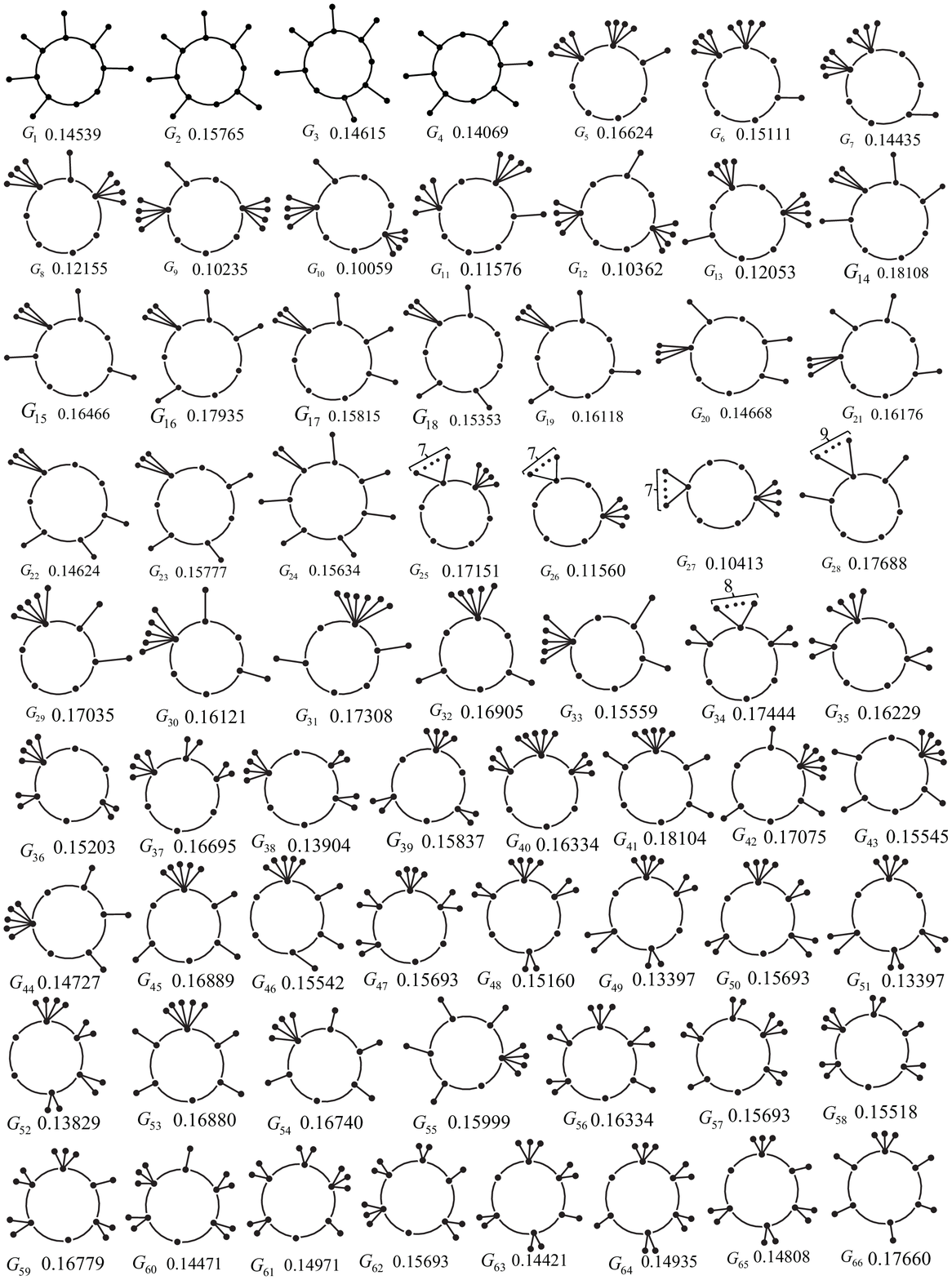}
  \caption{Unicyclic graphs $G_{i}$ and $\lambda_{2}(G_{i})$, $1\leq i\leq66$.}\label{Fig.2.}
\end{figure}
\\{ \bf Subcase 4.1.} If $|N|\geq6$, Hence Lemmas \ref{le:5} and \ref{le:2} imply that
$\lambda_{2}(U)\leq\max\{\lambda_{2}(G_1),$ $\lambda_{2}(G_2),$ $\lambda_{2}(G_3),\lambda_{2}(G_4)\}\doteq0.15765<1-\frac{\sqrt{6}}{3}$,
where $G_{i}$ ($1\leq i\leq4$) are shown in Figure \ref{Fig.2.}.
\\{ \bf Subcase 4.2.} If $|N|\leq5$, then there exists some $j\in\{1,2,\cdots,8\}$ such that $n_j\geq4$ since $n\geq21$. Hence Lemmas \ref{le:5} and \ref{le:2} imply that
$\lambda_{2}(U)\leq\lambda_{2}(S_8 (4,1,1,1,1,1,1,1))\doteq0.16841<1-\frac{\sqrt{6}}{3}$.
\\{ \bf Case 5.} $g=7$.

We rewrite $U$ in the form $C_7 (T_1,T_2\cdots, T_7 )$, where $|V(T_i)| = n_i\geq1$ for $i = 1,2,\cdots,7$ and $\sum\limits_{i=1}^7 n_i = n$. Let
$N=\{i\mid n_i\geq2, i=1,2,\cdots,7\}$.
\\{ \bf Subcase 5.1.} If $|N|\leq3$, then there exists some $j\in\{1,2,\cdots,7\}$ such that $n_j\geq6 $ since $n\geq21$. Hence Lemmas \ref{le:5} and \ref{le:2} imply that
$\lambda_{2}(U)\leq\lambda_{2}(S_7(6,1,1,1,1,1,1))$ $\doteq0.17465<1-\frac{\sqrt{6}}{3}$.
\\{ \bf Subcase 5.2.} If $|N|=4$, then there exists some $j\in\{1,2,\cdots,7\}$ such that $n_j\geq5$ since $n\geq21$.
\\{ \bf Subcase 5.2.1.} If $n_j\geq6$, then Lemmas \ref{le:5} and \ref{le:2} imply that
$\lambda_{2}(U)\leq\lambda_{2}(S_7(6,1,1,1,1,1,1))$ $\doteq0.17465<1-\frac{\sqrt{6}}{3}$.
\\{\bf Subcase 5.2.2.} If $n_j=5$, then there exists some $k\in\{1,2,\cdots,7\}$ and $k\neq j$
such that $n_k\geq5$ since $n\geq21$. Hence Lemmas \ref{le:5} and \ref{le:2} imply that
$\lambda_{2}(U)\leq\max\{\lambda_{2}(G_i)\mid 5\leq i\leq13\}\doteq0.16624<1-\frac{\sqrt{6}}{3}$, where $G_{i}$ ($5\leq i\leq13$) are shown in Figure \ref{Fig.2.}.
\\{ \bf Subcase 5.3.} If $|N|=5$ or 6, then there exists some $j\in\{1,2,\cdots,7\}$ such that $n_j\geq4$ since $n\geq21$. Hence Lemmas \ref{le:5} and \ref{le:2} imply that
$\lambda_{2}(U)\leq\max\{\lambda_{2}(G_i)\mid 14\leq i\leq23\}\doteq0.18108<1-\frac{\sqrt{6}}{3}$, where $G_{i}$ ($14\leq i\leq23$) are shown in Figure \ref{Fig.2.}.
\\{ \bf Subcase 5.4.} If $|N|=7$, then there exists some $j\in\{1,2,\cdots,7\}$ such that $n_j\geq3$ since $n\geq21$. Hence Lemmas \ref{le:5} and \ref{le:2} imply that
$\lambda_{2}(U)\leq\lambda_{2}(G_{24})\doteq0.15634<1-\frac{\sqrt{6}}{3}$, where $G_{24}$ is shown in Figure \ref{Fig.2.}.
\\{\bf Case 6.} $g=6$.

We rewrite $U$ in the form $C_6(T_1,T_2,\cdots, T_6 )$, where $|V(T_i)| = n_i\geq1$ for $i = 1,2,\cdots,6$ and $\sum\limits_{i=1}^6 n_i = n$. Let $N=\{i\mid n_i\geq2, i=1,2,\cdots,6\}$.
\\{ \bf Subcase 6.1.} If $|N|=1$, then there exists some $j\in\{1,2,\cdots,6\}$ such that $n_j\geq16 $ since $n\geq21$. Hence Lemmas \ref{le:5} and \ref{le:2} imply that
$\lambda_{2}(U)\leq\lambda_{2}(S_6(16,1,1,1,1,1))\leq\lambda_{2}(S_6(12,1,1,1,1,1))$ $\doteq0.17959<1-\frac{\sqrt{6}}{3}$.
\\{ \bf Subcase 6.2.} If $|N|=2$, let $N=\{i,k\}$. Without lose of generality, we assume $n_i\leq n_k$.
\\{ \bf Subcase 6.2.1.} If $n_i\leq4$, then $n_k\geq13$ since $n\geq21$.
Hence Lemmas \ref{le:5} and \ref{le:2} imply that
$\lambda_{2}(U)\leq\lambda_{2}(S_6(13,1,1,1,1,1))\leq\lambda_{2}(S_6(12,1,1,1,1,1))\doteq0.17959<1-\frac{\sqrt{6}}{3}$.
\\{ \bf Subcase 6.2.2.} If $5\leq n_i\leq n_k$, then $n_k\geq8$ since $n\geq21$. Hence Lemmas \ref{le:5} and \ref{le:2} imply that
$\lambda_{2}(U)\leq\max\{\lambda_{2}(G_i)\mid 25\leq i\leq27\}\doteq0.17151<1-\frac{\sqrt{6}}{3}$, where $G_{i}$ ($25\leq i\leq27$)
are shown in Figure \ref{Fig.2.}.
\\{ \bf Subcase 6.3.} If $|N|=3$, let $N=\{i,j,k\}$. Without lose of generality, we assume $n_i\leq n_j\leq n_k$.
\\{ \bf Subcase 6.3.1.} If $n_i=2$, $2\leq n_j\leq6$, then $n_k\geq10$ since $n\geq21$.  Hence Lemmas \ref{le:5} and \ref{le:2} imply that
$\lambda_{2}(U)\leq\max\{\lambda_{2}(G_i)\mid 28\leq i\leq33\}\doteq0.17688<1-\frac{\sqrt{6}}{3}$, where $G_{i}$ ($28\leq i\leq33$) are shown in Figure \ref{Fig.2.}.
\\{ \bf Subcase 6.3.2.} If $n_i=2$, $7\leq n_j\leq n_k$.  Hence Lemmas \ref{le:5} and \ref{le:2} imply that
$\lambda_{2}(U)\leq\max\{\lambda_{2}(G_i)\mid 29\leq i\leq33\}\doteq0.17308<1-\frac{\sqrt{6}}{3}$, where $G_{i}$  ($29\leq i\leq33$) are shown in Figure \ref{Fig.2.}.
\\{ \bf Subcase 6.3.3.} If $n_i=3$,
$3\leq n_j\leq6$, then $n_k\geq9$.  Hence Lemmas \ref{le:5} and \ref{le:2} imply that
$\lambda_{2}(U)\leq\max\{\lambda_{2}(G_i)\mid 34\leq i\leq39\}\doteq0.17444<1-\frac{\sqrt{6}}{3}$,
where $G_{i}$ ($34\leq i\leq39$) are shown in Figure \ref{Fig.2.}.
\\{ \bf Subcase 6.3.4.} If $n_i=3$, $7\leq n_j\leq n_k$. Hence Lemmas \ref{le:5} and \ref{le:2} imply that
$\lambda_{2}(U)\leq\max\{\lambda_{2}(G_i)\mid 35\leq i\leq39\}\doteq0.16695<1-\frac{\sqrt{6}}{3}$,
where $G_{i}$ ($35\leq i\leq39$) are shown in Figure \ref{Fig.2.}.
\\{ \bf Subcase 6.3.5.} If $4\leq n_i\leq n_j\leq n_k$, then $n_k\geq6$ since $n\geq21$.  Hence Lemmas \ref{le:5} and \ref{le:2} imply that
$\lambda_{2}(U)\leq\max\{\lambda_{2}(G_i)\mid 35\leq i\leq40\}\doteq0.16695<1-\frac{\sqrt{6}}{3}$,
where $G_{i}$ ($35\leq i\leq40$) are shown in Figure \ref{Fig.2.}.
\\{ \bf Subcase 6.4.} If $|N|=4$, then we discuss the following two cases.
\\{ \bf Subcase 6.4.1.} If there is at least one $i\in N$ such that $n_{i}=2$, then there exists some $j\in\{1,2,\cdots,6\}$ such that $n_j\geq6$ since $n\geq21$. Hence Lemmas \ref{le:5} and \ref{le:2} imply that
$\lambda_{2}(U)\leq\max\{\lambda_{2}(G_i)\mid 41\leq i\leq46\}\doteq0.18104<1-\frac{\sqrt{6}}{3}$, where $G_{i}$ ($41\leq i\leq46$) are shown in Figure \ref{Fig.2.}.
\\{ \bf Subcase 6.4.2.} If $n_i\geq3$ for each $i\in N$, then there exists some $j\in\{1,2,\cdots,6\}$ such that $n_j\geq5$ since $n\geq21$. Hence Lemmas \ref{le:5} and \ref{le:2} imply that
$\lambda_{2}(U)\leq\max\{\lambda_{2}(G_i)\mid 47\leq i\leq52\}\doteq0.15693<1-\frac{\sqrt{6}}{3}$, where $G_{i}$ ($47\leq i\leq52$) are shown in Figure \ref{Fig.2.}.
\\{ \bf Subcase 6.5} If $|N|=5$, then we discuss the following two cases.
\\{ \bf Subcase 6.5.1.} If there are at least two $i, j$ in $N$ such that $n_{i}=n_{j}=2$, then there exists some $k\in\{1,2,\cdots,6\}$ such that $n_k\geq6$ since $n\geq21$. Hence Lemmas \ref{le:5} and \ref{le:2} imply that
$\lambda_{2}(U)\leq\max\{\lambda_{2}(G_i)\mid 53\leq i\leq55\}\doteq0.16880<1-\frac{\sqrt{6}}{3}$, where $G_{i}$ ($53\leq i\leq55$) are shown in Figure \ref{Fig.2.}.
\\{ \bf Subcase 6.5.2.} If there is at most one $i$ in $N$ such that $n_{i}=2$, then there exists some $j\in\{1,2,\cdots,6\}$ such that $n_j\geq4$ since $n\geq21$. Hence Lemmas \ref{le:5} and \ref{le:2} imply that
$\lambda_{2}(U)\leq\max\{\lambda_{2}(G_i)\mid 56\leq i\leq65\}\doteq0.16779<1-\frac{\sqrt{6}}{3}$, where $G_{i}$ ($56\leq i\leq65$) are shown in Figure \ref{Fig.2.}.
\\{ \bf Subcase 6.6.} If $|N|=6$, then there exists some $k\in\{1,2,\cdots,7\}$ such that $n_k\geq4$ since $n\geq21$. Hence Lemmas \ref{le:5} and \ref{le:2} imply that
$\lambda_{2}(U)\leq\lambda_{2}(G_{66})\doteq0.17660<1-\frac{\sqrt{6}}{3}$, where $G_{66}$ is shown in Figure \ref{Fig.2.}.
\\From the discussion above, the proof is completed.
\end{proof}
\section{All unicyclic graphs of girth $5$ with $\lambda_{2}\geq1-\frac{\sqrt{6}}{3}$}

In this section, we determine all unicyclic graphs in $\mathcal{U}_n^5$ ($n\geq21$) with $\lambda_{2}\geq1-\frac{\sqrt{6}}{3}$.
Before introducing our results we recall some notation. Note that if $U\in\mathcal{U}_n^5$, then $U$ consists
of the cycle $C_5=v_1v_2v_3v_4v_5v_1$ and five trees $T_1$, $T_2$, $T_3$, $T_4$ and $T_5$ attached at the vertices $v_1$, $v_2$,
 $v_3$, $v_4$ and $v_5$, respectively.
\noindent\begin{theorem}\label{th:1} Suppose that $U\in\mathcal{U}_n^5$ with $n\geq21$.
Then $\lambda_{2}(U)\geq1-\frac{\sqrt{6}}{3}$ if and only if one of the following items holds:

(i) $U$ is isomorphic to $S_5(2,k,2,1,1)$ with $15\leq k\leq18$.

(ii) $U$ is isomorphic to $H_4$, or $H_5$, where $H_4$ and $H_5$ are shown in Figure \ref{Fig.13.}.

(iii) $U$ is unicyclic graph $G_0$, where $G_0$ is shown in Figure \ref{Fig.16.}, $l_0$ and $l_1$ are nonnegative integers,
$15\leq l_0+2l_1\leq29$, $0\leq l_1\leq7$ and one of the following items holds:

(1) $l_1=7$ and $l_0=1$.

(2) $l_1=6$ and $3\leq l_0\leq5$.

(3) $l_1=5$ and $5\leq l_0\leq9$.

(4) $l_1=4$ and $7\leq l_0\leq13$.

(5) $l_1=3$ and $9\leq l_0\leq17$.

(6) $l_1=2$ and $11\leq l_0\leq21$.

(7) $l_1=1$ and $13\leq l_0\leq25$.

(8) $l_1=0$ and $15\leq l_0\leq29$.

(iv) $U$ is isomorphic to $H_{42}$, where $H_{42}$ as shown in Figure \ref{Fig.16.}, and $l_i$ ($0\leq i\leq 2$) are nonnegative integers
and $l_0+2l_1+3l_2+5=n\geq21$.

Furthermore, the equality holds if and only if $U$ is isomorphic to $H_{42}$ with $l_2\geq2$.
\end{theorem}
\begin{proof} We rewrite $U$ in the form $C_5 (T_1,T_2,\cdots, T_5)$, where $|V(T_i)|=n_i\geq1$ for $i = 1,2,\cdots,5$ and $\sum\limits_{i=1}^5 n_i=n$. Let $N=\{i\mid n_i\geq2, i=1,2,\cdots,5\}$. In order to determine all unicyclic graphs in $\mathcal{U}_n^5$ ($n\geq21$) with $\lambda_{2}\geq1-\frac{\sqrt{6}}{3}$, the following cases are considered.
\\{ \bf Case 1.} $|N|=5$. We assume that $n_1',n_2',n_3',n_4',n_5'\in\{n_1,n_2,n_3,n_4,n_5\}$ and $2\leq n_1'\leq n_2'\leq n_3'\leq n_4'\leq n_5'$.
\\{ \bf Subcase 1.1.} If there are exactly two $n_1'$ and $n_2'$ such that $n_1'=n_2'=2$, and at least one $n_3'$ such that $n_3'=3$, then $n_5'\geq7 $ since $n\geq21$. Hence Lemmas \ref{le:5} and \ref{le:2} imply that $\lambda_{2}(U)\leq\max\{\lambda_{2}(G_i)\mid 73\leq i\leq76\}\doteq0.17531<1-\frac{\sqrt{6}}{3}$,
where $G_{i}$ ($73\leq i\leq76$) are shown in Figure \ref{Fig.11.}.
\\{ \bf Subcase 1.2.} If there are exactly two $n_1'$ and $n_2'$ such that $n_1'=n_2'=2$, and the others are at least 4, then $n_5'\geq6$ since $n\geq21$. Hence Lemmas \ref{le:5} and \ref{le:2} imply that $\lambda_{2}(U)\leq\max\{\lambda_{2}(G_i)\mid 74\leq i\leq77\}\doteq0.17953<1-\frac{\sqrt{6}}{3}$,
where $G_{i}$ ($74\leq i\leq77$) are shown in Figure \ref{Fig.11.}.
\\{ \bf Subcase 1.3.} If there is exactly one $n_1'$ such that $n_1'=2$, and the others are at least 3, then $n_5'\geq5$ since $n\geq21$. Hence Lemmas \ref{le:5} and \ref{le:2} imply that $\lambda_{2}(U)\leq\max\{\lambda_{2}(G_i)\mid 78\leq i\leq79\}\doteq0.18114<1-\frac{\sqrt{6}}{3}$,
where $G_{i}$ ($78\leq i\leq79$) are shown in Figure \ref{Fig.11.}.
\\{ \bf Subcase 1.4.} If there are at least three $n_1'$, $n_2'$ and $n_3'$ such that $n_1'=n_2'=n_3'=2$, then $n_5'\geq8$ since $n\geq21$. Hence Lemmas \ref{le:5} and \ref{le:2} imply that $\lambda_{2}(U)\leq\lambda_{2}(G_{80})\doteq0.16854<1-\frac{\sqrt{6}}{3}$,
where $G_{80}$ is shown in Figure \ref{Fig.11.}.
\\{ \bf Subcase 1.5.} If $n_i\geq3$ for each $i\in N$, and there is at least one $n_1'$ such that $n_1'=3$, then $n_5'\geq5$ since $n\geq21$. Hence Lemmas \ref{le:5} and \ref{le:2} imply that $\lambda_{2}(U)\leq\lambda_{2}(G_{81})\doteq0.16979<1-\frac{\sqrt{6}}{3}$, where $G_{81}$ is shown in Figure \ref{Fig.11.}.
\\{ \bf Subcase 1.6.} If $n_i\geq4$ for each $i\in N$, Hence Lemmas \ref{le:5} and \ref{le:2} imply that $\lambda_{2}(U)\leq\lambda_{2}(G_{82})\doteq0.16114<1-\frac{\sqrt{6}}{3}$, where $G_{82}$ is shown in Figure \ref{Fig.11.}.
\\{ \bf Case 2.} $|N|=4$. We assume that $n_1',n_2',n_3',n_4'\in\{n_1,n_2,n_3,n_4,n_5\}$ and $2\leq n_1'\leq n_2'\leq n_3'\leq n_4'$.
\\ { \bf Subcase 2.1.} If there is exactly one $n_1'$ such that $n_1'=2$, and at least one $n_2'$ such that $n_2'=3$, then $n_4'\geq8$ since $n\geq21$. Hence Lemmas \ref{le:5} and \ref{le:2} imply that $\lambda_{2}(U)\leq\max\{\lambda_{2}(G_i)\mid 83\leq i\leq88\}\doteq0.18204<1-\frac{\sqrt{6}}{3}$,
where $G_{i}$ ($83\leq i\leq88$) are shown in Figure \ref{Fig.11.}.
\begin{figure}[htbp]
  \centering
  \includegraphics[scale=0.6]{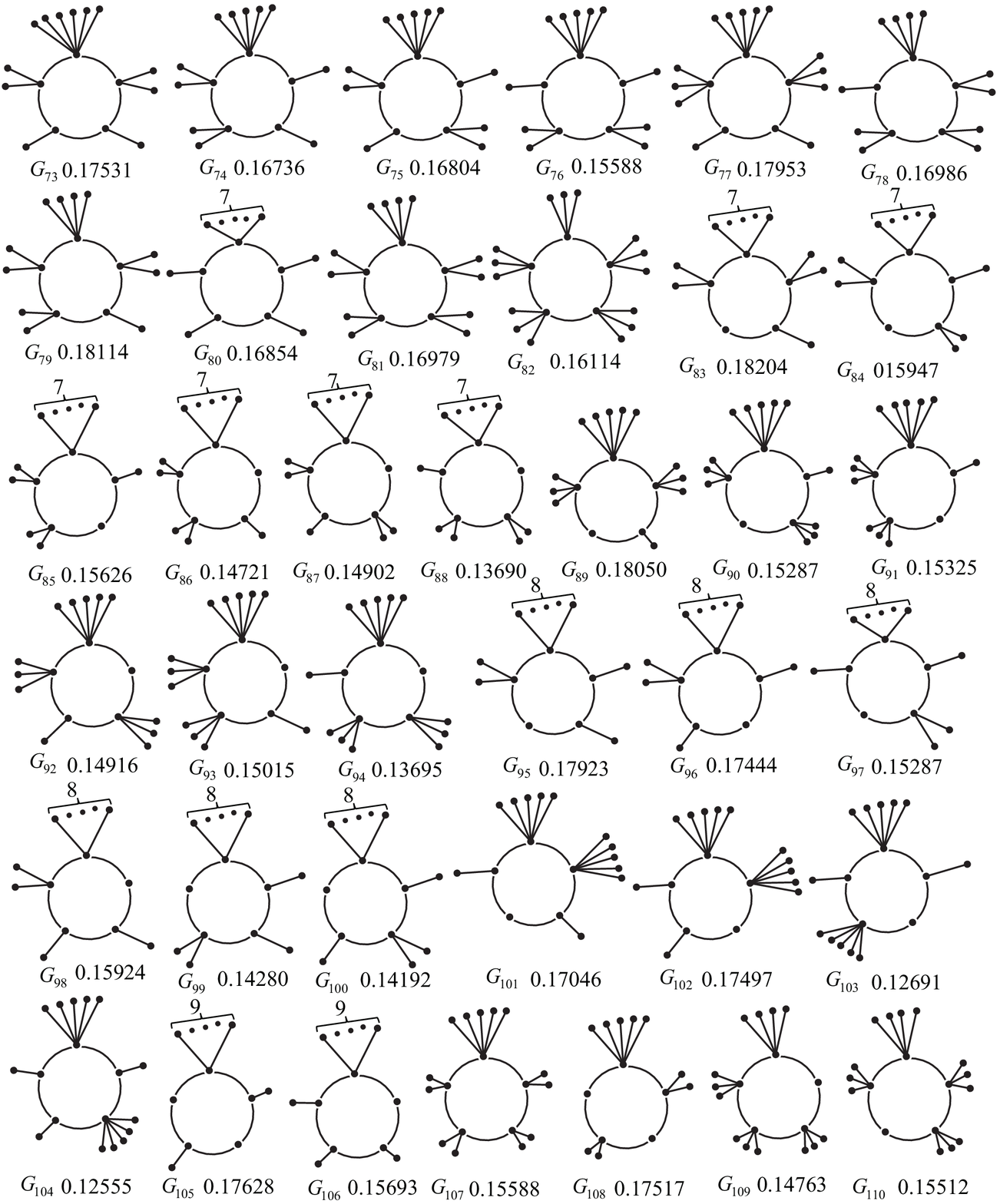}
  \caption{Unicyclic graphs $G_{i}$ and $\lambda_{2}(G_{i})$, $73\leq i\leq110$.}\label{Fig.11.}
\end{figure}
\\{ \bf Subcase 2.2.} If there is exactly one $n_1'$ such that $n_1'=2$, and $n_2'\geq4$, then $n_4'\geq6$ since $n\geq21$. Hence Lemmas \ref{le:5} and \ref{le:2} imply that $\lambda_{2}(U)\leq\max\{\lambda_{2}(G_i)\mid 89\leq i\leq94\}\doteq0.18050<1-\frac{\sqrt{6}}{3}$,
where $G_{i}$ ($89\leq i\leq94$) are shown in Figure \ref{Fig.11.}.
\\{ \bf Subcase 2.3.} If there are exactly two $n_1'$ and $n_2'$ in $N$ such that $n_1'=n_2'=2$, and $3\leq n_3'\leq7$, then $n_4'\geq9$ since $n\geq21$. Hence Lemmas \ref{le:5} and \ref{le:2} imply that $\lambda_{2}(U)\leq\max\{\lambda_{2}(G_i)\mid 95\leq i\leq100\}\doteq0.17923<1-\frac{\sqrt{6}}{3}$,
where $G_{i}$ ($95\leq i\leq100$) are shown in Figure \ref{Fig.11.}.
\\{ \bf Subcase 2.4.} If there are exactly two $n_1'$ and $n_2'$ in $N$ such that $n_1'=n_2'=2$, and $8\leq n_3'\leq n_4$. Hence Lemmas \ref{le:5} and \ref{le:2} imply that $\lambda_{2}(U)\leq\max\{\lambda_{2}(G_i)\mid 101\leq i\leq104\}\doteq0.17497<1-\frac{\sqrt{6}}{3}$, where $G_{i}$ ($101\leq i\leq104$) are shown in Figure \ref{Fig.11.}.
\\{ \bf Subcase 2.5.} If there are exactly three $n_1'$, $n_2'$ and $n_3$ such that $n_1'=n_2'=n_3'=2$, then $n_4'\geq14$ since $n\geq21$.
Hence Lemmas \ref{le:5} and \ref{le:2} imply that $\lambda_{2}(U)\leq\max\{\lambda_{2}(G_i)\mid 105\leq i\leq106\}\doteq0.17628<1-\frac{\sqrt{6}}{3}$, where $G_{i}$ ($105\leq i\leq106$) are shown in Figure \ref{Fig.11.}.
\\{ \bf Subcase 2.6.} If $3\leq n_1'\leq n_2'\leq n_3'\leq n_4'$, and there is at least one $n_1'=3$, then $n_4'\geq6$ since $n\geq21$. Hence Lemmas \ref{le:5} and \ref{le:2} imply that $\lambda_{2}(U)\leq\max\{\lambda_{2}(G_i)\mid 107\leq i\leq108\}\doteq0.17517<1-\frac{\sqrt{6}}{3}$, where $G_{i}$ ($107\leq i\leq108$) are shown in Figure \ref{Fig.11.}.
\\{ \bf Subcase 2.7.} If $4\leq n_1'\leq n_2'\leq n_3'\leq n_4'$, then $n_4'\geq5$ since $n\geq21$. Hence Lemmas \ref{le:5} and \ref{le:2} imply that $\lambda_{2}(U)\leq\max\{\lambda_{2}(G_i)\mid 109\leq i\leq110\}\doteq0.15512<1-\frac{\sqrt{6}}{3}$, where $G_{i}$ ($109\leq i\leq110$) are shown in Figure \ref{Fig.11.}.
\\{ \bf Case 3.} $|N|=3$. We assume that $n_1',n_2',n_3'\in\{n_1,n_2,n_3,n_4,n_5\}$ and $2\leq n_1'\leq n_2'\leq n_3'$.
\\{ \bf Subcase 3.1.} $C_5 (T_1,T_2,T_3,T_4,T_5)\cong C_5(T_2',T_3',T_1',S_1,S_1)$.
\\{ \bf Subcase 3.1.1.} If there are exactly two $n_1'$ and $n_2'$ such that $n_1'=n_2'=2$,
then $n_3'\geq15$ since $n\geq21$. Since $C_5 (T_1,T_2,T_3,T_4, T_5)\cong C_5 (S_2,T_3',S_2,S_1,S_1)$.
Let $d$ be the length of the longest path from $v_3'$ to the pendent vertex in $T_3'$. In order to make
$\lambda_{2}(C_5 (S_2,T_3',S_2,S_1,S_1))\geq1-\frac{\sqrt{6}}{3}$, then $d\leq2$. If $d\geq3$, then Lemma \ref{le:2} implies that
$\lambda_{2}(C_5 (S_2,T_3',S_2,S_1,S_1))\leq\lambda_{2}(H)\doteq0.14988<1-\frac{\sqrt{6}}{3}$, where $H$ is shown in Figure \ref{Fig.13.}.
Thus $C_5 (T_1,T_2,T_3,T_4, T_5)$ is the unicyclic graph $H_1$ as shown in Figure \ref{Fig.13.}, where $l_i$ ($0\leq i\leq t$) are nonnegative integers
and $l_0+2l_1+3l_2+\cdots+(t+1)l_t+7=n\geq21$.
If $l_3+l_4+\cdots+l_t\geq1$, then Lemma \ref{le:2} implies that
$\lambda_{2}(U)\leq\lambda_{2}(H_0)\doteq0.15520<1-\frac{\sqrt{6}}{3}$, where $H_0$ is shown in Figure \ref{Fig.13.}. Thus $l_3+l_4+\cdots+l_t=0$. If $l_2\geq1$, then Lemma \ref{le:2} implies that
$\lambda_{2}(U)\leq\lambda_{2}(H_2)\doteq0.16895<1-\frac{\sqrt{6}}{3}$, where $H_2$ is shown in Figure \ref{Fig.13.}. Thus $l_2=0$.
Since $\lambda_{2}(S_5(2,19,2,1,1))=0.18309<1-\frac{\sqrt{6}}{3}$, If $l_0+2l_1\geq18$, then Lemmas \ref{le:5} and \ref{le:2}
imply that $\lambda_{2}(U)\leq\lambda_{2}(S_5(2,19,2,1,1))=0.18309<1-\frac{\sqrt{6}}{3}$. Thus $14\leq l_0+2l_1\leq17$. If $l_1\geq2$,
then Lemmas \ref{le:3} and \ref{le:2} imply that
$\lambda_{2}(U)\leq\lambda_{2}(H_3)\doteq0.18313<1-\frac{\sqrt{6}}{3}$, where $H_3$ is shown in Figure \ref{Fig.13.}. Thus $l_1\leq1$.
\\{ \bf Subcase 3.1.1.1.} $l_1=1$, then $12\leq l_0\leq15$.
\\If $12\leq l_0\leq13$, then Lemma \ref{le:2} implies that $\lambda_{2}(U)\geq\lambda_{2}(H_5)\doteq0.18512>1-\frac{\sqrt{6}}{3}$, where $H_5$ is shown in Figure \ref{Fig.13.}. Otherwise, if $15\geq l_0\geq14$, then Lemma \ref{le:2} implies that $\lambda_{2}(U)\leq\lambda_{2}(H_6)\doteq0.18311<1-\frac{\sqrt{6}}{3}$, where $H_6$ is shown in Figure \ref{Fig.13.}.
\\{ \bf Subcase 3.1.1.2.} $l_1=0$, then $14\leq l_0\leq17$.
\\By Lemma \ref{le:2}, we have $\lambda_{2}(U)\geq\lambda_{2}(S_6(2,18,2,1,1))\doteq0.18523>1-\frac{\sqrt{6}}{3}$
\begin{figure}[htbp]
  \centering
  \includegraphics[scale=0.55]{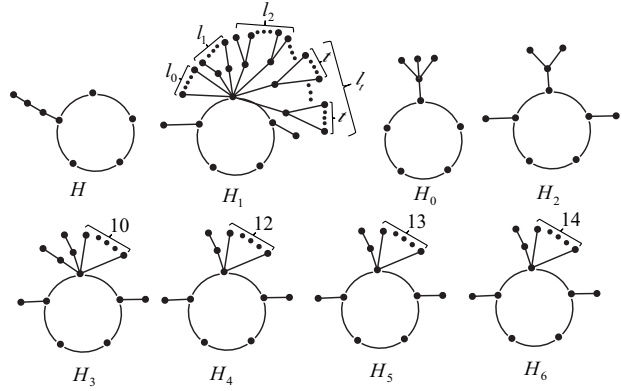}
  \caption{Unicyclic graphs $H$ and $H_i$, $0\leq i\leq6$.}\label{Fig.13.}
\end{figure}
\\{ \bf Subcase 3.1.2.} If there is exactly one $n_1'$ such that $n_1'=2$, and $3\leq n_2'\leq4$, then $n_3'\geq13$ since $n\geq21$.
Since $C_5 (T_1,T_2,T_3,T_4, T_5)\cong C_5 (T_2',T_3',S_2,S_1,S_1)$, then Lemmas \ref{le:5} and \ref{le:2} imply that
$\lambda_{2}(C_5(T_1,T_2,T_3,T_4, T_5))\leq\lambda_{2}(G_{111})\doteq0.18285<1-\frac{\sqrt{6}}{3}$,
where $G_{111}$ is shown in Figure \ref{Fig.12.}.
\\{ \bf Subcase 3.1.3.} If there is exactly one $n_1'$ such that $n_1'=2$,
and $n_3'\geq n_2'\geq5$, then $n_3'\geq8$ since $n\geq21$. Since $C_5 (T_1,T_2,T_3,T_4, T_5)\cong C_5 (T_2',T_3',S_2,S_1,S_1)$,
then Lemmas \ref{le:5} and \ref{le:2} imply that $\lambda_{2}(U)\leq\lambda_{2}(G_{112})\doteq0.17527<1-\frac{\sqrt{6}}{3}$,
where $G_{112}$ is shown in Figure \ref{Fig.12.}.
\\{ \bf Subcase 3.1.4.} If $n_1'=3$, $3\leq n_2'\leq4$, then $n_3'\geq12$ since $n\geq21$. Since
$C_5(T_1,T_2,T_3,T_4, T_5)$ $\cong C_5(T_2',T_3',T_1',S_1,S_1)$,
then Lemmas \ref{le:5} and \ref{le:2} imply that $\lambda_{2}(U)\leq\lambda_{2}(G_{113})\doteq0.18023<1-\frac{\sqrt{6}}{3}$,
where $G_{113}$ is shown in Figure \ref{Fig.12.}.
\\{ \bf Subcase 3.1.5.} If $n_1'=3$, $n_3\geq n_2'\geq5$, then $n_3'\geq8$ since $n\geq21$. Since $C_5(T_1,T_2,T_3,T_4, T_5)$ $\cong C_5(T_2',T_3',T_1',S_1,S_1)$,
then Lemmas \ref{le:5} and \ref{le:2} imply that $\lambda_{2}(U)\leq\lambda_{2}(G_{114})\doteq0.17728<1-\frac{\sqrt{6}}{3}$,
where $G_{114}$ is shown in Figure \ref{Fig.12.}.
\\{ \bf Subcase 3.1.6.} If $n_1'=4$, $4\leq n_2'\leq5$, then $n_3'\geq10$ since $n\geq21$. Since $C_5(T_1,T_2,T_3,T_4, T_5)$ $\cong C_5(T_2',T_3',T_1',S_1,S_1)$,
then Lemmas \ref{le:5} and \ref{le:2} imply that $\lambda_{2}(U)\leq\lambda_{2}(G_{115})\doteq0.17643<1-\frac{\sqrt{6}}{3}$,
where $G_{115}$ is shown in Figure \ref{Fig.12.}.
\\{ \bf Subcase 3.1.7.} If $n_1'=4$, $n_3'\geq n_2'\geq6$. Since $C_5 (T_1,T_2,T_3,T_4, T_5)\cong C_5 (T_2',T_3',T_1',S_1,S_1)$,
then Lemmas \ref{le:5} and \ref{le:2} imply that $\lambda_{2}(U)\leq\lambda_{2}(G_{116})\doteq0.15218<1-\frac{\sqrt{6}}{3}$,
where $G_{116}$ is shown in Figure \ref{Fig.12.}.
\\{ \bf Subcase 3.1.8.} If $n_3'\geq n_2'\geq n_1\geq5$, then $n_3'\geq7$ since $n\geq21$. Since $C_5(T_1,T_2,T_3,T_4, T_5)$ $\cong C_5(T_2',T_3',T_1',S_1,S_1)$,
then Lemmas \ref{le:5} and \ref{le:2} imply that $\lambda_{2}(U)\leq\lambda_{2}(G_{117})\doteq0.15201<1-\frac{\sqrt{6}}{3}$,
where $G_{117}$ is shown in Figure \ref{Fig.12.}.
\begin{figure}[htbp]
  \centering
  \includegraphics[scale=0.55]{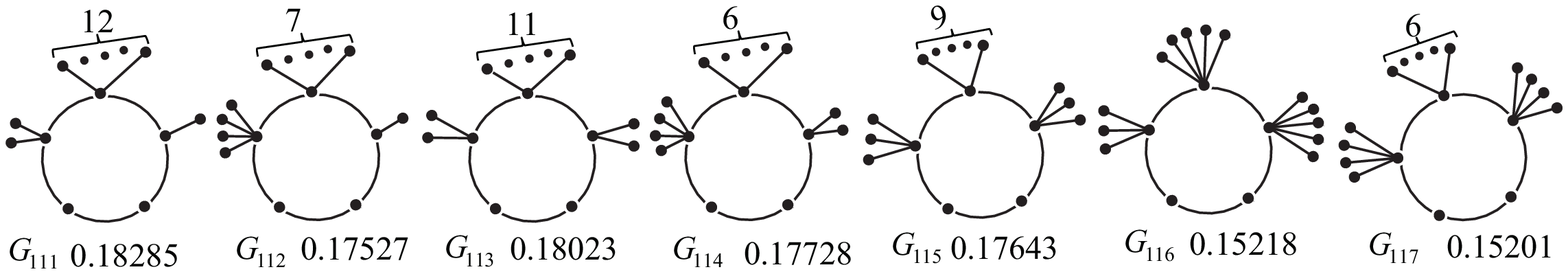}
  \caption{Unicyclic graphs $G_{i}$ and $\lambda_{2}(G_{i})$, $111\leq i\leq117$.}\label{Fig.12.}
\end{figure}
\\{ \bf Subcase 3.2.} $C_5 (T_1,T_2,T_3,T_4,T_5)\ncong C_5(T_2',T_3',T_1',S_1,S_1)$.
\begin{figure}[htbp]
  \centering
  \includegraphics[scale=0.55]{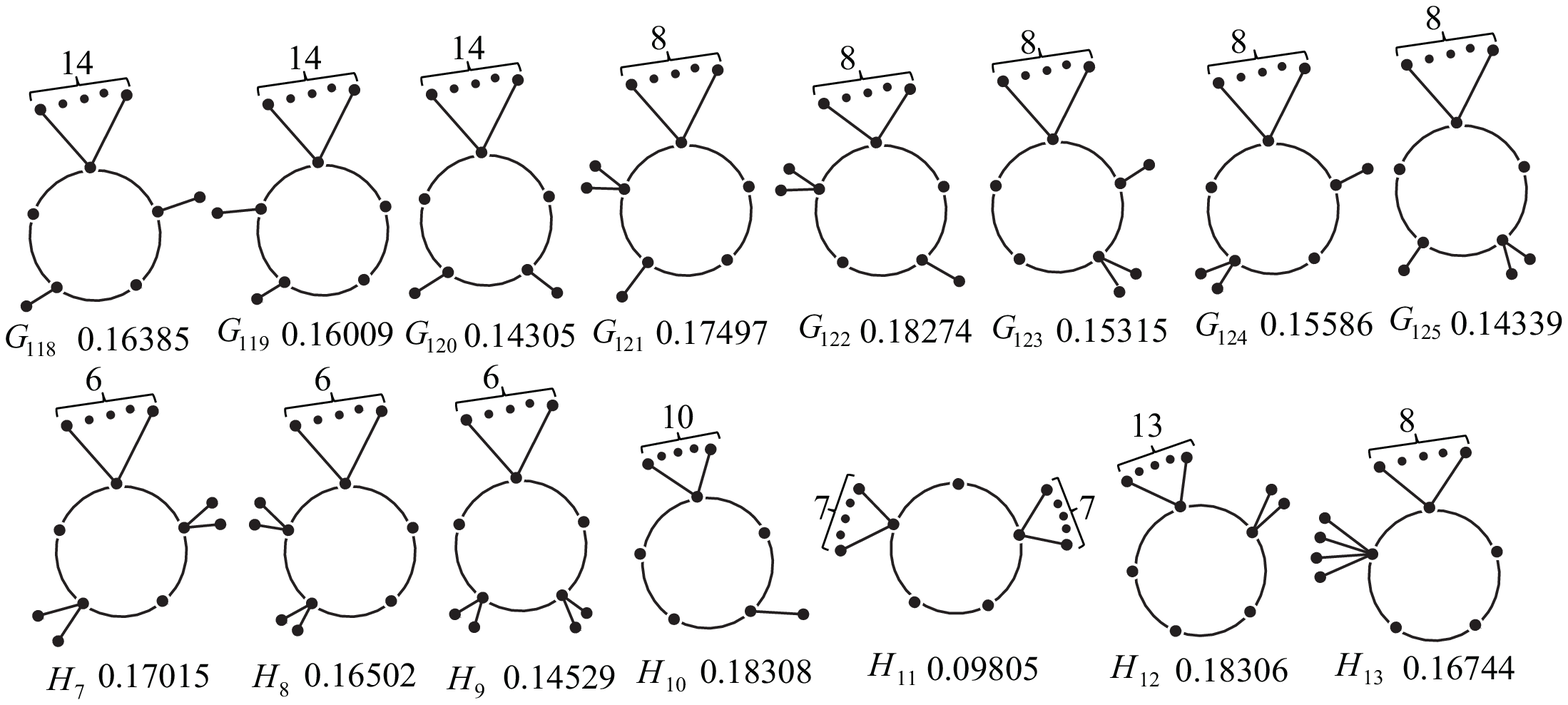}
  \caption{Unicyclic graphs $G_{i}$, $H_j$, and $\lambda_{2}(G_{i})$, $\lambda_{2}(H_{j})$, $118\leq i\leq125$, $7\leq j\leq13$.}\label{Fig.14.}
\end{figure}
\\{ \bf Subcase 3.2.1.} If there are exactly two $n_1'$ and $n_2'$ such that $n_1'=n_2'=2$,
then $n_3'\geq15$ since $n\geq21$. Since $C_5 (T_1,T_2,T_3,T_4, T_5)\ncong C_5 (S_2,T_3',S_2,S_1,S_1)$,
then Lemmas \ref{le:5} and \ref{le:2} imply that $\lambda_{2}(U)\leq\max\{\lambda_{2}(G_i)\mid 118\leq i\leq120\}\doteq0.16385<1-\frac{\sqrt{6}}{3}$,
where $G_{i}$ ($118\leq i\leq120$) are shown in Figure \ref{Fig.14.}.
\\{ \bf Subcase 3.2.2.} If there is exactly one $n_1'$ such that $n_1'=2$, and $n_3'\geq n_2'\geq3$, then $n_3'\geq9$ since $n\geq21$.
Since $C_5 (T_1,T_2,T_3,T_4, T_5)\ncong C_5 (T_2',T_3',S_2,S_1,S_1)$, then Lemmas \ref{le:5} and \ref{le:2} imply that
$\lambda_{2}(C_5(T_1,T_2,T_3,T_4, T_5))\leq\max\{\lambda_{2}(G_i)\mid 121\leq i\leq125\}\doteq0.18274<1-\frac{\sqrt{6}}{3}$,
where $G_{i}$ ($121\leq i\leq125$) are shown in Figure \ref{Fig.14.}.
\\{ \bf Subcase 3.2.3.} If $n_3'\geq n_2'\geq n_1'\geq3$, then $n_3'\geq7$ since $n\geq21$. Since $C_5 (T_1,T_2,T_3,T_4, T_5)\ncong C_5 (T_2',T_3',S_2,S_1,S_1)$, then Lemmas \ref{le:5} and \ref{le:2} imply that $\lambda_{2}(U)\leq\max\{\lambda_{2}(H_j)\mid 7\leq j\leq9\}\doteq0.17015<1-\frac{\sqrt{6}}{3}$,
where $H_{j}$ ($7\leq j\leq9$) are shown in Figure \ref{Fig.14.}.
\\{ \bf Case 4.} $|N|=2$. We assume that $n_1',n_2'\in\{n_1,n_2,n_3,n_4,n_5\}$, and $2\leq n_1'\leq n_2'$.
\\{ \bf Subcase 4.1.} $C_5 (T_1,T_2,T_3,T_4, T_5)\ncong C_5 (T_1',T_2',S_1,S_1,S_1)$.
\\{ \bf Subcase 4.1.1.} If $2\leq n_1'\leq7$, then $n_2'\geq11$ since $n\geq21$. Hence Lemmas \ref{le:5} and \ref{le:2} imply that $\lambda_{2}(U)\leq\lambda_{2}(H_{10})\doteq0.18308<1-\frac{\sqrt{6}}{3}$,
where $H_{10}$ is shown in Figure \ref{Fig.14.}.
\\{ \bf Subcase 4.1.2.} If $n_2'\geq n_1'\geq8$, then Lemmas \ref{le:5} and \ref{le:2} imply that $\lambda_{2}(U)\leq\lambda_{2}(H_{11})\doteq0.09805<1-\frac{\sqrt{6}}{3}$,
where $H_{11}$ is shown in Figure \ref{Fig.14.}.
\\{ \bf Subcase 4.2.}  $C_5 (T_1,T_2,T_3,T_4, T_5)\cong C_5 (T_1',T_2',S_1,S_1,S_1)$.
\\{ \bf Subcase 4.2.1.} If $3\leq n_1'\leq4$, then $n_2'\geq14$ since $n\geq21$. Hence Lemmas \ref{le:5} and \ref{le:2} imply that $\lambda_{2}(U)\leq\lambda_{2}(H_{12})\doteq0.18306<1-\frac{\sqrt{6}}{3}$,
where $H_{12}$ is shown in Figure \ref{Fig.14.}.
\\{ \bf Subcase 4.2.2.} If $n_2'\geq n_1'\geq5$,
then $n_2'\geq9$ since $n\geq21$. Hence \ref{le:5} and \ref{le:2} imply that $\lambda_{2}(U)\leq\lambda_{2}(H_{13})\doteq0.16744<1-\frac{\sqrt{6}}{3}$,
where $H_{13}$ is shown in Figure \ref{Fig.14.}.
\\{ \bf Subcase 4.2.3.} If $n_1'=2$, then $n_2'\geq16$ since $n\geq21$.
Let $d$ be the length of the longest path from $v_2'$ to the pendent vertex in $T_2'$.
In order to make
$\lambda_{2}(C_5 (S_2,T_2',S_1,S_1,S_1))\geq1-\frac{\sqrt{6}}{3}$, then $d\leq2$. If $d\geq3$, then Lemma \ref{le:2} imply that
$\lambda_{2}(C_5 (S_2,T_3',S_1,S_1,S_1))\leq\lambda_{2}(H)\doteq0.14988<1-\frac{\sqrt{6}}{3}$, where $H$ is shown in Figure \ref{Fig.13.}.
Thus $C_5 (S_2,T_2',S_1,S_1,S_1)$ is the unicyclic graph $H_{24}$ as shown in Figure \ref{Fig.16.}, where $l_i$ ($0\leq i\leq t$) are nonnegative integers
and $l_0+2l_1+3l_2+\cdots+(t+1)l_t+6=n\geq21$.
If $l_3+l_4+\cdots+l_t\geq1$, then Lemma \ref{le:2} implies that
$\lambda_{2}(U)\leq\lambda_{2}(H_0)\doteq0.15520<1-\frac{\sqrt{6}}{3}$, where $H_0$ is shown in Figure \ref{Fig.13.}. Thus $l_3+l_4+\cdots+l_t=0$. If $l_2\geq1$, then Lemma \ref{le:2} implies that
$\lambda_{2}(U)\leq\lambda_{2}(H_{25})\doteq0.17497<1-\frac{\sqrt{6}}{3}$, where $H_{25}$ is shown in Figure \ref{Fig.16.}. Thus $l_2=0$. Therefore
$U$ is the unicyclic graph $G_0$, where $G_0$ is shown in Figure \ref{Fig.16.}, $l_0$ and $l_1$ are nonnegative integers, and $l_0+2l_1+6=n\geq21$.
Since $\lambda_{2}(S_5(2,31,1,1,1))=0.18329<1-\frac{\sqrt{6}}{3}$, If $l_0+2l_1\geq30$, then Lemmas \ref{le:5} and \ref{le:2}
imply that $\lambda_{2}(U)\leq\lambda_{2}(S_5(2,31,1,1,1))=0.18329<1-\frac{\sqrt{6}}{3}$. Thus $15\leq l_0+2l_1\leq29$. If $l_1\geq8$,
then Lemmas \ref{le:3} and \ref{le:2} imply that
$\lambda_{2}(U)\leq\lambda_{2}(H_{26})\doteq0.18215<1-\frac{\sqrt{6}}{3}$, where $H_{26}$ is shown in Figure \ref{Fig.16.}. Thus $l_1\leq7$.
\\{ \bf Subcase 4.2.3.1.} $l_1=7$, then $1\leq l_0\leq15$.
\begin{figure}[htbp]
  \centering
  \includegraphics[scale=0.55]{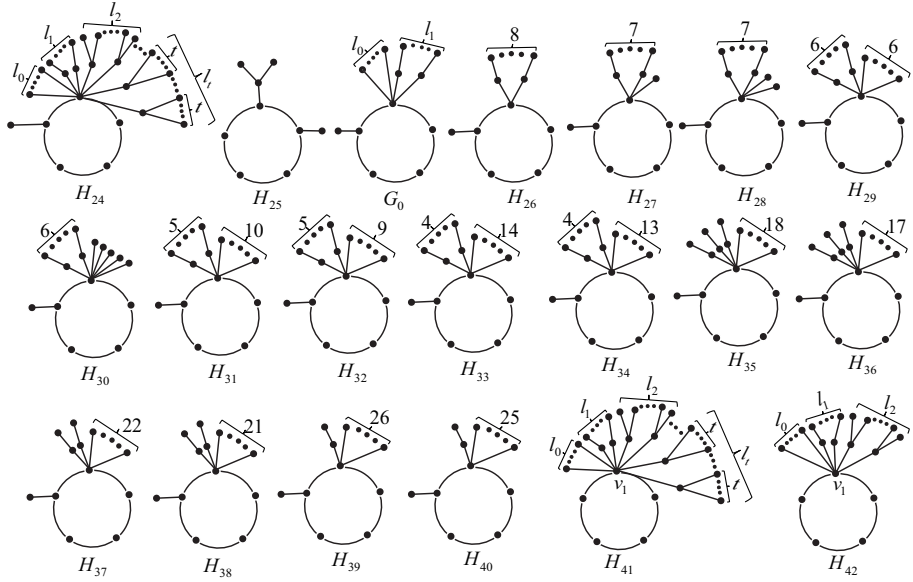}
  \caption{Unicyclic graphs $G_0$ and $H_i$, $24\leq i\leq42$.}\label{Fig.16.}
\end{figure}
\\ If $l_0=1$, then $\lambda_{2}(U)=\lambda_{2}(H_{27})\doteq0.18395>1-\frac{\sqrt{6}}{3}$, where $H_{27}$ is shown in Figure \ref{Fig.16.}. Otherwise, if $2\leq l_0\leq15$, then Lemma \ref{le:2} implies  that $\lambda_{2}(U)\leq\lambda_{2}(H_{28})\doteq0.18331<1-\frac{\sqrt{6}}{3}$, where $H_{28}$ is shown in Figure \ref{Fig.16.}.
\\{ \bf Subcase 4.2.3.2.} $l_1=6$, then $3\leq l_0\leq17$.
\\ If $6\leq l_0\leq17$, then Lemma \ref{le:2} implies  that $\lambda_{2}(U)\leq\lambda_{2}(H_{29})\doteq0.18331<1-\frac{\sqrt{6}}{3}$, where $H_{29}$ is shown in Figure \ref{Fig.16.}. Otherwise, if $3\leq l_0\leq5$, then Lemma \ref{le:2} implies  that $\lambda_{2}(U)\geq\lambda_{2}(H_{30})\doteq0.18396>1-\frac{\sqrt{6}}{3}$, where $H_{30}$ is shown in Figure \ref{Fig.16.}.
\\{ \bf Subcase 4.2.3.3.} $l_1=5$, then $5\leq l_0\leq19$.
\\ If $10\leq l_0\leq19$, then Lemma \ref{le:2} implies  that $\lambda_{2}(U)\leq\lambda_{2}(H_{31})\doteq0.18331<1-\frac{\sqrt{6}}{3}$, where $H_{31}$ is shown in Figure \ref{Fig.16.}. Otherwise, if $5\leq l_0\leq9$, then Lemma \ref{le:2} implies  that $\lambda_{2}(U)\geq\lambda_{2}(H_{32})\doteq0.18397>1-\frac{\sqrt{6}}{3}$, where $H_{32}$ is shown in Figure \ref{Fig.16.}.
\\{ \bf Subcase 4.2.3.4.} $l_1=4$, then $7\leq l_0\leq21$.
\\ If $14\leq l_0\leq21$, then Lemma \ref{le:2} implies  that $\lambda_{2}(U)\leq\lambda_{2}(H_{33})\doteq0.18330<1-\frac{\sqrt{6}}{3}$, where $H_{33}$ is shown in Figure \ref{Fig.16.}. Otherwise, if $7\leq l_0\leq13$, then Lemma \ref{le:2} implies  that $\lambda_{2}(U)\geq\lambda_{2}(H_{34})\doteq0.18398>1-\frac{\sqrt{6}}{3}$, where $H_{34}$ is shown in Figure \ref{Fig.16.}.
\\{ \bf Subcase 4.2.3.5.} $l_1=3$, then $9\leq l_0\leq23$.
\\ If $18\leq l_0\leq23$, then Lemma \ref{le:2} implies  that $\lambda_{2}(U)\leq\lambda_{2}(H_{35})\doteq0.18330<1-\frac{\sqrt{6}}{3}$, where $H_{35}$ is shown in Figure \ref{Fig.16.}. Otherwise, if $9\leq l_0\leq17$, then Lemma \ref{le:2} implies  that $\lambda_{2}(U)\geq\lambda_{2}(H_{36})\doteq0.18399>1-\frac{\sqrt{6}}{3}$, where $H_{36}$ is shown in Figure \ref{Fig.16.}.
\\{ \bf Subcase 4.2.3.6.} $l_1=2$, then $11\leq l_0\leq25$.
\\ If $22\leq l_0\leq25$, then Lemma \ref{le:2} implies  that $\lambda_{2}(U)\leq\lambda_{2}(H_{37})\doteq0.18329<1-\frac{\sqrt{6}}{3}$, where $H_{37}$ is shown in Figure \ref{Fig.16.}. Otherwise, if $11\leq l_0\leq21$, then Lemma \ref{le:2} implies  that $\lambda_{2}(U)\geq\lambda_{2}(H_{38})\doteq0.18400>1-\frac{\sqrt{6}}{3}$, where $H_{38}$ is shown in Figure \ref{Fig.16.}.
\\{ \bf Subcase 4.2.3.7.} $l_1=1$, then $13\leq l_0\leq27$.
\\ If $26\leq l_0\leq27$, then Lemma \ref{le:2} implies that $\lambda_{2}(U)\leq\lambda_{2}(H_{39})\doteq0.18329<1-\frac{\sqrt{6}}{3}$, where $H_{39}$ is shown in Figure \ref{Fig.16.}. Otherwise, if $13\leq l_0\leq25$, then Lemma \ref{le:2} implies  that $\lambda_{2}(U)\geq\lambda_{2}(H_{40})\doteq0.18401>1-\frac{\sqrt{6}}{3}$, where $H_{40}$ is shown in Figure \ref{Fig.16.}.
\\{ \bf Subcase 4.2.3.8.} $l_1=0$, then $15\leq l_0\leq29$.
\\ By Lemma \ref{le:2}, we have $\lambda_{2}(U)\geq\lambda_{2}(S_5(2,30,1,1,1))\doteq0.18402>1-\frac{\sqrt{6}}{3}$.
\\{ \bf Case 5.} $|N|=1$. Without loss of generality, we assume that $n_1\geq2$. That is $C_5(T_1,T_2,T_3,$ $T_4,T_5)=C_5(T_1,S_1,S_1,S_1,S_1)$.
Let $d$ be the length of the longest path from $v_1$ to the pendent vertex in $T_1$.
In order to make
$\lambda_{2}(C_5(T_1,S_1,S_1,S_1,S_1))\geq1-\frac{\sqrt{6}}{3}$, then $d\leq2$. If $d\geq3$, then Lemma \ref{le:2} implies  that
$\lambda_{2}(C_5(T_1,S_1,S_1,S_1,S_1))\leq\lambda_{2}(H)\doteq0.14988<1-\frac{\sqrt{6}}{3}$, where $H$ is shown in Figure \ref{Fig.13.}.
Thus $C_5(T_1,S_1,S_1,S_1,S_1)$ is the unicyclic graph $H_{41}$ as shown in Figure \ref{Fig.16.}, where $l_i$ ($0\leq i\leq t$) are nonnegative integers
and $l_0+2l_1+3l_2+\cdots+(t+1)l_t+5=n\geq21$. If $l_3+l_4+\cdots+l_t\geq1$, then Lemma \ref{le:2} implies that
$\lambda_{2}(U)\leq\lambda_{2}(H_0)\doteq0.15520<1-\frac{\sqrt{6}}{3}$, where $H_0$ is shown in Figure \ref{Fig.13.}. Thus $l_3+l_4+\cdots+l_t=0$.
\\Therefore, in order to make
$\lambda_{2}(C_5(T_1,S_1,S_1,S_1,S_1))\geq1-\frac{\sqrt{6}}{3}$, $C_5(T_1,S_1,S_1,S_1,S_1)$ is the unicyclic graph $H_{42}$ as shown in Figure \ref{Fig.16.}, where $l_i$ ($0\leq i\leq 2$) are nonnegative integers
and $l_0+2l_1+3l_2+5=n\geq21$. Recall that $v_1$ is a vertex degree $l_0+l_1+l_2+2$ in $H_{42}$. Note that the eigenvalues of
$\mathcal{L}_{v_1}(H_{42})$ are $\underbrace{1-\frac{\sqrt{6}}{3},\cdots,1-\frac{\sqrt{6}}{3}}_{l_{2}}$, $\frac{3}{4}-\frac{\sqrt{5}}{4}$, $\underbrace{1-\frac{\sqrt{2}}{2},\cdots,1-\frac{\sqrt{2}}{2}}_{l_{1}}$, $\frac{5}{4}-\frac{\sqrt{5}}{4}$, $\underbrace{1,\cdots,1}_{l_{0}+l_{2}}$, $\frac{3}{4}+\frac{\sqrt{5}}{4}$, $\underbrace{1+\frac{\sqrt{2}}{2},\cdots,1+\frac{\sqrt{2}}{2}}_{l_{1}}$, $\frac{5}{4}+\frac{\sqrt{5}}{4}$, $\underbrace{1+\frac{\sqrt{6}}{3},\cdots,1+\frac{\sqrt{6}}{3}}_{l_{2}}$.
\\If $l_2\geq2$, then Lemma \ref{le:1} implies that $\lambda_{2}(H_{42})=1-\frac{\sqrt{6}}{3}$.
Analogously, when $l_2=1$, $\lambda_{2}(H_{42})\geq1-\frac{\sqrt{6}}{3}$, and when $l_2=0$, $\lambda_{2}(H_{42})\geq\frac{3}{4}-\frac{\sqrt{5}}{4}>1-\frac{\sqrt{6}}{3}$.
\end{proof}
\section{ All unicyclic graphs of girth $4$ with $\lambda_{2}\geq1-\frac{\sqrt{6}}{3}$}

In this section, we determine all unicyclic graphs in $\mathcal{U}_n^4$ ($n\geq21$)with $\lambda_{2}\geq1-\frac{\sqrt{6}}{3}$.
Before introducing our results we recall some notation. Note that if $U\in\mathcal{U}_n^4$, then $U $consists
of the cycle $C_4=v_1v_2v_3v_4v_1$ and four trees $T_1$, $T_2$, $T_3$ and $T_4$ attached at the vertices $v_1$, $v_2$,
 $v_3$ and $v_4$, respectively.
 \begin{figure}[htbp]
  \centering
  \includegraphics[scale=0.6]{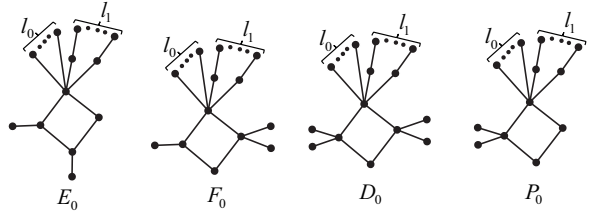}
  \caption{Unicyclic graphs $E_{0}$, $F_{0}$, $D_0$ and $P_0$.}\label{Fig.0.}
\end{figure}
\noindent\begin{theorem}\label{th:2} Suppose that $U\in\mathcal{U}_n^4$ with $n\geq21$.
Then $\lambda_{2}(U)\geq1-\frac{\sqrt{6}}{3}$ if and only if one of the following items holds:

(i) $U$ is isomorphic to $S_4(2,2,2,k)$ with $15\leq k\leq17$.

(ii) $U$ is isomorphic to $H_{65}$, where $H_{65}$ is shown in Figure \ref{Fig.20.}.

(iii) $U$ is the unicyclic graph $E_0$, where $E_0$ is shown in Figure \ref{Fig.0.}, $l_0$ and $l_1$ are nonnegative integers,
$15\leq l_0+2l_1\leq25$, $0\leq l_1\leq5$ and one of the following items holds:

(1) $l_1=5$ and $l_0=5$.

(2) $l_1=4$ and $7\leq l_0\leq9$.

(3) $l_1=3$ and $9\leq l_0\leq13$.

(4) $l_1=2$ and $11\leq l_0\leq17$.

(5) $l_1=1$ and $13\leq l_0\leq21$.

(6) $l_1=0$ and $15\leq l_0\leq25$.

(iv) $U$ is isomorphic to $H_{94}$, where $H_{94}$ as shown in  Figure \ref{Fig.23.}, and $l_i$ ($0\leq i\leq 2$) are nonnegative integers
and $l_0+2l_1+3l_2+6=n\geq21$.

(v) $U$ is the unicyclic graph $F_0$, where $F_0$ is shown in Figure \ref{Fig.0.}, $l_0$ and $l_1$ are nonnegative integers,
$14\leq l_0+2l_1\leq24$, $0\leq l_1\leq5$ and one of the following items holds:

(1) $l_1=5$ and $l_0=4$.

(2) $l_1=4$ and $6\leq l_0\leq8$.

(3) $l_1=3$ and $8\leq l_0\leq12$.

(4) $l_1=2$ and $10\leq l_0\leq16$.

(5) $l_1=1$ and $12\leq l_0\leq20$.

(6) $l_1=0$ and $14\leq l_0\leq24$.

(vi) $U$ is the unicyclic graph $D_0$, where $D_0$ is shown in Figure \ref{Fig.0.}, $l_0$ and $l_1$ are nonnegative integers,
$13\leq l_0+2l_1\leq16$, $0\leq l_1\leq1$ and one of the following items holds:

(1) $l_1=1$ and $11\leq l_0\leq12$.

(2) $l_1=0$ and $13\leq l_0\leq16$.

(vii) $U$ is isomorphic to $A_8$, where $A_8$ as shown in  Figure \ref{Fig.26.}, and $l_i$ ($0\leq i\leq 2$) are nonnegative integers
and $l_0+2l_1+3l_2+5=n\geq21$.

(viii) $U$ is isomorphic to $A_{11}$, where $A_{11}$ as shown in \ref{Fig.26.}, and $l_i$ ($0\leq i\leq 2$) are nonnegative integers
and $l_0+2l_1+3l_2+5=n\geq21$.

(ix) $U$ is the unicyclic graph $P_0$, where $P_0$ is shown in Figure \ref{Fig.0.}, $l_0$ and $l_1$ are nonnegative integers,
$15\leq l_0+2l_1\leq40$, $0\leq l_1\leq10$ and one of the following items holds:

(1) $l_1=10$ and $l_0=0$.

(2) $l_1=9$ and $0\leq l_0\leq4$.

(3) $l_1=8$, then $0\leq l_0\leq8$.

(4) $l_1=7$ and $1\leq l_0\leq12$.

(5) $l_1=6$ and $3\leq l_0\leq16$.

(6) $l_1=5$ and $5\leq l_0\leq20$.

(7) $l_1=4$ and $7\leq l_0\leq24$.

(8) $l_1=3$ and $9\leq l_0\leq28$.

(9) $l_1=2$ and $11\leq l_0\leq32$.

(10) $l_1=1$ and $13\leq l_0\leq26$.

(11) $l_1=0$ and  $15\leq l_0\leq40$.

(x) $U$ is isomorphic to $A_{35}$, where $A_{35}$ as shown in  Figure \ref{Fig.28.}, and $l_i$ ($0\leq i\leq 2$) are nonnegative integers
and $l_0+2l_1+3l_2+4=n\geq21$.

Furthermore, the equality holds if and only if $U$ is isomorphic to $H_{65}$, or $S_4(2,2,2,17)$, or $H_{82}$,
or $H_{84}$, or $H_{86}$, or $H_{88}$, or $H_{90}$,  where $H_{82}$, $H_{84}$, $H_{86}$, $H_{88}$ and $H_{90}$ are shown in  Figure \ref{Fig.22.}, or $S_4(2,2,26,1)$, or $H_{94}$ with $l_2\geq1$, or $H_{99}$, or $H_{101}$, or $H_{103}$, or $H_{105}$, or $H_{107}$, where $H_{99}$, $H_{101}$, $H_{103}$, $H_{105}$ and $H_{107}$ are shown in Figure \ref{Fig.24.}, or $S_4(3,25,2,1)$, or $A_{4}$, where $A_{4}$ is shown in Figure \ref{Fig.25.}, or $S_4(3,17,3,1)$, or $A_{8}$ with $l_2\geq1$, or $A_{11}$ with $l_2\geq2$, or $A_{13}$, or $A_{15}$,
or $A_{17}$, or $A_{19}$, or $A_{21}$, or $A_{23}$, or $A_{25}$, or $A_{27}$, or $A_{29}$, or $A_{31}$, where $A_{13}$, $A_{15}$, $A_{17}$, $A_{19}$, $A_{21}$, $A_{23}$, $A_{25}$, $A_{27}$, $A_{29}$ and $A_{31}$ are shown in Figure \ref{Fig.27.}, or $S_4(3,41,1,1)$, or $A_{35}$ with
$l_2\geq2$
\end{theorem}
\begin{proof} We rewrite $U$ in the form $C_4(T_1,T_2,T_3,T_4)$, where $|V(T_i)|=n_i\geq1$ for $i = 1,2,3,4$ and $\sum\limits_{i=1}^4 n_i = n\geq21$. Let $N=\{i\mid n_i\geq2, i=1,2,3,4\}$. We consider the following cases.
\\{ \bf Case 1.} $|N|=4$. We assume that $n_1',n_2',n_3',n_4'\in\{n_1,n_2,n_3,n_4,\}$ and $2\leq n_1'\leq n_2'\leq n_3'\leq n_4'$.
\\{ \bf Subcase 1.1.} If there are exactly one $n_1'$ such that $n_1'=2$, and there are exactly two $n_2'$ and $n_3'$ such that $n_2'=n_3'=3$, then $n_4'\geq13$ since $n\geq21$. Hence Lemmas \ref{le:5} and \ref{le:2} imply that $\lambda_{2}(C_4(T_1,T_2,T_3,T_4))\leq\max\{\lambda_{2}(H_i)\mid 43\leq i\leq44\}\doteq0.17781<1-\frac{\sqrt{6}}{3}$,
where $H_{i}$ ($43\leq i\leq44$) are shown in Figure \ref{Fig.19.}.
\\{ \bf Subcase 1.2.} If $n_1'=2$, $n_2'=3$, and $4\leq n_3'\leq6$, then $n_4'\geq10$ since $n\geq21$. Hence Lemmas \ref{le:5} and \ref{le:2} imply that $\lambda_{2}(C_4(T_1,T_2,T_3,T_4))\leq\max\{\lambda_{2}(H_i)\mid 45\leq i\leq47\}\doteq0.17963<1-\frac{\sqrt{6}}{3}$,
where $H_{i}$ ($45\leq i\leq47$) are shown in Figure \ref{Fig.19.}.
\\{ \bf Subcase 1.3.} If $n_1'=2$, $n_2'=3$, and $n_4'\geq n_3'\geq7$. Hence Lemmas \ref{le:5} and \ref{le:2} imply that $\lambda_{2}(C_4(T_1,T_2,T_3,T_4))\leq\max\{\lambda_{2}(H_i)\mid 48\leq i\leq49\}\doteq0.16580<1-\frac{\sqrt{6}}{3}$,
where $H_{i}$ ($48\leq i\leq49$) are shown in Figure \ref{Fig.19.}.
\begin{figure}[htbp]
  \centering
  \includegraphics[scale=0.55]{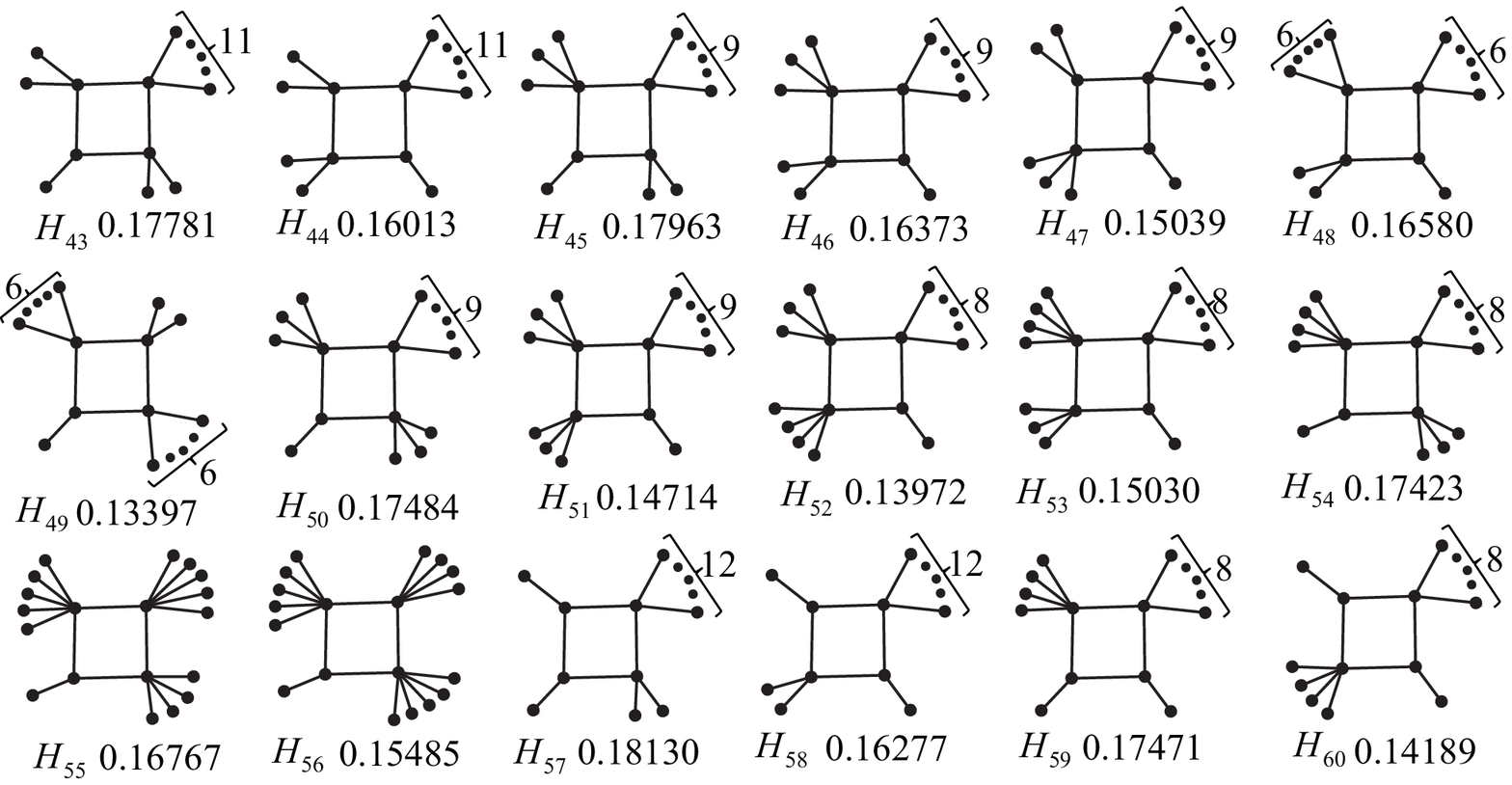}
  \caption{Unicyclic graphs $H_{i}$ and $\lambda_{2}(H_{i})$, $43\leq i\leq60$.}\label{Fig.19.}
\end{figure}
\\{ \bf Subcase 1.4.} If there are exactly one $n_1'$ such that $n_1'=2$, and there are exactly two $n_2'$ and $n_3'$ such that $n_2'=n_3'=4$, then $n_4'\geq10$ since $n\geq21$.  Hence Lemmas \ref{le:5} and \ref{le:2} imply that $\lambda_{2}(C_4(T_1,T_2,T_3,T_4))\leq\max\{\lambda_{2}(H_i)\mid 50\leq i\leq51\}\doteq0.17484<1-\frac{\sqrt{6}}{3}$,
where $H_{i}$ ($50\leq i\leq51$) are shown in Figure \ref{Fig.19.}.
\\{ \bf Subcase 1.5.} If $n_1'=2$, $n_2'=4$, and $5\leq n_3'\leq6$, then $n_4'\geq9$ since $n\geq21$. Hence Lemmas \ref{le:5} and \ref{le:2} imply that $\lambda_{2}(C_4(T_1,T_2,T_3,T_4))\leq\max\{\lambda_{2}(H_i)\mid 52\leq i\leq54\}\doteq0.17423<1-\frac{\sqrt{6}}{3}$,
where $H_{i}$ ($52\leq i\leq54$) are shown in Figure \ref{Fig.19.}.
\\{ \bf Subcase 1.6.} If $n_1'=2$, $n_2'=4$, and $n_4'\geq n_3'\geq7$. Hence Lemmas \ref{le:5} and \ref{le:2} imply that $\lambda_{2}(C_4(T_1,T_2,T_3,T_4))\leq\max\{\lambda_{2}(H_i)\mid 48\leq i\leq49\}\doteq0.16580<1-\frac{\sqrt{6}}{3}$,
where $H_{i}$ ($48\leq i\leq49$) are shown in Figure \ref{Fig.19.}.
\\{ \bf Subcase 1.7.} If there are exactly one $n_1'$ such that $n_1'=2$, and there are exactly two $n_2'$ and $n_3'$ such that $n_2'=n_3'=5$, then $n_4'\geq9$ since $n\geq21$. Hence Lemmas \ref{le:5} and \ref{le:2} imply that $\lambda_{2}(C_4(T_1,T_2,T_3,T_4))\leq\max\{\lambda_{2}(H_i)\mid 53\leq i\leq54\}\doteq0.17423<1-\frac{\sqrt{6}}{3}$,
where $H_{i}$ ($53\leq i\leq54$) are shown in Figure \ref{Fig.19.}.
\\{ \bf Subcase 1.8.} If there are exactly one $n_1'$ such that $n_1'=2$, and there are at most one $n_2'$ such that $n_2'=5$, and $n_4'\geq n_3'\geq6$. Hence Lemmas \ref{le:5} and \ref{le:2} imply that $\lambda_{2}(C_4(T_1,T_2,T_3,T_4))\leq\max\{\lambda_{2}(H_i)\mid 55\leq i\leq56\}\doteq0.16767<1-\frac{\sqrt{6}}{3}$, where $H_{i}$ ($55\leq i\leq56$) are shown in Figure \ref{Fig.19.}.
\\{ \bf Subcase 1.9.} If there are exactly two $n_1'$ and $n_2'$ such that $n_1'=n_2'=2$, and $3\leq n_3'\leq4$, then $n_4'\geq13$ since $n\geq21$. Hence Lemmas \ref{le:5} and \ref{le:2} imply that $\lambda_{2}(C_4(T_1,T_2,T_3,T_4))\leq\max\{\lambda_{2}(H_i)\mid 57\leq i\leq58\}\doteq0.18130<1-\frac{\sqrt{6}}{3}$,
where $H_{i}$ ($57\leq i\leq58$) are shown in Figure \ref{Fig.19.}.
\\{ \bf Subcase 1.10.} If there are exactly two $n_1'$ and $n_2'$ such that $n_1'=n_2'=2$, and $n_4'\geq n_3'\geq5$, then $n_4'\geq9$ since $n\geq21$. Hence Lemmas \ref{le:5} and \ref{le:2} imply that $\lambda_{2}(C_4(T_1,T_2,T_3,T_4))\leq\max\{\lambda_{2}(H_i)\mid 59\leq i\leq60\}\doteq0.17471<1-\frac{\sqrt{6}}{3}$,
where $H_{i}$ ($59\leq i\leq60$) are shown in Figure \ref{Fig.19.}.
\\{ \bf Subcase 1.11.} If $n_1'=n_2'=n_3'=2$, then $n_4'\geq15$ since $n\geq21$.
Since $\lambda_{2}(C_4(T_1,T_2,T_3,T_4))\leq\lambda_{2}(S_4(2,2,2,n_4'))$. If $n_4'\geq18$, by Lemma \ref{le:2}, we have $\lambda_{2}(U)\leq\lambda_{2}(S_4(2,2,2,18))\doteq0.18120<1-\frac{\sqrt{6}}{3}$. Therefore, in order to make
$\lambda_{2}(C_4(T_1,T_2,T_3,T_4))\geq1-\frac{\sqrt{6}}{3}$, we have $15\leq n_4'\leq17$.
Let $d$ be the length of the longest path from $v_4'$ to the pendent vertex in $T_4'$. In order to make
$\lambda_{2}(C_4(T_1,T_2,T_3,T_4))\geq1-\frac{\sqrt{6}}{3}$, then $d\leq2$. If $d\geq3$, then Lemma \ref{le:2} implies  that
$\lambda_{2}(C_4(T_1,T_2,T_3,T_4))\leq\lambda_{2}(H_{61})\doteq0.16959<1-\frac{\sqrt{6}}{3}$, where $H_{61}$ is shown in Figure \ref{Fig.20.}.
Thus $C_4(T_1,T_2,T_3,T_4)$ is the unicyclic graph $H_{62}$ as shown in Figure \ref{Fig.20.}, where $l_i$ ($0\leq i\leq t$) are nonnegative integers
and $21\leq l_0+2l_1+3l_2+\cdots+(t+1)l_t+7=n\leq23$. If $l_2+l_3+l_4+\cdots+l_t\geq1$, then Lemma \ref{le:2} implies  that
$\lambda_{2}(C_4(T_1,T_2,T_3,T_4))\leq\lambda_{2}(H_{63})\doteq0.17353<1-\frac{\sqrt{6}}{3}$, where $H_{63}$ is shown in Figure \ref{Fig.20.}. Thus $l_2+l_3+l_4+\cdots+l_t=0$. If $l_1\geq2$, then Lemma \ref{le:2} implies  that
$\lambda_{2}(C_4(T_1,T_2,T_3,T_4))\leq\lambda_{2}(H_{64})\doteq0.17958<1-\frac{\sqrt{6}}{3}$,
where $H_{64}$ is shown in Figure \ref{Fig.20.}. Thus $l_1\leq1$, and $14\leq l_0+2l_1\leq16$.
\begin{figure}[htbp]
  \centering
  \includegraphics[scale=0.55]{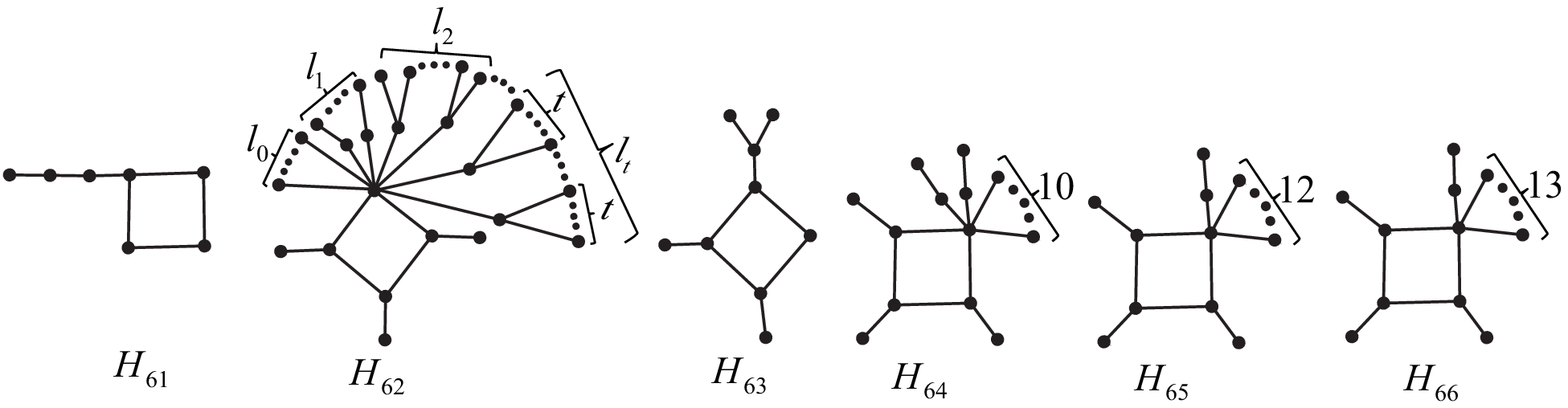}
  \caption{Unicyclic graphs $H_{i}$ and $\lambda_{2}(H_{i})$, $61\leq i\leq66$.}\label{Fig.20.}
\end{figure}
\\{ \bf Subcase 1.11.1.} $l_1=1$, then $12\leq l_0\leq14$.
\\If $l_0=12$, then Lemma \ref{le:2} implies  that $\lambda_{2}(C_4(T_1,T_2,T_3,T_4))=\lambda_{2}(H_{65})=1-\frac{\sqrt{6}}{3}$, where $H_{65}$ is shown in Figure \ref{Fig.20.}. Otherwise, if $14\geq l_0\geq13$, then Lemma \ref{le:2} implies that $\lambda_{2}(C_4(T_1,T_2,T_3,T_4))\leq$ $\lambda_{2}(H_{66})\doteq0.18134<1-\frac{\sqrt{6}}{3}$, where $H_{66}$ is shown in Figure \ref{Fig.20.}.
\\{ \bf Subcase 1.11.2.} $l_1=0$, then $14\leq l_0\leq16$.
\\By Lemma \ref{le:2}, we have $\lambda_{2}(C_4(T_1,T_2,T_3,T_4))\geq\lambda_{2}(S_5(2,2,2,17))=1-\frac{\sqrt{6}}{3}$.
\\{ \bf Subcase 1.12.} If $n_1'=n_2'=n_3'=3$, then $n_4'\geq12$ since $n\geq21$. Hence Lemmas \ref{le:5} and \ref{le:2} imply that $\lambda_{2}(C_4(T_1,T_2,T_3,T_4))\leq\lambda_{2}(H_{67})\doteq0.17444<1-\frac{\sqrt{6}}{3}$,
where $H_{67}$ is shown in Figure \ref{Fig.21.}.
\\{ \bf Subcase 1.13.} If $n_1'=n_2'=3$, and $n_4'\geq n_3'\geq4$, then $n_4'\geq8$ since $n\geq21$. Hence Lemmas \ref{le:5} and \ref{le:2} imply that $\lambda_{2}(C_4(T_1,T_2,T_3,T_4))\leq\max\{\lambda_{2}(H_i)\mid 68\leq i\leq69\}\doteq0.17881<1-\frac{\sqrt{6}}{3}$,
where $H_{i}$ ($68\leq i\leq69$) are shown in Figure \ref{Fig.21.}.
\\{ \bf Subcase 1.14.} If $n_1'=3$, $5\geq n_3'\geq n_2'\geq4$, and $n_4'\geq8$. Hence Lemmas \ref{le:5} and \ref{le:2} imply that $\lambda_{2}(C_4(T_1,T_2,T_3,T_4))\leq\max\{\lambda_{2}(H_i)\mid 68\leq i\leq69\}\doteq0.17881<1-\frac{\sqrt{6}}{3}$,
where $H_{i}$ ($68\leq i\leq69$) are shown in Figure \ref{Fig.21.}.
\\{ \bf Subcase 1.15.} If $n_1'=3$, $5\geq n_2'\geq4$, and $n_4'\geq n_3'\geq6$. Hence Lemmas \ref{le:5} and \ref{le:2} imply that $\lambda_{2}(C_4(T_1,T_2,T_3,T_4))\leq\max\{\lambda_{2}(H_i)\mid 70\leq i\leq71\}\doteq0.17957<1-\frac{\sqrt{6}}{3}$,
where $H_{i}$ ($70\leq i\leq71$) are shown in Figure \ref{Fig.21.}.
\\{ \bf Subcase 1.16.} If $n_1'=3$, then $n_4'\geq n_3'\geq n_2'\geq6$. Hence Lemmas \ref{le:5} and \ref{le:2} imply that $\lambda_{2}(C_4(T_1,T_2,T_3,T_4))\leq\lambda_{2}(H_{72})\doteq0.15485<1-\frac{\sqrt{6}}{3}$,
where $H_{72}$ is shown in Figure \ref{Fig.21.}.
\begin{figure}[htbp]
  \centering
  \includegraphics[scale=0.55]{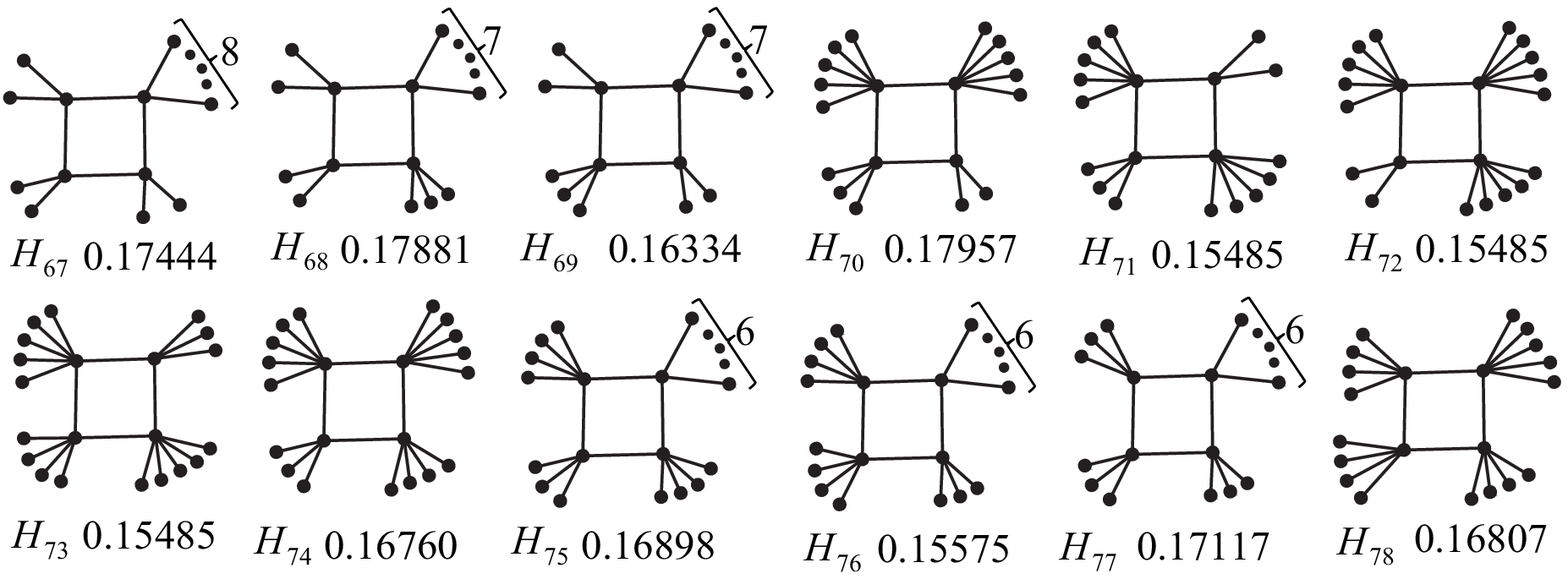}
  \caption{Unicyclic graphs $H_{i}$ and $\lambda_{2}(H_{i})$ , $67\leq i\leq78$.}\label{Fig.21.}
\end{figure}
\\{ \bf Subcase 1.17.} If $n_1'=4$, then $n_4'\geq n_3'\geq n_2'\geq5$. Since $n\geq21$, then by Lemmas \ref{le:5} and \ref{le:2} imply that $\lambda_{2}(C_4(T_1,T_2,T_3,T_4))\leq\max\{\lambda_{2}(H_i)\mid 73\leq i\leq76\}\doteq0.16898<1-\frac{\sqrt{6}}{3}$,
where $H_{i}$ ($73\leq i\leq76$) are shown in Figure \ref{Fig.21.}.
\\{ \bf Subcase 1.18.} If there are at least two $n_1'$ and $n_2'$ such that $n_1'= n_2'=4$, $n_4'\geq n_3'\geq4$, then $n_4'\geq7$ since $n\geq21$. Hence Lemmas \ref{le:5} and \ref{le:2} imply that $\lambda_{2}(C_4(T_1,T_2,T_3,T_4))\leq\lambda_{2}(H_{77})\doteq0.17117<1-\frac{\sqrt{6}}{3}$,
where $H_{77}$ is shown in Figure \ref{Fig.21.}.
\\{ \bf Subcase 1.19.} If $n_4'\geq n_3'\geq n_2'\geq n_1'\geq5$, then $n_4'\geq6$ since $n\geq21$. Hence Lemmas \ref{le:5} and \ref{le:2} imply that $\lambda_{2}(C_4(T_1,T_2,T_3,T_4))\leq\lambda_{2}(H_{79})\doteq0.16807<1-\frac{\sqrt{6}}{3}$,
where $H_{78}$ is shown in Figure \ref{Fig.21.}.
\\{ \bf Case 2.} $|N|=3$. We assume that $n_1',n_2',n_3'\in\{n_1,n_2,n_3,n_4\}$ and $2\leq n_1'\leq n_2'\leq n_3'$.
\\{ \bf Subcase 2.1.} If $n_1'=n_2'=2$, then $n_3'\geq16$ since $n\geq21$.
\\{\bf Subcase 2.1.1.} If $C_4(T_1,T_2,T_3,T_4)\cong C_4(T_1',T_2',T_3',1)$, then $C_4(T_1,T_2,T_3,T_4)\cong C_4(S_2,S_2,T_3',$ $S_1)$.
If $n_3'\geq27$, by Lemmas \ref{le:5} and \ref{le:2}, we have $\lambda_{2}(C_4(S_2,S_2,T_3',S_1))\leq\lambda_{2}(S_4(2,2,n_3',1))\leq\lambda_{2}(S_4(2,2,27,1))\doteq0.18257<1-\frac{\sqrt{6}}{3}$. Therefore, in order to make
$\lambda_{2}(C_4(S_2,S_2,T_3',S_1))\geq1-\frac{\sqrt{6}}{3}$, we have $16\leq n_3'\leq26$.
Let $d$ be the length of the longest path from $v_3'$ to the pendent vertex in $T_3'$. In order to make
$\lambda_{2}(C_4(S_2,S_2,T_3',S_1))\geq1-\frac{\sqrt{6}}{3}$, then $d\leq2$. If $d\geq3$, then Lemma \ref{le:2} implies  that
$\lambda_{2}(C_4(S_2,S_2,T_3',S_1))\leq\lambda_{2}(H_{61})\doteq0.16959<1-\frac{\sqrt{6}}{3}$, where $H_{61}$ is shown in Figure \ref{Fig.20.}.
Thus $C_4(2,2,T_3',1)$ is the unicyclic graph $H_{80}$ as shown in Figure \ref{Fig.22.}, where $l_i$ ($0\leq i\leq t$) are nonnegative integers
and $21\leq l_0+2l_1+3l_2+\cdots+(t+1)l_t+6=n\leq31$. If $l_2+l_3+l_4+\cdots+l_t\geq1$, then Lemma \ref{le:2} implies  that
$\lambda_{2}(C_4(S_2,S_2,T_3',S_1))\leq\lambda_{2}(H_{63})\doteq0.17353<1-\frac{\sqrt{6}}{3}$, where $H_{63}$ is shown in Figure \ref{Fig.20.}. Thus $l_2+l_3+l_4+\cdots+l_t=0$. If $l_1\geq6$, then Lemmas \ref{le:3} and \ref{le:2} imply that
$\lambda_{2}(C_4(S_2,S_2,T_3',S_1))\leq\lambda_{2}(H_{81})\doteq0.18193<1-\frac{\sqrt{6}}{3}$,
where $H_{81}$ is shown in Figure \ref{Fig.22.}. Thus $l_1\leq5$, and $15\leq l_0+2l_1\leq25$.
\\{\bf Subcase 2.1.1.1.} $l_1=5$, then $5\leq l_0\leq 15$.
\\If $l_0=5$, then Lemma \ref{le:2} implies  that $\lambda_{2}(C_4(S_2,S_2,T_3',S_1))=\lambda_{2}(H_{82})=1-\frac{\sqrt{6}}{3}$, where $H_{82}$ is shown in Figure \ref{Fig.22.}. Otherwise, if $15\geq l_0\geq6$, then Lemma \ref{le:2} implies  that $\lambda_{2}(C_4(S_2,S_2,T_3',S_1))\leq\lambda_{2}(H_{83})\doteq0.18268<1-\frac{\sqrt{6}}{3}$, where $H_{83}$ is shown in Figure \ref{Fig.22.}.
\\{\bf Subcase 2.1.1.2.} $l_1=4$, then $7\leq l_0\leq 17$.
\\If $7\leq l_0\leq9$, then Lemma \ref{le:2} implies  that $\lambda_{2}(C_4(S_2,S_2,T_3',S_1))\geq\lambda_{2}(H_{84})=1-\frac{\sqrt{6}}{3}$, where $H_{84}$ is shown in Figure \ref{Fig.22.}. Otherwise, if $17\geq l_0\geq10$, then Lemma \ref{le:2} implies that $\lambda_{2}(C_4(S_2,S_2,T_3',$ $S_1))\leq\lambda_{2}(H_{85})\doteq0.18266<1-\frac{\sqrt{6}}{3}$, where $H_{85}$ is shown in Figure \ref{Fig.22.}.
\begin{figure}[htbp]
  \centering
  \includegraphics[scale=0.55]{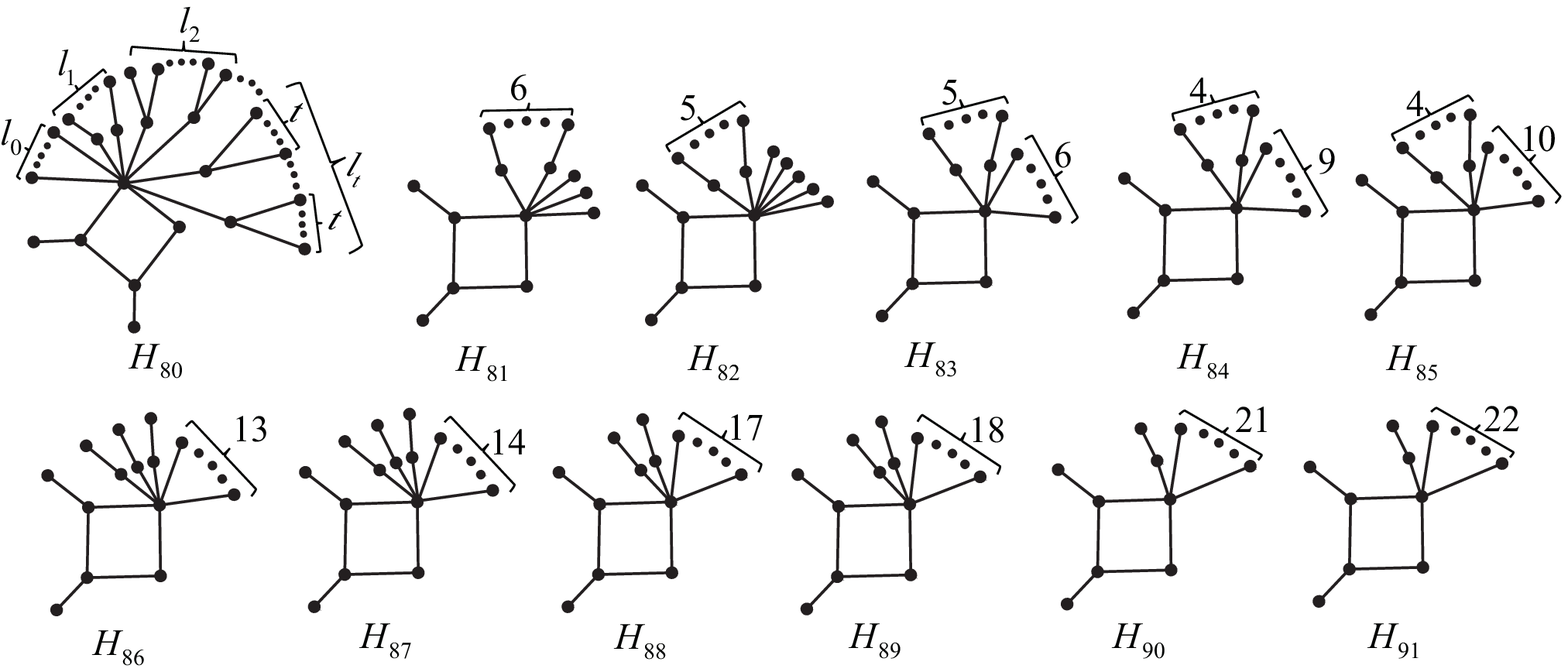}
  \caption{Unicyclic graphs $H_i$, $80\leq i\leq91$.}\label{Fig.22.}
\end{figure}
\\{\bf Subcase 2.1.1.3.} $l_1=3$, then $9\leq l_0\leq 19$.
\\If $9\leq l_0\leq13$, then Lemma \ref{le:2} implies  that $\lambda_{2}(C_4(S_2,S_2,T_3',S_1))\geq\lambda_{2}(H_{86})=1-\frac{\sqrt{6}}{3}$, where $H_{86}$ is shown in Figure \ref{Fig.22.}. Otherwise, if $19\geq l_0\geq14$, then Lemma \ref{le:2} implies that $\lambda_{2}(C_4(S_2,S_2,T_3',S_1))\leq\lambda_{2}(H_{87})\doteq0.18264<1-\frac{\sqrt{6}}{3}$, where $H_{87}$ is shown in Figure \ref{Fig.22.}.
\\{\bf Subcase 2.1.1.4.} $l_1=2$, then $11\leq l_0\leq 21$.
\\If $11\leq l_0\leq17$, then Lemma \ref{le:2} implies that $\lambda_{2}(C_4(S_2,S_2,T_3',S_1))\geq\lambda_{2}(H_{88})=1-\frac{\sqrt{6}}{3}$, where $H_{88}$ is shown in Figure \ref{Fig.22.}. Otherwise, if $21\geq l_0\geq18$, then Lemma \ref{le:2} implies that $\lambda_{2}(C_4(S_2,S_2,T_3',S_1))\leq\lambda_{2}(H_{89})\doteq0.18261<1-\frac{\sqrt{6}}{3}$, where $H_{89}$ is shown in Figure \ref{Fig.22.}.
\\{\bf Subcase 2.1.1.5.} $l_1=1$, then $13\leq l_0\leq 23$.
\\If $13\leq l_0\leq21$, then Lemma \ref{le:2} implies that $\lambda_{2}(C_4(S_2,S_2,T_3',S_1))\geq\lambda_{2}(H_{90})=1-\frac{\sqrt{6}}{3}$, where $H_{90}$ is shown in Figure \ref{Fig.22.}. Otherwise, if $23\geq l_0\geq22$, then Lemma \ref{le:2} implies that $\lambda_{2}(C_4(S_2,S_2,T_3',S_1))\leq\lambda_{2}(H_{91})\doteq0.18259<1-\frac{\sqrt{6}}{3}$, where $H_{91}$ is shown in Figure \ref{Fig.22.}.
\\{\bf Subcase 2.1.1.6.} $l_1=0$, then $15\leq l_0\leq 25$.
\\By Lemma \ref{le:2}, we have $\lambda_{2}(C_4(S_2,S_2,T_3',S_1))\geq\lambda_{2}(S_5(2,2,26,1))=1-\frac{\sqrt{6}}{3}$.
\\{\bf Subcase 2.1.2.} If $C_4(T_1,T_2,T_3,T_4)\cong C_4(T_2',T_3',T_1',S_1)$, that is $C_4(T_1,T_2,T_3,T_4)\cong C_4(S_2,T_3',$ $S_2,S_1)$.
Let $d$ be the length of the longest path from $v_3'$ to the pendent vertex in $T_3'$. In order to make
$\lambda_{2}(C_4(S_2,T_3',S_2,1))\geq1-\frac{\sqrt{6}}{3}$, then $d\leq2$. If $d\geq3$, then Lemma \ref{le:2} implies that
$\lambda_{2}(C_4(S_2,T_3',S_2,1))\leq\lambda_{2}(H_{61})\doteq0.16959<1-\frac{\sqrt{6}}{3}$, where $H_{61}$ is shown in Figure \ref{Fig.20.}.
Thus $C_4(2,T_3',2,1)$ has the form $H_{92}$ as shown in Figure \ref{Fig.23.}, where $l_i$ ($0\leq i\leq t$) are nonnegative integers
and $l_0+2l_1+3l_2+\cdots+(t+1)l_t+6\geq21$. If $l_3+l_4+\cdots+l_t\geq1$, then Lemma \ref{le:2} implies that
$\lambda_{2}(U)\leq\lambda_{2}(H_{93})\doteq0.15016<1-\frac{\sqrt{6}}{3}$, where $H_{93}$ is shown in Figure \ref{Fig.23.}. Thus $l_3+l_4+\cdots+l_t=0$. Therefore $C_4(S_2,T_3',S_2,1)$ is the unicyclic graph $H_{94}$ as shown in Figure \ref{Fig.23.}, here $l_i$ ($0\leq i\leq 2$) are nonnegative integers and $l_0+2l_1+3l_2+6=n\geq21$. Recall that $v_3'$ is a vertex degree $l_0+l_1+l_2+2$ in $H_{94}$. Note that the eigenvalues of
$\mathcal{L}_{v_3'}(H_{94})$ are $\underbrace{1-\frac{\sqrt{6}}{3},\cdots,1-\frac{\sqrt{6}}{3}}_{l_{2}+1}$, $\underbrace{1-\frac{\sqrt{2}}{2},\cdots,1-\frac{\sqrt{2}}{2}}_{l_{1}}$, $1-\frac{\sqrt{3}}{3}$, $\underbrace{1,\cdots,1}_{l_{0}+l_{2}+1}$, $1+\frac{\sqrt{3}}{3}$, $\underbrace{1+\frac{\sqrt{2}}{2},\cdots,1+\frac{\sqrt{2}}{2}}_{l_{1}}$, $\underbrace{1+\frac{\sqrt{6}}{3},\cdots,1+\frac{\sqrt{6}}{3}}_{l_{2}+1}$.
\\If $l_2\geq1$, then Lemma \ref{le:1} implies that $\lambda_{2}(H_{94})=1-\frac{\sqrt{6}}{3}$.
Analogously, when $l_2=0$, $\lambda_{2}(H_{94})\geq1-\frac{\sqrt{6}}{3}$.
\begin{figure}[htbp]
  \centering
  \includegraphics[scale=0.55]{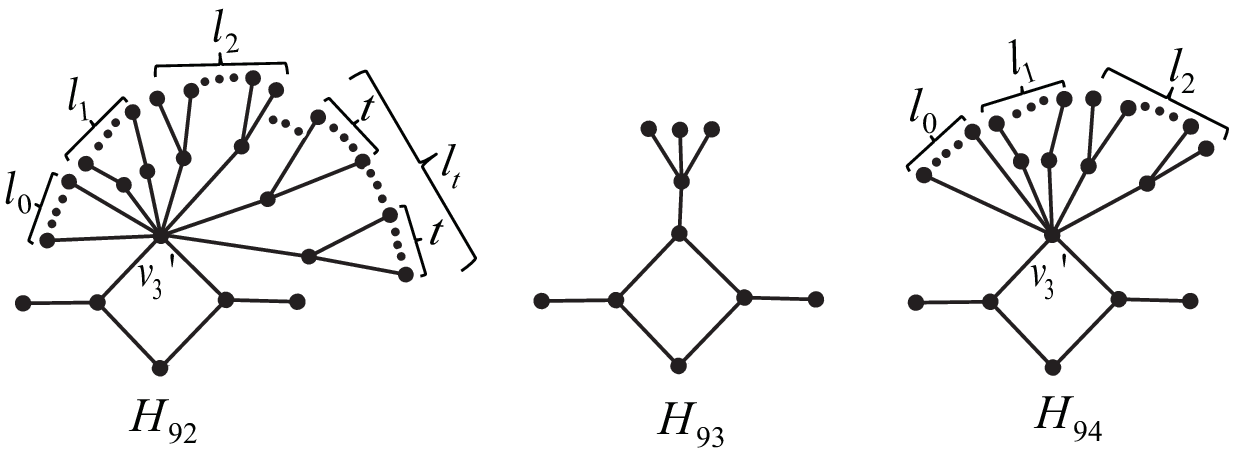}
  \caption{Unicyclic graphs $H_i$, $92\leq i\leq94$.}\label{Fig.23.}
\end{figure}
\\{\bf Subcase 2.2.} If $n_1'=2$, $n_2'=3$, then $n_3'\geq15$ since $n\geq21$.
\\{\bf Subcase 2.2.1.} $C_4(T_1,T_2,T_3,T_4)\cong C_4( T_2',T_3',S_2,S_1)$.
If $n_3'\geq26$, by Lemmas \ref{le:5} and \ref{le:2}, we have
$\lambda_{2}(C_4( T_2',T_3',S_2,S_1))\leq\lambda_{2}(S_4(3,n_3',2,1))\leq\lambda_{2}(S_4(3,26,2,1))\doteq0.18231<1-\frac{\sqrt{6}}{3}$.
Therefore, in order to make
$\lambda_{2}(C_4( T_2',T_3',S_2,S_1))\geq1-\frac{\sqrt{6}}{3}$, we have $15\leq n_3'\leq25$.
Let $d$ be the length of the longest path from $v_3'$ to the pendent vertex in $T_3'$. In order to make
$\lambda_{2}(C_4( T_2',T_3',S_2,S_1))\geq1-\frac{\sqrt{6}}{3}$, then $d\leq2$. If $d\geq3$, then Lemma \ref{le:2} implies that
$\lambda_{2}(C_4( T_2',T_3',S_2,S_1))\leq\lambda_{2}(H_{61})\doteq0.16959<1-\frac{\sqrt{6}}{3}$, where $H_{61}$ is shown in Figure \ref{Fig.20.}. If $T_2'$ is a path of length 2, then Lemmas \ref{le:3} and \ref{le:2} imply that
$\lambda_{2}(C_4(T_2',T_3',S_2,S_1))\leq\lambda_{2}(H_{95})\doteq0.16560<1-\frac{\sqrt{6}}{3}$, where $H_{95}$ is shown in Figure \ref{Fig.24.}. Thus $C_4(T_2',T_3',S_2,S_1)$ is the unicyclic graph $H_{96}$ as shown in Figure \ref{Fig.24.}, where $l_i$ ($0\leq i\leq t$) are nonnegative integers
and $21\leq l_0+2l_1+3l_2+\cdots+(t+1)l_t+7=n\leq31$. If $l_2+l_3+l_4+\cdots+l_t\geq1$, then Lemma \ref{le:2} implies that
$\lambda_{2}(C_4( T_2',T_3',S_2,S_1))\leq\lambda_{2}(H_{97})\doteq0.17710<1-\frac{\sqrt{6}}{3}$, where $H_{97}$ is shown in Figure \ref{Fig.24.}. Thus $l_2+l_3+l_4+\cdots+l_t=0$. If $l_1\geq6$, then Lemmas \ref{le:3} and \ref{le:2} imply that
$\lambda_{2}(C_4( T_2',T_3',S_2,S_1))\leq\lambda_{2}(H_{98})\doteq0.18156<1-\frac{\sqrt{6}}{3}$,
where $H_{98}$ is shown in Figure \ref{Fig.24.}. Thus $l_1\leq5$, and $14\leq l_0+2l_1\leq24$.
\begin{figure}[htbp]
  \centering
  \includegraphics[scale=0.55]{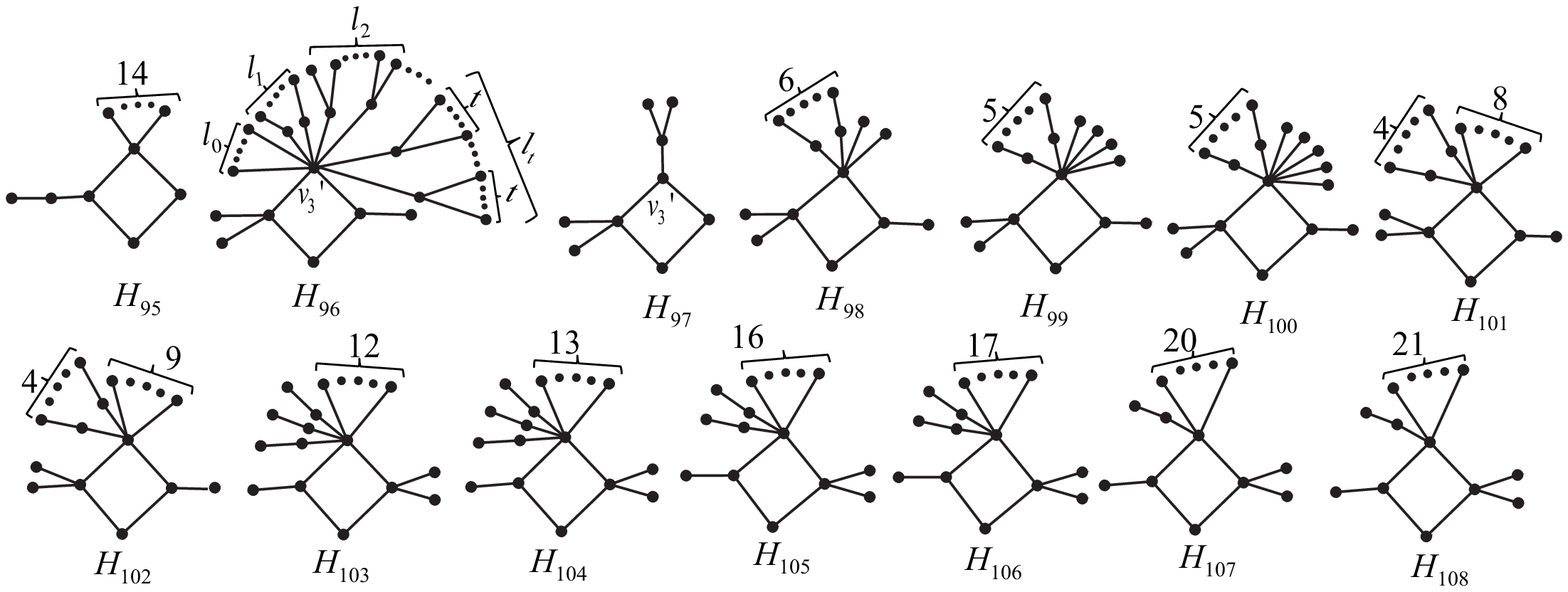}
  \caption{Unicyclic graphs $H_i$, $95\leq i\leq108$.}\label{Fig.24.}
\end{figure}
\\{\bf Subcase 2.2.1.1.} $l_1=5$, then $4\leq l_0\leq14$.
\\If $l_0=4$, then Lemma \ref{le:2} implies that $\lambda_{2}(C_4(T_2',T_3',S_2,S_1))=\lambda_{2}(H_{99})=1-\frac{\sqrt{6}}{3}$, where $H_{99}$ is shown in Figure \ref{Fig.24.}. Otherwise, if $5\leq l_0\leq14$, then Lemma \ref{le:2} implies that
$\lambda_{2}(C_4(T_2',T_3',S_2,S_1))\leq\lambda_{2}(H_{100})\doteq0.18248<1-\frac{\sqrt{6}}{3}$, where $H_{100}$ is shown in Figure \ref{Fig.24.}.
\\{\bf Subcase 2.2.1.2.} $l_1=4$, then $6\leq l_0\leq16$.
\\If $6\leq l_0\leq8$, then Lemma \ref{le:2} implies that $\lambda_{2}(C_4(T_2',T_3',S_2,S_1))\geq\lambda_{2}(H_{101})=1-\frac{\sqrt{6}}{3}$, where $H_{101}$ is shown in Figure \ref{Fig.24.}. Otherwise, if $9\leq l_0\leq16$, then Lemma \ref{le:2} implies that $\lambda_{2}(C_4(2,2,T_3',1))\leq\lambda_{2}(H_{102})\doteq0.18245<1-\frac{\sqrt{6}}{3}$, where $H_{102}$ is shown in Figure \ref{Fig.24.}.
\\{\bf Subcase 2.2.1.3.} $l_1=3$, then $8\leq l_0\leq18$.
\\If $8\leq l_0\leq12$, then Lemma \ref{le:2} implies that $\lambda_{2}(C_4(T_2',T_3',S_2,S_1))\geq\lambda_{2}(H_{103})=1-\frac{\sqrt{6}}{3}$, where $H_{103}$ is shown in Figure \ref{Fig.24.}. Otherwise, if $13\leq l_0\leq18$, then Lemma \ref{le:2} implies that $\lambda_{2}(C_4(T_2',T_3',S_2,S_1))\leq\lambda_{2}(H_{104})\doteq0.18242<1-\frac{\sqrt{6}}{3}$, where $H_{104}$ is shown in Figure \ref{Fig.24.}.
\\{\bf Subcase 2.2.1.4.} $l_1=2$, then $10\leq l_0\leq20$.
\\If $10\leq l_0\leq16$, then Lemma \ref{le:2} implies that $\lambda_{2}(C_4(T_2',T_3',S_2,S_1))\geq\lambda_{2}(H_{105})=1-\frac{\sqrt{6}}{3}$, where $H_{105}$ is shown in Figure \ref{Fig.24.}. Otherwise, if $17\leq l_0\geq20$, then Lemma \ref{le:2} implies that $\lambda_{2}(C_4(T_2',T_3',S_2,S_1))\leq\lambda_{2}(H_{106})\doteq0.18238<1-\frac{\sqrt{6}}{3}$, where $H_{106}$ is shown in Figure \ref{Fig.24.}.
\\{\bf Subcase 2.2.1.5.} $l_1=1$, then $12\leq l_0\leq22$.
\\If $12\leq l_0\leq20$, then Lemma \ref{le:2} implies that $\lambda_{2}(C_4(T_2',T_3',S_2,S_1))\geq\lambda_{2}(H_{107})=1-\frac{\sqrt{6}}{3}$, where $H_{107}$ is shown in Figure \ref{Fig.24.}. Otherwise, if $21\leq l_0\leq22$, then Lemma \ref{le:2} implies that $\lambda_{2}(C_4(T_2',T_3',S_2,S_1))\leq\lambda_{2}(H_{108})\doteq0.18234<1-\frac{\sqrt{6}}{3}$, where $H_{108}$ is shown in Figure \ref{Fig.24.}.
\\{\bf Subcase 2.2.1.6} $l_1=0$, then $14\leq l_0\leq24$.
\\By Lemma \ref{le:2}, we have $\lambda_{2}(C_4(T_2',T_3',S_2,S_1))\geq\lambda_{2}(S_5(3,25,2,1))=1-\frac{\sqrt{6}}{3}$.
\\{\bf Subcase 2.2.2.}  $C_4(T_1,T_2,T_3,T_4)\ncong C_4( T_2',T_3',S_2,S_1)$.
Hence Lemmas \ref{le:5} and \ref{le:2} imply that $\lambda_{2}(C_4(T_1,T_2,T_3,T_4))\leq\max\{\lambda_{2}(S_4(3,15,1,2)), \lambda_{2}(S_4(3,2,15,1))\}\doteq0.18032<1-\frac{\sqrt{6}}{3}$.
\\{\bf Subcase 2.3.} If $n_1'=2$, $n_2'=4$, then $n_3'\geq14$ since $n\geq21$. Hence Lemmas \ref{le:5} and \ref{le:2} imply that $\lambda_{2}(C_4(T_1,T_2,T_3,T_4))\leq\max\{\lambda_{2}(S_4(4,14,2,1)),\lambda_{2}(S_4(4,14,1,2)),\lambda_{2}(S_4(2,14,1,4))\}
\doteq0.17969<1-\frac{\sqrt{6}}{3}$.
\\{\bf Subcase 2.4.} If $n_1'=2$, $5\leq n_2'\leq8$, then $n_3'\geq10$ since $n\geq21$. Hence Lemmas \ref{le:5} and \ref{le:2} imply that $\lambda_{2}(C_4(T_1,T_2,T_3,T_4))\leq\max\{\lambda_{2}(S_4(5,10,2,1)),\lambda_{2}(S_4(5,10,1,2)),\lambda_{2}(S_4(2,10,1,$ $5))\}
$ $\doteq0.17683<1-\frac{\sqrt{6}}{3}$.
\\{\bf Subcase 2.5.} If $n_1'=2$, $n_3'\geq n_2'\geq9$. Hence Lemmas \ref{le:5} and \ref{le:2} imply that $\lambda_{2}(C_4(T_1,T_2,T_3,$ $T_4))\leq\max\{\lambda_{2}(S_4(9,9,2,1)),\lambda_{2}(S_4(2,9,1,9))\}
\doteq0.13556<1-\frac{\sqrt{6}}{3}$.
\\{\bf Subcase 2.6.} $n_3'\geq n_2'\geq n_1'\geq3$
\\{\bf Subcase 2.6.1.} $C_4(T_1,T_2,T_3,T_4)\cong C_4(T_1',T_3',T_2',S_1)$.
\\{\bf Subcase 2.6.1.1.} If $n_1'=n_2'=3$, then $ n_3'\geq14$ since $n\geq21$. If $n_3'\geq18$, by Lemmas \ref{le:5} and \ref{le:2}, we have
$\lambda_{2}(C_4(T_1',T_3',T_2',S_1))\leq\lambda_{2}(S_4(3,n_3',3,1))\leq\lambda_{2}(S_4(3,18,3,1))\doteq0.18082<1-\frac{\sqrt{6}}{3}$.
Therefore, in order to make
$\lambda_{2}(C_4( T_1',T_3',T_2',S_1))\geq1-\frac{\sqrt{6}}{3}$, we have $14\leq n_3'\leq17$.
Let $d$ be the length of the longest path from $v_3'$ to the pendent vertex in $T_3'$. In order to make
$\lambda_{2}(C_4( T_2',T_3',S_2,S_1))\geq1-\frac{\sqrt{6}}{3}$, then $d\leq2$. If $d\geq3$, then Lemma \ref{le:2} implies that
$\lambda_{2}(C_4( T_2',T_3',S_2,S_1))\leq\lambda_{2}(H_{61})\doteq0.16959<1-\frac{\sqrt{6}}{3}$, where $H_{61}$ is shown in Figure \ref{Fig.20.}. If $T_2'$ is a path of length 2, then Lemmas \ref{le:3} and \ref{le:2} imply that
$\lambda_{2}(C_4( T_1',T_3',T_2',S_1))\leq\lambda_{2}(A_{1})\doteq0.16534<1-\frac{\sqrt{6}}{3}$, where $A_{1}$ is shown in Figure \ref{Fig.25.}. Thus $C_4( T_1',T_3',T_2',S_1)$ is the unicyclic graph $A_{2}$ as shown in Figure \ref{Fig.25.}, where $l_i$ ($0\leq i\leq t$) are nonnegative integers
and $21\leq l_0+2l_1+3l_2+\cdots+(t+1)l_t+8=n\leq24$. If $l_2+l_3+l_4+\cdots+l_t\geq1$, then Lemma \ref{le:2} implies that
$\lambda_{2}(C_4( T_1',T_3',T_2',S_1))\leq\lambda_{2}(H_{97})\doteq0.17710<1-\frac{\sqrt{6}}{3}$, where $H_{97}$ is shown in Figure \ref{Fig.24.}. Thus $l_2+l_3+l_4+\cdots+l_t=0$. If $l_1\geq2$, then Lemmas \ref{le:3} and \ref{le:2} imply that
$\lambda_{2}(C_4( T_1',T_3',T_2',S_1))\leq\lambda_{2}(A_{3})\doteq0.18117<1-\frac{\sqrt{6}}{3}$,
where $A_{3}$ is shown in Figure \ref{Fig.25.}. Thus $l_1\leq1$, and $13\leq l_0+2l_1\leq16$.
\begin{figure}[htbp]
  \centering
  \includegraphics[scale=0.55]{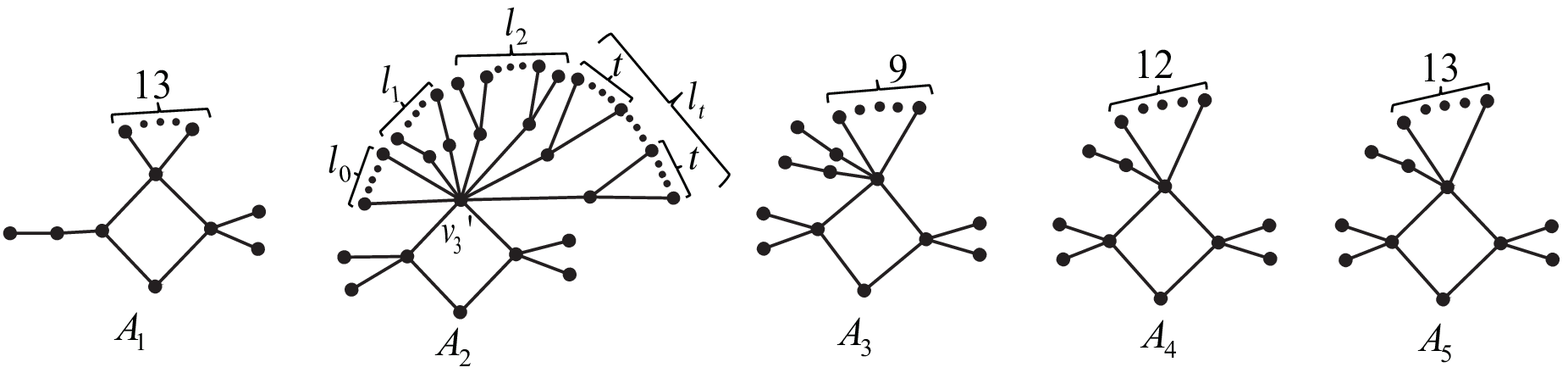}
  \caption{Unicyclic graphs $A_i$, $1\leq i\leq5$.}\label{Fig.25.}
\end{figure}
\\{\bf Subcase 2.6.1.1.1.} $l_1=1$, then $11\leq l_0\leq14$
\\If $11\leq l_0\leq12$, then Lemma \ref{le:2} implies that $\lambda_{2}(C_4( T_1',T_3',T_2',S_1))\geq\lambda_{2}(A_{4})=1-\frac{\sqrt{6}}{3}$, where $A_{4}$ is shown in Figure \ref{Fig.25.}. Otherwise, if $13\leq l_0\leq14$, then Lemma \ref{le:2} implies that
$\lambda_{2}(C_4(T_1',T_3',T_2',S_1))\leq\lambda_{2}(A_{5})\doteq0.18100<1-\frac{\sqrt{6}}{3}$, where $A_{5}$ is shown in Figure \ref{Fig.25.}.
\\{\bf Subcase 2.6.1.1.2.} $l_1=0$, then $13\leq l_0\leq16$
\\By Lemma \ref{le:2}, we have $\lambda_{2}(C_4( T_1',T_3',T_2',S_1)\geq\lambda_{2}(S_5(3,17,3,1))=1-\frac{\sqrt{6}}{3}$.
\\{\bf Subcase 2.6.1.2.} If $n_1'=3$, $n_2'=4, $then $ n_3'\geq13$ since $n\geq21$. Hence Lemmas \ref{le:5} and \ref{le:2} imply that $\lambda_{2}(C_4(T_1',T_3',T_2',S_1))\leq\lambda_{2}(S_4(2,13,4,1))\doteq0.17964<1-\frac{\sqrt{6}}{3}$.
\\{\bf Subcase 2.6.1.3.}  If $n_1'=3$, $5\leq n_2'\leq7$, then $ n_3'\geq10$ since $n\geq21$. Hence Lemmas \ref{le:5} and \ref{le:2} imply that $\lambda_{2}(C_4(T_1',T_3',T_2',S_1))\leq\lambda_{2}(S_4(2,10,5,1))\doteq0.17671<1-\frac{\sqrt{6}}{3}$.
\\{\bf Subcase 2.6.1.4.} If $n_1'=3$, $n_3'\geq n_2'\geq8$. Hence Lemmas \ref{le:5} and \ref{le:2} imply that $\lambda_{2}(C_4(T_1',T_3',T_2',$ $S_1))\leq\lambda_{2}(S_4(2,8,8,1))\doteq0.14689<1-\frac{\sqrt{6}}{3}$.
\\{\bf Subcase 2.6.1.5.} If $n_1'=n_2'=4$, $n_3'\geq12$. Hence Lemmas \ref{le:5} and \ref{le:2} imply that $\lambda_{2}(C_4(T_1',T_3',T_2',$ $S_1))\leq\lambda_{2}(S_4(4,12,4,1))\doteq0.17725<1-\frac{\sqrt{6}}{3}$.
\\{\bf Subcase 2.6.1.6.} If $n_1'=4$, $5\leq n_2'\leq6$, then $ n_3'\geq10$ since $n\geq21$. Hence Lemmas \ref{le:5} and \ref{le:2} imply that $\lambda_{2}(C_4(T_1',T_3',T_2',S_1))\leq\lambda_{2}(S_4(4,10,5,1))\doteq0.17642<1-\frac{\sqrt{6}}{3}$.
\\{\bf Subcase 2.6.1.7.} If $n_1'=4$, $n_3'\geq n_2'\geq7$. Hence Lemmas \ref{le:5} and \ref{le:2} imply that $\lambda_{2}(C_4(T_1',T_3',T_2',$ $S_1))\leq\lambda_{2}(S_4(4,7,7,1))\doteq0.16147<1-\frac{\sqrt{6}}{3}$.
\\{\bf Subcase 2.6.1.8.} If $n_1'=n_2'=5$, $n_3'\geq10$. Hence Lemmas \ref{le:5} and \ref{le:2} imply that $\lambda_{2}(C_4(T_1',T_3',T_2',$ $S_1))\leq\lambda_{2}(S_4(5,10,5,1))\doteq0.17428<1-\frac{\sqrt{6}}{3}$.
\\{\bf Subcase 2.6.1.9.} If there are at most one $n_1'=5$, $n_3'\geq n_2'\geq6$. Hence Lemmas \ref{le:5} and \ref{le:2} imply that $\lambda_{2}(C_4(T_1',T_3',T_2',S_1))\leq\lambda_{2}(S_4(5,6,6,1))\doteq0.16768<1-\frac{\sqrt{6}}{3}$.
\\{\bf Subcase 2.6.2.}  $C_4(T_1,T_2,T_3,T_4)\ncong C_4(T_1',T_3',T_2',S_1)$.
\\{\bf Subcase 2.6.2.1.} If there are at least one $n_1'$ such that $n_1'=3$, then $n_3'\geq9$ since $n\geq21$.
Hence Lemmas \ref{le:5} and \ref{le:2} imply that $\lambda_{2}(C_4(T_1,T_2,T_3,T_4))\leq\lambda_{2}(S_4(3,3,9,1))\doteq0.18043<1-\frac{\sqrt{6}}{3}$.
\\{\bf Subcase 2.6.2.2.} If $n_3'\geq n_2'\geq n_1'\geq4$, then $n_3'\geq7$ since $n\geq21$.
Hence Lemmas \ref{le:5} and \ref{le:2} imply that $\lambda_{2}(C_4(T_1,T_2,T_3,T_4))\leq\lambda_{2}(S_4(4,4,7,1))\doteq0.17279<1-\frac{\sqrt{6}}{3}$.
\\{\bf Case 3.} $|N|=2$. We assume that $n_1',n_2'\in\{n_1,n_2,n_3,n_4\}$, and $2\leq n_1'\leq n_2'$.
\\{\bf Subcase 3.1}  $C_4(T_1,T_2,T_3,T_4)\ncong C_4(T_1',T_2',S_1,S_1)$.
\\{\bf Subcase 3.1.1.} If $n_1'=2$, then $n_2'\geq17$ since $n\geq21$. Since $C_4(T_1,T_2,T_3,T_4)\ncong C_4(S_2,T_2',S_1,S_1)$,
then $C_4(T_1,T_2,T_3,T_4)\cong C_4(S_1,T_2',S_1,S_2)$. Let $d$ be the length of the longest
path from $v_2'$ to the pendent vertex in $T_2'$. In order to make
$\lambda_{2}(C_4(S_2,T_2',S_1,S_1))\geq1-\frac{\sqrt{6}}{3}$, then $d\leq2$. If $d\geq3$, then Lemma \ref{le:2} implies that
$\lambda_{2}(C_4(S_2,T_3',S_2,1))\leq\lambda_{2}(H_{61})\doteq0.16959<1-\frac{\sqrt{6}}{3}$, where $H_{61}$ is shown in Figure \ref{Fig.20.}.
Thus $C_4(2,T_3',2,1)$ is the unicyclic graph $A_{6}$ as shown in Figure \ref{Fig.26.}, where $l_i$ ($0\leq i\leq t$) are nonnegative integers
and $l_0+2l_1+3l_2+\cdots+(t+1)l_t+5=n\geq21$. If $l_3+l_4+\cdots+l_t\geq1$, then Lemma \ref{le:2} implies that
$\lambda_{2}(U)\leq\lambda_{2}(A_{7})\doteq0.15317<1-\frac{\sqrt{6}}{3}$, where $A_{7}$ is shown in Figure \ref{Fig.26.}. Thus $l_3+l_4+\cdots+l_t=0$. Therefore $C_4(S_2,T_3',S_2,1)$ is the unicyclic graph $A_{8}$ as shown in Figure \ref{Fig.26.}, here $l_i$ ($0\leq i\leq 2$) are nonnegative integers and $l_0+2l_1+3l_2+6=n\geq21$. Recall that $v_2'$ is a vertex degree $l_0+l_1+l_2+2$ in $A_{8}$. Note that the eigenvalues of
$\mathcal{L}_{v_2'}(A_{8})$ are $\underbrace{1-\frac{\sqrt{6}}{3},\cdots,1-\frac{\sqrt{6}}{3}}_{l_{2}+1}$, $\underbrace{1-\frac{\sqrt{2}}{2},\cdots,1-\frac{\sqrt{2}}{2}}_{l_{1}}$, $\underbrace{1,\cdots,1}_{l_{0}+l_{2}+2}$, $\underbrace{1+\frac{\sqrt{2}}{2},\cdots,1+\frac{\sqrt{2}}{2}}_{l_{1}}$, $\underbrace{1+\frac{\sqrt{6}}{3},\cdots,1+\frac{\sqrt{6}}{3}}_{l_{2}+1}$.
\\If $l_2\geq1$, then Lemma \ref{le:1} implies that $\lambda_{2}(A_{8})=1-\frac{\sqrt{6}}{3}$.
Analogously, when $l_2=0$, $\lambda_{2}(A_{8})$ $\geq1-\frac{\sqrt{6}}{3}$.
\begin{figure}[htbp]
  \centering
  \includegraphics[scale=0.55]{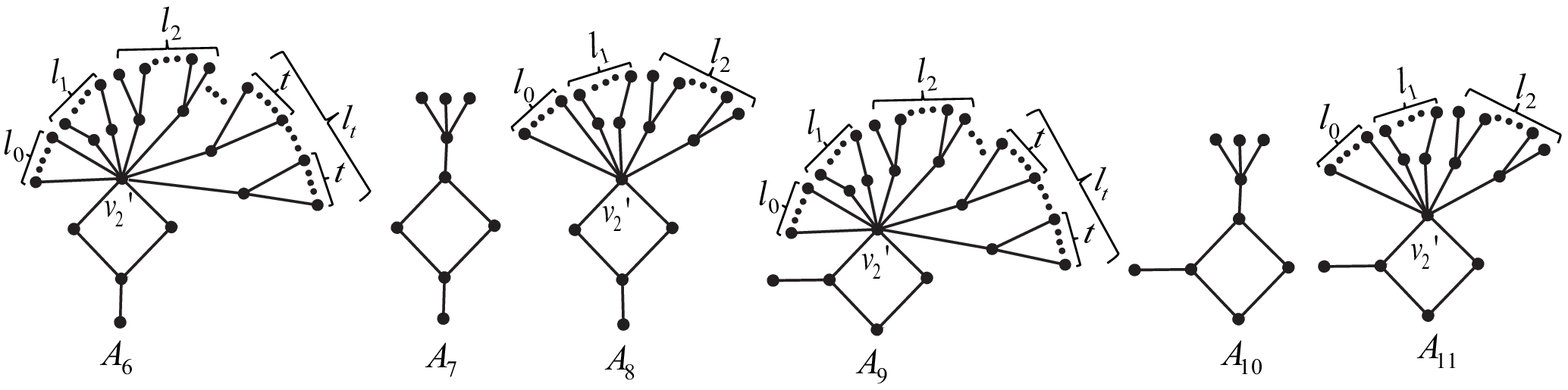}
  \caption{Unicyclic graphs $A_i$, $6\leq i\leq11$.}\label{Fig.26.}
\end{figure}
\\{\bf Subcase 3.1.2.} If $3\leq n_1'\leq7$, then $n_2'\geq12$ since $n\geq21$. Hence Lemmas \ref{le:5} and \ref{le:2} imply that $\lambda_{2}(C_4(T_1,T_2,T_3,T_4))\leq\lambda_{2}(S_4(1,12,1,3))\doteq0.17959<1-\frac{\sqrt{6}}{3}$.
\\{\bf Subcase 3.1.3.} If $n_2'\geq n_1'\geq8$, then Lemmas \ref{le:5} and \ref{le:2} imply that $\lambda_{2}(C_4(T_1,T_2,T_3,T_4))\leq\lambda_{2}(S_4(8,1,8,1))\doteq0.11808<1-\frac{\sqrt{6}}{3}$.
\\{\bf Subcase 3.2.} $C_4(T_1,T_2,T_3,T_4)\cong C_4(T_1',T_2',S_1,S_1)$.
\\{\bf Subcase 3.2.1.} If $n_1'=2$, then $n_2'\geq17$ since $n\geq21$. Since $C_4(T_1,T_2,T_3,T_4)\cong C_4(S_2,T_2',S_1,S_1)$,
let $d$ be the length of the longest
path from $v_2'$ to the pendent vertex in $T_2'$. In order to make
$\lambda_{2}(C_4(S_2,T_2',S_1,S_1))\geq1-\frac{\sqrt{6}}{3}$, then $d\leq2$. If $d\geq3$, then Lemma \ref{le:2} implies that
$\lambda_{2}(C_4(S_2,T_3',S_2,1))\leq\lambda_{2}(H_{61})\doteq0.16959<1-\frac{\sqrt{6}}{3}$, where $H_{61}$ is shown in Figure \ref{Fig.20.}.
Thus $C_4(2,T_3',2,1)$ is the unicyclic graph $A_{9}$ as shown in Figure \ref{Fig.26.}, where $l_i$ ($0\leq i\leq t$) are nonnegative integers
and $l_0+2l_1+3l_2+\cdots+(t+1)l_t+5=n\geq21$. If $l_3+l_4+\cdots+l_t\geq1$, then Lemma \ref{le:2} implies that
$\lambda_{2}(U)\leq\lambda_{2}(A_{10})\doteq0.16150<1-\frac{\sqrt{6}}{3}$, where $A_{10}$ is shown in Figure \ref{Fig.26.}. Thus $l_3+l_4+\cdots+l_t=0$. Therefore $C_4(S_2,T_3',S_2,1)$ has the form $A_{11}$ as shown in Figure \ref{Fig.26.}, here $l_i$ ($0\leq i\leq 2$) are nonnegative integers and $l_0+2l_1+3l_2+6=n\geq21$. Recall that $v_2'$ is a vertex degree $l_0+l_1+l_2+2$ in $A_{11}$. Note that the eigenvalues of
$\mathcal{L}_{v_2'}(A_{11})$ are $\underbrace{1-\frac{\sqrt{6}}{3},\cdots,1-\frac{\sqrt{6}}{3}}_{l_{2}}$, 0.21619, $\underbrace{1-\frac{\sqrt{2}}{2},\cdots,1-\frac{\sqrt{2}}{2}}_{l_{1}}$, 0.63170, $\underbrace{1,\cdots,1}_{l_{0}+l_{2}}$, 1.36830, $\underbrace{1+\frac{\sqrt{2}}{2},\cdots,1+\frac{\sqrt{2}}{2}}_{l_{1}}$, 1.78381 $\underbrace{1+\frac{\sqrt{6}}{3},\cdots,1+\frac{\sqrt{6}}{3}}_{l_{2}}$.
\\If $l_2\geq2$, then Lemma \ref{le:1} implies that $\lambda_{2}(A_{11})=1-\frac{\sqrt{6}}{3}$.
Analogously, when $l_2=1$, $\lambda_{2}(A_{11})\geq1-\frac{\sqrt{6}}{3}$, and when $l_2=0$, $\lambda_{2}(A_{11})>1-\frac{\sqrt{6}}{3}$.
\\{\bf Subcase 3.2.2.} If $n_1'=3$, then $n_2'\geq16$ since $n\geq21$. Since $C_4(T_1,T_2,T_3,T_4)\cong C_4(T_1',T_2',S_1,S_1)$,
let $d$ be the length of the longest path from $v_2'$ to the pendent vertex in $T_2'$.
In order to make
$\lambda_{2}(C_4(T_1',T_2',S_1,S_1))\geq1-\frac{\sqrt{6}}{3}$, then $d\leq2$. If $d\geq3$, then Lemma \ref{le:2} implies that
$\lambda_{2}(C_4(T_1',T_2',S_1,S_1))\leq\lambda_{2}(H_{61})\doteq0.16959<1-\frac{\sqrt{6}}{3}$, where $H_{61}$ is shown in Figure \ref{Fig.20.}.
If $T_1'$ is a path of length 2, then Lemmas \ref{le:3} and \ref{le:2} imply that
$\lambda_{2}(C_4(T_1',T_2',S_1,S_1))\leq\lambda_{2}(H_{95})\doteq0.16560<1-\frac{\sqrt{6}}{3}$, where $H_{95}$ is shown in Figure \ref{Fig.24.}. Thus $C_4(T_1',T_2',S_1,S_1)$ is the unicyclic graph $A_{12}$ as shown in Figure \ref{Fig.27.}, where $l_i$ ($0\leq i\leq t$) are nonnegative integers and $l_0+2l_1+3l_2+\cdots+(t+1)l_t+6=n\geq21$. If $l_2+l_3+\cdots+l_t\geq1$, then Lemma \ref{le:2} implies that
$\lambda_{2}(C_4(T_1',T_2',S_1,S_1))\leq\lambda_{2}(H_{97})\doteq0.17710<1-\frac{\sqrt{6}}{3}$, where $H_{97}$ is shown in Figure \ref{Fig.24.}. Thus $l_2+l_4+\cdots+l_t=0$. Since $\lambda_{2}(S_4(3,42,1,1))=0.18310<1-\frac{\sqrt{6}}{3}$, If $l_0+2l_1\geq41$, then Lemmas \ref{le:5} and \ref{le:2}
imply that $\lambda_{2}(C_4(T_1',T_2',S_1,S_1))\leq\lambda_{2}(S_4(3,42,1,1,1))=0.18310<1-\frac{\sqrt{6}}{3}$. Thus $15\leq l_0+2l_1\leq40$. If $l_1\geq11$,
then Lemmas \ref{le:3} and \ref{le:2} imply that
$\lambda_{2}(C_4(T_1',T_2',S_1,S_1))\leq\lambda_{2}(A_{12}')\doteq0.18214<1-\frac{\sqrt{6}}{3}$, where $A_{12}'$ is a unicyclic graph which obtained form $A_{12}$ with $l_2=l_3=\cdots=l_t=0$ and $l_1=11$. Thus $l_1\leq10$.
\\{\bf Subcase 3.2.2.1.} $l_1=10$, then $0\leq l_0\leq20$.
\\ If $l_0=0$, then $\lambda_{2}(C_4(T_1',T_2',S_1,S_1))=\lambda_{2}(A_{13})=1-\frac{\sqrt{6}}{3}$, where $A_{13}$ is shown in Figure \ref{Fig.27.}. Otherwise, if $1\leq l_0\leq20$, then Lemma \ref{le:2} implies that $\lambda_{2}(C_4(T_1',T_2',S_1,S_1))\leq\lambda_{2}(A_{14})\doteq0.18314<1-\frac{\sqrt{6}}{3}$, where $A_{14}$ is shown in Figure \ref{Fig.27.}.
\\{\bf Subcase 3.2.2.2.} $l_1=9$, then $0\leq l_0\leq22$.
\\ If $0\leq l_0\leq4$, then Lemma \ref{le:2} implies that $\lambda_{2}(C_4(T_1',T_2',S_1,S_1))\geq\lambda_{2}(A_{15})=1-\frac{\sqrt{6}}{3}$, where $A_{15}$ is shown in Figure \ref{Fig.27.}. Otherwise, if $5\leq l_0\leq22$, then Lemma \ref{le:2} implies that $\lambda_{2}(C_4(T_1',T_2',S_1,S_1))\leq\lambda_{2}(A_{16})\doteq0.18314<1-\frac{\sqrt{6}}{3}$, where $A_{16}$ is shown in Figure \ref{Fig.27.}.
\\{\bf Subcase 3.2.2.3.} $l_1=8$, then $0\leq l_0\leq24$.
\\ If $0\leq l_0\leq8$, then Lemma \ref{le:2} implies that $\lambda_{2}(C_4(T_1',T_2',S_1,S_1))\geq\lambda_{2}(A_{17})=1-\frac{\sqrt{6}}{3}$, where $A_{17}$ is shown in Figure \ref{Fig.27.}. Otherwise, if $9\leq l_0\leq24$, then Lemma \ref{le:2} implies that $\lambda_{2}(C_4(T_1',T_2',S_1,S_1))\leq\lambda_{2}(A_{18})\doteq0.18313<1-\frac{\sqrt{6}}{3}$, where $A_{18}$ is shown in Figure \ref{Fig.27.}.
\\{\bf Subcase 3.2.2.4.} $l_1=7$, then $1\leq l_0\leq26$.
\\ If $1\leq l_0\leq12$, then Lemma \ref{le:2} implies that $\lambda_{2}(C_4(T_1',T_2',S_1,S_1))\geq\lambda_{2}(A_{19})=1-\frac{\sqrt{6}}{3}$, where $A_{19}$ is shown in Figure \ref{Fig.27.}. Otherwise, if $12\leq l_0\leq26$, then Lemma \ref{le:2} implies that $\lambda_{2}(C_4(T_1',T_2',S_1,S_1))\leq\lambda_{2}(A_{20})\doteq0.18313<1-\frac{\sqrt{6}}{3}$, where $A_{20}$ is shown in Figure \ref{Fig.27.}.
\\{\bf Subcase 3.2.2.5.} $l_1=6$, then $3\leq l_0\leq28$.
\\ If $3\leq l_0\leq16$, then Lemma \ref{le:2} implies that $\lambda_{2}(C_4(T_1',T_2',S_1,S_1))\geq\lambda_{2}(A_{21})=1-\frac{\sqrt{6}}{3}$, where $A_{21}$ is shown in Figure \ref{Fig.27.}. Otherwise, if $17\leq l_0\leq28$, then Lemma \ref{le:2} implies that $\lambda_{2}(C_4(T_1',T_2',S_1,S_1))\leq\lambda_{2}(A_{22})\doteq0.18312<1-\frac{\sqrt{6}}{3}$, where $A_{22}$ is shown in Figure \ref{Fig.27.}.
\\{\bf Subcase 3.2.2.6.} $l_1=5$, then $5\leq l_0\leq30$.
\\ If $5\leq l_0\leq20$, then Lemma \ref{le:2} implies that $\lambda_{2}(C_4(T_1',T_2',S_1,S_1))\geq\lambda_{2}(A_{23})=1-\frac{\sqrt{6}}{3}$, where $A_{23}$ is shown in Figure \ref{Fig.27.}. Otherwise, if $21\leq l_0\leq28$, then Lemma \ref{le:2} implies that $\lambda_{2}(C_4(T_1',T_2',S_1,S_1))\leq\lambda_{2}(A_{24})\doteq0.18312<1-\frac{\sqrt{6}}{3}$, where $A_{24}$ is shown in Figure \ref{Fig.27.}.
\begin{figure}[htbp]
  \centering
  \includegraphics[scale=0.55]{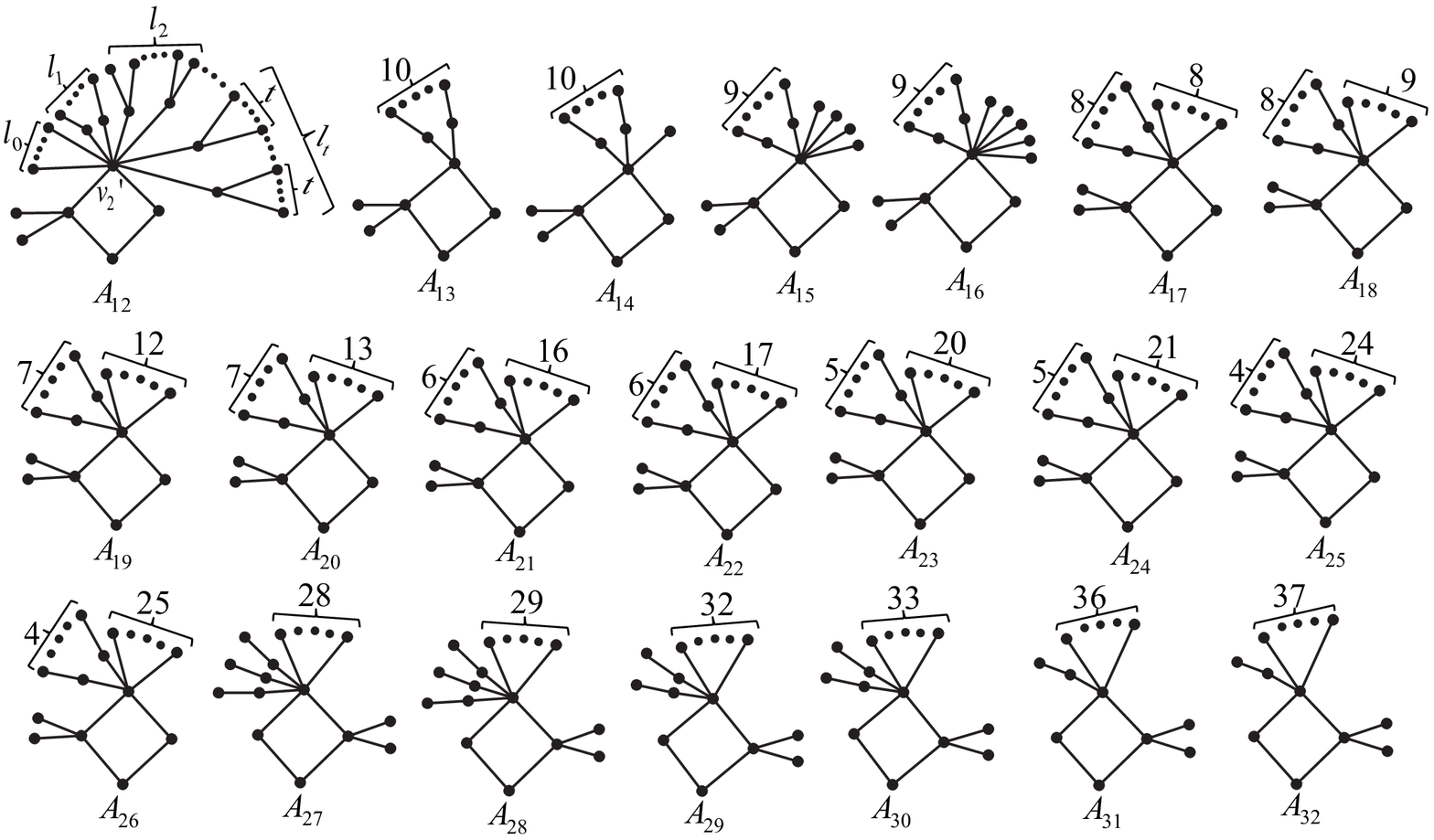}
  \caption{Unicyclic graphs $A_i$, $12\leq i\leq32$.}\label{Fig.27.}
\end{figure}
\\{\bf Subcase 3.2.2.7.} $l_1=4$, then $7\leq l_0\leq32$.
\\ If $7\leq l_0\leq24$, then Lemma \ref{le:2} implies that $\lambda_{2}(C_4(T_1',T_2',S_1,S_1))\geq\lambda_{2}(A_{25})=1-\frac{\sqrt{6}}{3}$, where $A_{25}$ is shown in Figure \ref{Fig.27.}. Otherwise, if $25\leq l_0\leq32$, then Lemma \ref{le:2} implies that $\lambda_{2}(C_4(T_1',T_2',S_1,S_1))\leq\lambda_{2}(A_{26})\doteq0.18312<1-\frac{\sqrt{6}}{3}$, where $A_{26}$ is shown in Figure \ref{Fig.27.}.
\\{\bf Subcase 3.2.2.8.} $l_1=3$, then $9\leq l_0\leq34$.
\\ If $9\leq l_0\leq28$, then Lemma \ref{le:2} implies that $\lambda_{2}(C_4(T_1',T_2',S_1,S_1))\geq\lambda_{2}(A_{27})=1-\frac{\sqrt{6}}{3}$, where $A_{27}$ is shown in Figure \ref{Fig.27.}. Otherwise, if $29\leq l_0\leq34$, then Lemma \ref{le:2} implies that $\lambda_{2}(C_4(T_1',T_2',S_1,S_1))\leq\lambda_{2}(A_{28})\doteq0.18311<1-\frac{\sqrt{6}}{3}$, where $A_{28}$ is shown in Figure \ref{Fig.27.}.
\\{\bf Subcase 3.2.2.9.} $l_1=2$, then $11\leq l_0\leq36$.
\\ If $11\leq l_0\leq32$, then Lemma \ref{le:2} implies that $\lambda_{2}(C_4(T_1',T_2',S_1,S_1))\geq\lambda_{2}(A_{29})=1-\frac{\sqrt{6}}{3}$, where $A_{29}$ is shown in Figure \ref{Fig.27.}. Otherwise, if $33\leq l_0\leq36$, then Lemma \ref{le:2} implies that $\lambda_{2}(C_4(T_1',T_2',S_1,S_1))\leq\lambda_{2}(A_{30})\doteq0.18311<1-\frac{\sqrt{6}}{3}$, where $A_{30}$ is shown in Figure \ref{Fig.27.}.
\\{\bf Subcase 3.2.2.10.} $l_1=1$, then $13\leq l_0\leq38$.
\\ If $13\leq l_0\leq36$, then Lemma \ref{le:2} implies that $\lambda_{2}(C_4(T_1',T_2',S_1,S_1))\geq\lambda_{2}(A_{31})=1-\frac{\sqrt{6}}{3}$, where $A_{31}$ is shown in Figure \ref{Fig.27.}. Otherwise, if $37\leq l_0\leq38$, then Lemma \ref{le:2} implies that $\lambda_{2}(C_4(T_1',T_2',S_1,S_1))\leq\lambda_{2}(A_{32})\doteq0.18310<1-\frac{\sqrt{6}}{3}$, where $A_{32}$ is shown in Figure \ref{Fig.27.}.
\\{\bf Subcase 3.2.2.11.} $l_1=0$, then $15\leq l_0\leq40$.
\\ By Lemma \ref{le:2}, we have $\lambda_{2}(T_1',T_2',S_1,S_1))\geq\lambda_{2}(S_4(3,41,1,1))=1-\frac{\sqrt{6}}{3}$.
\\{\bf Subcase 3.2.3.} If $4\leq n_1'\leq5$, then $n_2'\geq14$ since $n\geq21$. Since $C_4(T_1,T_2,T_3,T_4)\cong C_4(T_1',T_2',S_1,S_1)$,
then Lemmas \ref{le:5} and \ref{le:2} imply that $\lambda_{2}(C_4(T_1',T_2',S_1,S_1))\geq\lambda_{2}(S_4(4,14,$ $1,1)\doteq0.18206<1-\frac{\sqrt{6}}{3}$.
\\{\bf Subcase 3.2.4.} If $n_2'\geq n_1'\geq6$, then $n_2'\geq10$ since $n\geq21$. Since $C_4(T_1,T_2,T_3,T_4)\cong C_4(T_1',T_2',S_1,S_1)$,
then Lemmas \ref{le:5} and \ref{le:2} imply that $\lambda_{2}(C_4(T_1',T_2',S_1,S_1))\geq\lambda_{2}(S_4(6,10,$ $1,1)\doteq0.16024<1-\frac{\sqrt{6}}{3}$.
\begin{figure}[htbp]
  \centering
  \includegraphics[scale=0.55]{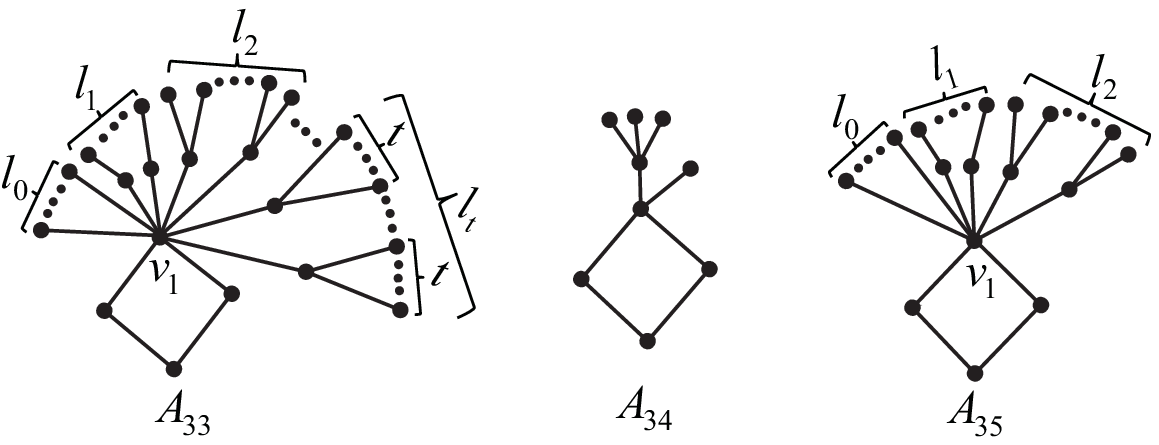}
  \caption{Unicyclic graphs $A_i$, $33\leq i\leq35$.}\label{Fig.28.}
\end{figure}
\\{\bf Case 4.} $|N|=1$. Without loss of generality, we assume that $n_1\geq2$. That is $C_4(T_1,T_2,T_3,$ $T_4)=C_4(T_1,S_1,S_1,S_1)$.
let $d$ be the length of the longest
path from $v_1$ to the pendent vertex in $T_1$. In order to make
$\lambda_{2}(C_4(T_1,S_1,S_1,S_1))\geq1-\frac{\sqrt{6}}{3}$, then $d\leq2$. If $d\geq3$, then Lemma \ref{le:2} implies that
$\lambda_{2}(C_4(T_1,S_1,S_1,S_1))\leq\lambda_{2}(H_{61})\doteq0.16959<1-\frac{\sqrt{6}}{3}$, where $H_{61}$ is shown in Figure \ref{Fig.20.}.
Thus $C_4(T_1,S_1,S_1,S_1)$ is the unicyclic graph $A_{33}$ as shown in Figure \ref{Fig.28.}, where $l_i$ ($0\leq i\leq t$) are nonnegative integers
and $l_0+2l_1+3l_2+\cdots+(t+1)l_t+4=n\geq21$. If $l_3+l_4+\cdots+l_t\geq1$, then Lemma \ref{le:2} implies that
$\lambda_{2}(C_4(T_1,S_1,S_1,S_1))\leq\lambda_{2}(A_{34})\doteq0.17948<1-\frac{\sqrt{6}}{3}$, where $A_{34}$ is shown in Figure \ref{Fig.28.}. Thus $l_3+l_4+\cdots+l_t=0$. Therefore $C_4(T_1,S_1,S_1,S_1)$ is the unicyclic graph $A_{35}$ as shown in Figure \ref{Fig.28.}, here $l_i$ ($0\leq i\leq 2$) are nonnegative integers and $l_0+2l_1+3l_2+4=n\geq21$. Recall that $v_1$ is a vertex degree $l_0+l_1+l_2+2$ in $A_{35}$. Note that the eigenvalues of
$\mathcal{L}_{v_2}(A_{35})$ are $\underbrace{1-\frac{\sqrt{6}}{3},\cdots,1-\frac{\sqrt{6}}{3}}_{l_{2}}$, 0.21619, $\underbrace{1-\frac{\sqrt{2}}{2},\cdots,1-\frac{\sqrt{2}}{2}}_{l_{1}+1}$, $\underbrace{1,\cdots,1}_{l_{0}+l_{2}+1}$, $\underbrace{1+\frac{\sqrt{2}}{2},\cdots,1+\frac{\sqrt{2}}{2}}_{l_{1}+1}$, $\underbrace{1+\frac{\sqrt{6}}{3},\cdots,1+\frac{\sqrt{6}}{3}}_{l_{2}}$.
\\If $l_2\geq2$, then Lemma \ref{le:1} implies that $\lambda_{2}(A_{35})=1-\frac{\sqrt{6}}{3}$.
Analogously, when $l_2=1$, $\lambda_{2}(A_{35})\geq1-\frac{\sqrt{6}}{3}$, and
when $l_2=0$, $\lambda_{2}(A_{35})>1-\frac{\sqrt{6}}{3}$.
\end{proof}

\section{ All unicyclic graphs of girth $3$ with $\lambda_{2}\geq1-\frac{\sqrt{6}}{3}$}

In this section, we determine all unicyclic graphs in $\mathcal{U}_n^3$ ($n\geq21$) with $\lambda_{2}\geq1-\frac{\sqrt{6}}{3}$.
Before introducing our results we recall some notation. Note that if $U\in\mathcal{U}_n^3$, then $U$ consists
of the cycle $C_3=v_1v_2v_3v_1$ and three trees $T_1$, $T_2$ and $T_3$ attached at the vertices $v_1$, $v_2$
and $v_3$, respectively.

\noindent\begin{theorem}\label{th:3} Suppose that $U\in\mathcal{U}_n^3$ with $n\geq21$.
Then $\lambda_{2}(U)\geq1-\frac{\sqrt{6}}{3}$ if and only if one of the following items holds:

(i) $U$ is the unicyclic graph $B_{4}$, where $B_{4}$ is shown in Figure \ref{Fig.29.}, here $l_i$ ($0\leq i\leq 2$) are nonnegative integers and $l_0+2l_1+3l_2+5=n\geq21$.

(ii) $U$ is the unicyclic graph $B_{7}$, where $B_{7}$ is shown in Figure \ref{Fig.29.}, here $l_i$ ($0\leq i\leq 2$) are nonnegative integers and $l_0+2l_1+3l_2+6=n\geq21$.

(iii) $U$ is the unicyclic graph $B_{10}$, where $B_{10}$ is shown in Figure \ref{Fig.29.}, where $l_i$ ($0\leq i\leq 1$) are nonnegative integers, $21\leq l_0+2l_1+7=n\leq28$, $0\leq l_1\leq3$ and one of the following items holds:

(1) $l_1=3$ and $8\leq l_0\leq9$.

(2) $l_1=2$ and $10\leq l_0\leq13$.

(3) $l_1=1$ and $12\leq l_0\leq17$.

(4) $l_1=0$ and $14\leq l_0\leq21$.

(iv) $U$ is the unicyclic graph $C_{2}$, where $C_{2}$ is shown in Figure \ref{Fig.30.}, where $l_i$ ($0\leq i\leq 1$) are nonnegative integers, $21\leq l_0+2l_1+7=n\leq37$, $0\leq l_1\leq7$ and one of the following items holds:

(1) $l_1=7$ and $0\leq l_0\leq2$.

(2) $l_1=6$ and $2\leq l_0\leq6$.

(3) $l_1=5$ and $4\leq l_0\leq10$.

(4) $l_1=4$ and $6\leq l_0\leq14$.

(5) $l_1=3$ and $8\leq l_0\leq18$.

(6) $l_1=2$ and $10\leq l_0\leq22$.

(7) $l_1=1$ and $12\leq l_0\leq26$.

(8)  $l_1=0$ and $14\leq l_0\leq30$.

(v) $U$ is isomorphic to $B_{35}$, where $B_{35}$ is shown in Figure \ref{Fig.32.}.

(vi) $U$ is isomorphic to $S_3(3,4,k)$ with $14\leq k\leq16$.

(vii) $U$ is isomorphic to $S_3(4,4,13)$.

(viii) $U$ is isomorphic to $S_3(7,7,7)$.

(ix) $U$ is the unicyclic graph $B_{42}$, where $B_{42}$ is shown in Figure \ref{Fig.32.}, here $l_i$ ($0\leq i\leq 2$) are nonnegative integers and $l_0+2l_1+3l_2+4=n\geq21$.

(x) $U$ is the unicyclic graph $B_{44}$, where $B_{44}$ is shown in Figure \ref{Fig.32.}, here $l_i$ ($0\leq i\leq 2$) are nonnegative integers and $l_0+2l_1+3l_2+5=n\geq21$.

(xi) $U$ is the unicyclic graph $B_{45}$, where $B_{45}$ is shown in Figure \ref{Fig.34.}, where $l_i$ ($0\leq i\leq 1$) are nonnegative integers,
 $21\leq l_0+2l_1+6=n\leq39$, $0\leq l_1\leq8$ and one of the following items holds:

(1) $l_1=8$ and $0\leq l_0\leq1$.

(2) $l_1=7$ and $1\leq l_0\leq5$.

(3) $l_1=6$ and $3\leq l_0\leq9$.

(4) $l_1=5$ and $5\leq l_0\leq13$.

(5) $l_1=4$ and $7\leq l_0\leq17$.

(6) $l_1=3$ and $9\leq l_0\leq21$.

(7) $l_1=2$ and $11\leq l_0\leq25$.

(8) $l_1=1$ and  $13\leq l_0\leq19$.

(9) $l_1=0$ and $15\leq l_0\leq33$.

(xii) $U$ is the unicyclic graph $B_{66}$, where $B_{66}$ is shown in Figure \ref{Fig.35.}, here $l_i$ ($0\leq i\leq 2$) are nonnegative integers and $l_0+2l_1+3l_2+3=n\geq21$.

Furthermore, the equality holds if and only if $U$ is isomorphic to $B_{4}$ with $l_2\geq2$, or $B_{7}$ with $l_2\geq1$,
or $B_{42}$ with $l_2\geq2$, $B_{44}$ with $l_2\geq2$, $B_{66}$ with $l_2\geq2$.
\end{theorem}

\begin{proof} We rewrite $U$ in the form $C_3(T_1,T_2,T_3)$, where $|V(T_i)|=n_i\geq1$ for $i = 1,2,3$, $n_3\geq n_2\geq n_1\geq1$, and $\sum\limits_{i=1}^3 n_i = n\geq21$. Let $N=\{i\mid n_i\geq2, i=1,2,3\}$. We consider the following cases.
\\{\bf Case 1.} $|N|=3$. Without lose of generality, we assume that $n_3\geq n_2\geq n_1\geq1$.
\\{\bf Subcase 1.1.} If $n_1=n_2=2$, then $n_3\geq17$ since $n\geq21$.
Let $d$ be the length of the longest path from $v_3$ to the pendent vertex in $T_3$. In order to make
$\lambda_{2}(C_3(T_1,T_2,T_3))\geq1-\frac{\sqrt{6}}{3}$, then $d\leq2$. If $d\geq3$, then Lemma \ref{le:2} implies that
$\lambda_{2}(C_3(T_1,T_2,T_3))\leq\lambda_{2}(B_{1})\doteq0.16989<1-\frac{\sqrt{6}}{3}$, where $B_{1}$ is shown in Figure \ref{Fig.29.}.
Thus $C_3(T_1,T_2,T_3)$ is the unicyclic graph $B_{2}$ as shown in Figure \ref{Fig.29.}, where $l_i$ ($0\leq i\leq t$) are nonnegative integers
and $l_0+2l_1+3l_2+\cdots+(t+1)l_t+5=n\geq21$. If $l_3+l_4+\cdots+l_t\geq1$, then Lemma \ref{le:2} implies that
$\lambda_{2}(C_3(T_1,T_2,T_3))\leq\lambda_{2}(B_{3})\doteq0.17985<1-\frac{\sqrt{6}}{3}$, where $B_{3}$ is shown in Figure \ref{Fig.29.}.
Thus $l_3+l_4+\cdots+l_t=0$. Therefore $C_3(T_1,T_2,T_3)$ has the form $B_{4}$ as shown in Figure \ref{Fig.29.}, here $l_i$ ($0\leq i\leq 2$) are nonnegative integers and $l_0+2l_1+3l_2+5=n\geq21$. Recall that $v_3$ is a vertex degree $l_0+l_1+l_2+2$ in $B_{4}$. Note that the eigenvalues of
$\mathcal{L}_{v_3}(B_{4})$ are $\underbrace{1-\frac{\sqrt{6}}{3},\cdots,1-\frac{\sqrt{6}}{3}}_{l_{2}}$, 0.23240, $\underbrace{1-\frac{\sqrt{2}}{2},\cdots,1-\frac{\sqrt{2}}{2}}_{l_{1}}$, 0.56574, $\underbrace{1,\cdots,1}_{l_{0}+l_{2}}$, 1.43426, $\underbrace{1+\frac{\sqrt{2}}{2},\cdots,1+\frac{\sqrt{2}}{2}}_{l_{1}}$, 1.76759,  $\underbrace{1+\frac{\sqrt{6}}{3},\cdots,1+\frac{\sqrt{6}}{3}}_{l_{2}}$.
\\If $l_2\geq2$, then Lemma \ref{le:1} implies that $\lambda_{2}(B_{4})=1-\frac{\sqrt{6}}{3}$.
Analogously, when $l_2=1$, $\lambda_{2}(B_{4})\geq1-\frac{\sqrt{6}}{3}$, and when $l_2=0$, $\lambda_{2}(B_{4})>1-\frac{\sqrt{6}}{3}$.
\begin{figure}[htbp]
  \centering
  \includegraphics[scale=0.55]{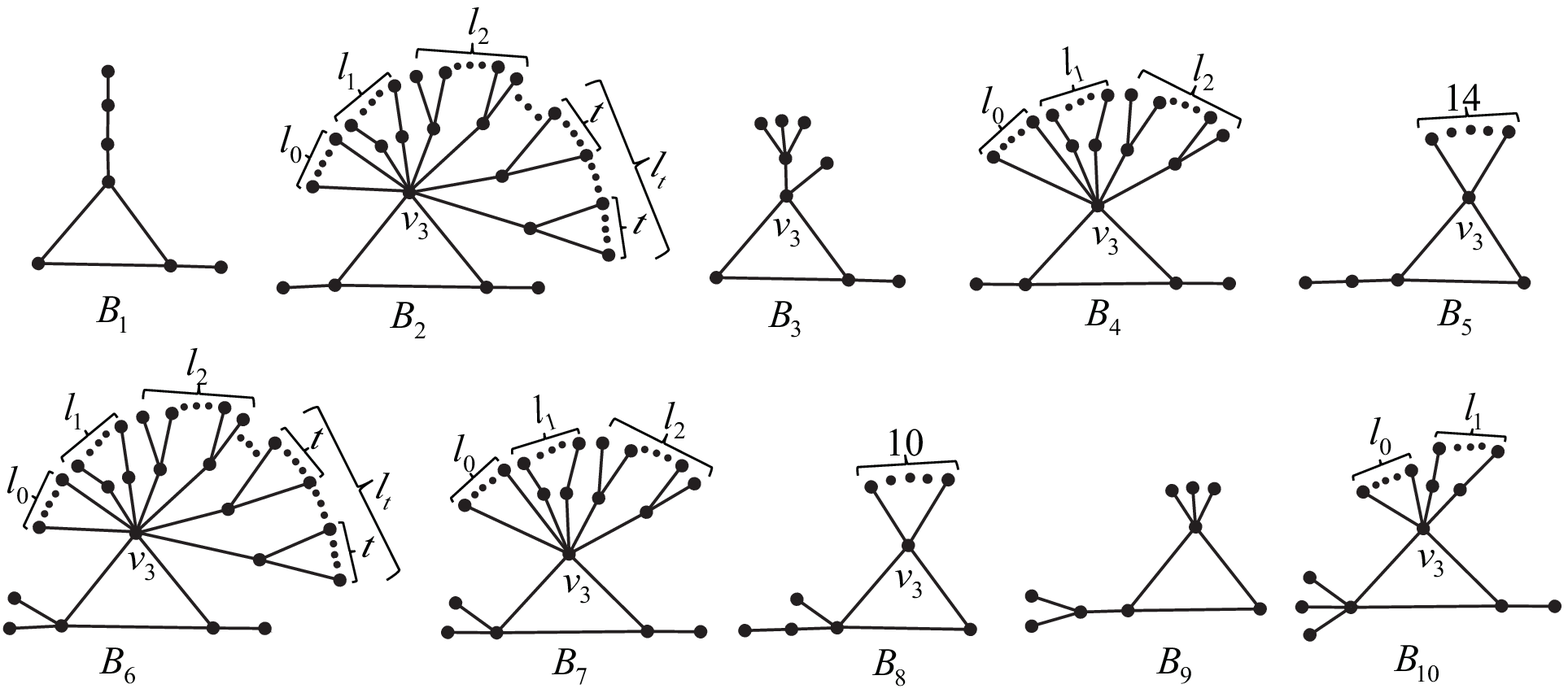}
  \caption{Unicyclic graphs $B_i$, $1\leq i\leq10$.}\label{Fig.29.}
\end{figure}
\\{\bf Subcase 1.2.} If $n_1=2$, $n_2=3$, then $n_3\geq16$ since $n\geq21$.
Let $d$ be the length of the longest path from $v_3$ to the pendent vertex in $T_3$. In order to make
$\lambda_{2}(C_3(T_1,T_2,T_3))\geq1-\frac{\sqrt{6}}{3}$, then $d\leq2$. If $d\geq3$, then Lemma \ref{le:2} implies that
$\lambda_{2}(C_3(T_1,T_2,T_3))\leq\lambda_{2}(B_{1})\doteq0.16989<1-\frac{\sqrt{6}}{3}$, where $B_{1}$ is shown in Figure \ref{Fig.29.}.
 If $T_2$ is a path of length 2, then Lemmas \ref{le:3} and \ref{le:2} imply that
$\lambda_{2}(C_3(T_1,T_2,T_3))\leq\lambda_{2}(B_{5})\doteq0.18144<1-\frac{\sqrt{6}}{3}$, where $B_{5}$ is shown in Figure \ref{Fig.29.}.
Thus $C_3(T_1,T_2,T_3)$ is the unicyclic graph $B_{6}$ as shown in Figure \ref{Fig.29.}, where $l_i$ ($0\leq i\leq t$) are nonnegative integers
and $l_0+2l_1+3l_2+\cdots+(t+1)l_t+6=n\geq21$. If $l_3+l_4+\cdots+l_t\geq1$, then Lemma \ref{le:2} implies that
$\lambda_{2}(C_3(T_1,T_2,T_3))\leq\lambda_{2}(B_{3})\doteq0.17985<1-\frac{\sqrt{6}}{3}$, where $B_{3}$ is shown in Figure \ref{Fig.29.}.
Thus $l_3+l_4+\cdots+l_t=0$. Therefore $C_3(T_1,T_2,T_3)$ has the form $B_{7}$ as shown in Figure \ref{Fig.29.}, here $l_i$ ($0\leq i\leq 2$) are nonnegative integers and $l_0+2l_1+3l_2+6=n\geq21$. Recall that $v_3$ is a vertex degree $l_0+l_1+l_2+2$ in $B_{7}$. Note that the eigenvalues of
$\mathcal{L}_{v_3}(B_{7})$ are $\underbrace{1-\frac{\sqrt{6}}{3},\cdots,1-\frac{\sqrt{6}}{3}}_{l_{2}+1}$, $\underbrace{1-\frac{\sqrt{2}}{2},\cdots,1-\frac{\sqrt{2}}{2}}_{l_{1}}$, $\frac{1}{2}$, $\underbrace{1,\cdots,1}_{l_{0}+l_{2}+1}$, $\frac{3}{2}$, $\underbrace{1+\frac{\sqrt{2}}{2},\cdots,1+\frac{\sqrt{2}}{2}}_{l_{1}}$, $\underbrace{1+\frac{\sqrt{6}}{3},\cdots,1+\frac{\sqrt{6}}{3}}_{l_{2}+1}$.
\\If $l_2\geq1$, then Lemma \ref{le:1} implies that $\lambda_{2}(B_{7})=1-\frac{\sqrt{6}}{3}$.
Analogously, when $l_2=0$, $\lambda_{2}(B_{7})$ $\geq1-\frac{\sqrt{6}}{3}$.
\\{\bf Subcase 1.3.} If $n_1=2$, $n_2=4$, then $n_3\geq15$ since $n\geq21$.
If $n_3\geq23$, by Lemmas \ref{le:5} and \ref{le:2}, we have $\lambda_{2}(C_3(T_1,T_2,T_3))\leq\lambda_{2}(S_3(2,4,23))\doteq0.18321<1-\frac{\sqrt{6}}{3}$. Therefore, in order to make
$\lambda_{2}(C_3(T_1,T_2,T_3)))\geq1-\frac{\sqrt{6}}{3}$, we have $15\leq n_3\leq22$.
Let $d$ be the length of the longest path from $v_3$ to the pendent vertex in $T_3$. In order to make
$\lambda_{2}(C_3(T_1,T_2,T_3))\geq1-\frac{\sqrt{6}}{3}$, then $d\leq2$. If $d\geq3$, then Lemma \ref{le:2} implies that
$\lambda_{2}(C_3(T_1,T_2,T_3))\leq\lambda_{2}(B_{1})\doteq0.16989<1-\frac{\sqrt{6}}{3}$, where $B_{1}$ is shown in Figure \ref{Fig.29.}.
Similarly, we have the length of the longest path from $v_2$ to the pendent vertex in $T_2$ is also at most 2. Since $\lambda_{2}(B_{8})\doteq0.17839<1-\frac{\sqrt{6}}{3}$ and $\lambda_{2}(B_{9})\doteq0.17595<1-\frac{\sqrt{6}}{3}$, where
$B_{8}$ and $B_{9}$ are shown in Figure \ref{Fig.29.}, thus by Lemmas \ref{le:3}
and \ref{le:2}, we have $C_3(T_1,T_2,T_3)$ is the unicyclic graph $B_{10}$ as shown in Figure \ref{Fig.29.}, where $l_i$ ($0\leq i\leq 1$) are nonnegative integers
and $21\leq l_0+2l_1+7=n\leq28$. If $l_1\geq4$, then Lemmas \ref{le:3} and \ref{le:2} imply that
$\lambda_{2}(C_3(T_1,T_2,T_3))\leq\lambda_{2}(B_{11})\doteq0.18325<1-\frac{\sqrt{6}}{3}$,
where $B_{11}$ is shown in Figure \ref{Fig.30.}. Thus $l_1\leq3$.
\\{\bf Subcase 1.3.1.} $l_1=3$, then $8\leq l_0\leq15$.
\\If $8\leq l_0\leq9$, then Lemma \ref{le:2} implies that $\lambda_{2}(C_3(T_1,T_2,T_3))\geq\lambda_{2}(B_{12})=0.18456>1-\frac{\sqrt{6}}{3}$, where $B_{12}$ is shown in Figure \ref{Fig.30.}. Otherwise, if $15\geq l_0\geq10$, then Lemma \ref{le:2} implies that $\lambda_{2}(C_3(T_1,T_2,T_3))\leq\lambda_{2}(B_{13})\doteq0.18324<1-\frac{\sqrt{6}}{3}$, where $B_{13}$ is shown in Figure \ref{Fig.30.}.
\\{\bf Subcase 1.3.2.}  $l_1=2$, then $10\leq l_0\leq17$.
\\If $10\leq l_0\leq13$, then Lemma \ref{le:2} implies that $\lambda_{2}(C_3(T_1,T_2,T_3))\geq\lambda_{2}(B_{14})=0.18460>1-\frac{\sqrt{6}}{3}$, where $B_{14}$ is shown in Figure \ref{Fig.30.}. Otherwise, if $14\leq l_0\leq17$, then Lemma \ref{le:2} implies that $\lambda_{2}(C_3(T_1,T_2,T_3))\leq\lambda_{2}(B_{15})\doteq0.18323<1-\frac{\sqrt{6}}{3}$, where $B_{15}$ is shown in Figure \ref{Fig.30.}.
\\{\bf Subcase 1.3.3.}  $l_1=1$, then $12\leq l_0\leq19$.
\\If $12\leq l_0\leq17$, then Lemma \ref{le:2} implies that $\lambda_{2}(C_3(T_1,T_2,T_3))\geq\lambda_{2}(B_{16})=0.18465>1-\frac{\sqrt{6}}{3}$, where $B_{16}$ is shown in Figure \ref{Fig.30.}. Otherwise, if $19\geq l_0\geq18$, then Lemma \ref{le:2} implies that $\lambda_{2}(C_3(T_1,T_2,T_3))\leq\lambda_{2}(B_{17})\doteq0.18322<1-\frac{\sqrt{6}}{3}$, where $B_{17}$ is shown in Figure \ref{Fig.30.}.
\\{\bf Subcase 1.3.4.}  $l_1=0$, then $14\leq l_0\leq21$.
\\Then Lemma \ref{le:2} implies that $\lambda_{2}(C_3(T_1,T_2,T_3))\geq\lambda_{2}(S_3(2,4,22))=0.18471>1-\frac{\sqrt{6}}{3}$.
\begin{figure}[htbp]
  \centering
  \includegraphics[scale=0.55]{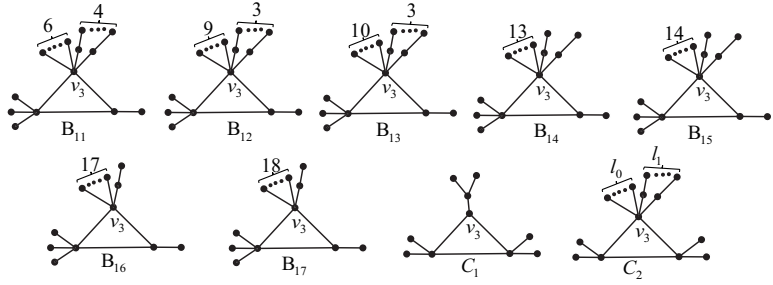}
  \caption{Unicyclic graphs $B_i$, $11\leq i\leq17$, $C_1$ and $C_2$ .}\label{Fig.30.}
\end{figure}
\\{\bf Subcase 1.4.} If $n_1=2$, $n_2=5$, then $n_3\geq14$ since $n\geq21$. Hence Lemmas \ref{le:5} and \ref{le:2} imply that $\lambda_{2}(C_3(T_1,T_2,T_3))\leq\lambda_{2}(S_3(2,5,14))\doteq0.17915<1-\frac{\sqrt{6}}{3}$.
\\{\bf Subcase 1.5.} If $n_1=2$, $n_3\geq n_2\geq 6$, then $n_3\geq10$ since $n\geq21$. Hence Lemmas \ref{le:5} and \ref{le:2} imply that $\lambda_{2}(C_3(T_1,T_2,T_3))\leq\lambda_{2}(S_3(2,6,10))\doteq0.18069<1-\frac{\sqrt{6}}{3}$.
\\{\bf Subcase 1.6.} If $n_1=n_2=3$, then $n_3\geq15$ since $n\geq21$.
If $n_3\geq32$, by Lemmas \ref{le:5} and \ref{le:2}, we have $\lambda_{2}(C_3(T_1,T_2,T_3))\leq\lambda_{2}(S_3(3,3,32))\doteq0.18326<1-\frac{\sqrt{6}}{3}$. Therefore, in order to make
$\lambda_{2}(C_3(T_1,T_2,T_3))\geq1-\frac{\sqrt{6}}{3}$, we have $15\leq n_3\leq31$.
Let $d$ be the length of the longest path from $v_3$ to the pendent vertex in $T_3$. In order to make
$\lambda_{2}(C_3(T_1,T_2,T_3))\geq1-\frac{\sqrt{6}}{3}$, then $d\leq2$. If $d\geq3$, then Lemma \ref{le:2} imply that
$\lambda_{2}(C_3(T_1,T_2,T_3))\leq\lambda_{2}(B_{1})\doteq0.16989<1-\frac{\sqrt{6}}{3}$, where $B_{1}$ is shown in Figure \ref{Fig.29.}.
 Since $\lambda_{2}(B_{5})\doteq0.18144<1-\frac{\sqrt{6}}{3}$ and $\lambda_{2}(C_{1})\doteq0.17464<1-\frac{\sqrt{6}}{3}$, where
$B_{5}$ and $C_{1}$ are shown in Figure \ref{Fig.29.} and Figure \ref{Fig.30.}, respectively. Thus by Lemmas \ref{le:3}
and \ref{le:2}, we have $C_3(T_1,T_2,T_3)$ is the unicyclic graph $C_{2}$ as shown in Figure \ref{Fig.30.}, where $l_i$ ($0\leq i\leq 1$) are nonnegative integers and $21\leq l_0+2l_1+7=n\leq37$. If $l_1\geq8$, then Lemmas \ref{le:3} and \ref{le:2} imply that
$\lambda_{2}(C_3(T_1,T_2,T_3))\leq\lambda_{2}(B_{18})\doteq0.18263<1-\frac{\sqrt{6}}{3}$,
where $B_{18}$ is shown in Figure \ref{Fig.31.}. Thus $l_1\leq7$.
\\{\bf Subcase 1.6.1.} $l_1=7$, then $0\leq l_0\leq16$.
\\If $0\leq l_0\leq2$, then Lemma \ref{le:2} implies that $\lambda_{2}(C_3(T_1,T_2,T_3))\geq\lambda_{2}(B_{19})=0.18400>1-\frac{\sqrt{6}}{3}$, where $B_{19}$ is shown in Figure \ref{Fig.31.}. Otherwise, if $16\geq l_0\geq3$, then Lemma \ref{le:2} implies that $\lambda_{2}(C_3(T_1,T_2,T_3))\leq\lambda_{2}(B_{20})\doteq0.18329<1-\frac{\sqrt{6}}{3}$, where $B_{20}$ is shown in Figure \ref{Fig.31.}.
\\{\bf Subcase 1.6.2.}  $l_1=6$, then $2\leq l_0\leq18$.
\\If $2\leq l_0\leq6$, then Lemma \ref{le:2} implies that $\lambda_{2}(C_3(T_1,T_2,T_3))\geq\lambda_{2}(B_{21})=0.18401>1-\frac{\sqrt{6}}{3}$, where $B_{21}$ is shown in Figure \ref{Fig.31.}. Otherwise, if $18\geq l_0\geq7$, then Lemma \ref{le:2} implies that $\lambda_{2}(C_3(T_1,T_2,T_3))\leq\lambda_{2}(B_{22})\doteq0.18329<1-\frac{\sqrt{6}}{3}$, where $B_{22}$ is shown in Figure \ref{Fig.31.}.
\\{\bf Subcase 1.6.3.}  $l_1=5$, then $4\leq l_0\leq20$.
\\If $4\leq l_0\leq10$, then Lemma \ref{le:2} implies that $\lambda_{2}(C_3(T_1,T_2,T_3))\geq\lambda_{2}(B_{23})=0.18403>1-\frac{\sqrt{6}}{3}$, where $B_{23}$ is shown in Figure \ref{Fig.31.}. Otherwise, if $20\geq l_0\geq11$, then Lemma \ref{le:2} implies that $\lambda_{2}(C_3(T_1,T_2,T_3))\leq\lambda_{2}(B_{24})\doteq0.18328<1-\frac{\sqrt{6}}{3}$, where $B_{24}$ is shown in Figure \ref{Fig.31.}.
\begin{figure}[htbp]
  \centering
  \includegraphics[scale=0.55]{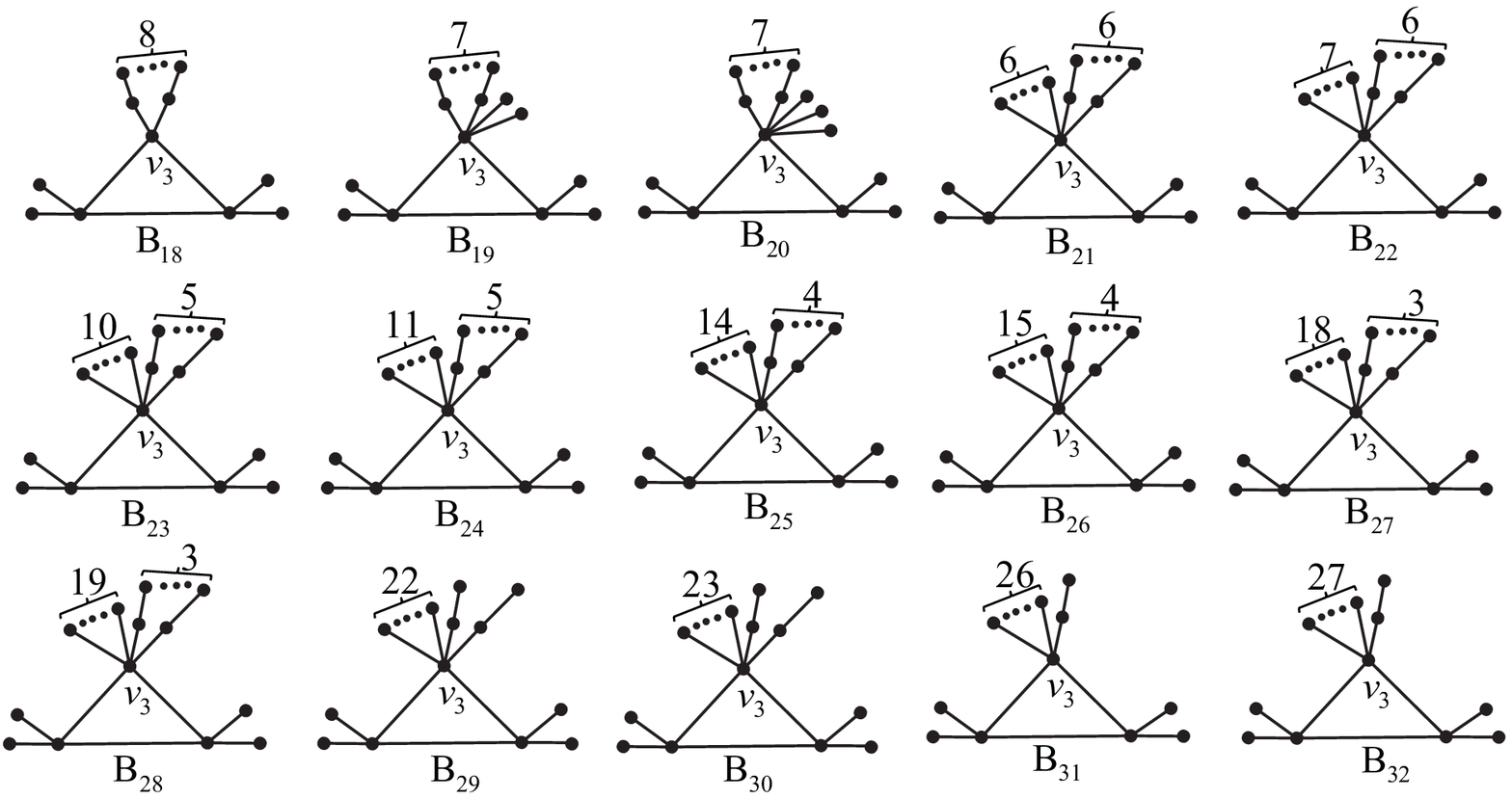}
  \caption{Unicyclic graphs $B_i$, $18\leq i\leq32$.}\label{Fig.31.}
\end{figure}
\\{\bf Subcase 1.6.4.}  $l_1=4$, then $6\leq l_0\leq22$.
\\If $6\leq l_0\leq14$, then Lemma \ref{le:2} implies that $\lambda_{2}(C_3(T_1,T_2,T_3))\geq\lambda_{2}(B_{25})=0.18404>1-\frac{\sqrt{6}}{3}$, where $B_{25}$ is shown in Figure \ref{Fig.31.}. Otherwise, if $22\geq l_0\geq15$, then Lemma \ref{le:2} implies that $\lambda_{2}(C_3(T_1,T_2,T_3))\leq\lambda_{2}(B_{26})\doteq0.18328<1-\frac{\sqrt{6}}{3}$, where $B_{26}$ is shown in Figure \ref{Fig.31.}.
\\{\bf Subcase 1.6.5.}  $l_1=3$, then $8\leq l_0\leq24$.
\\If $8\leq l_0\leq18$, then Lemma \ref{le:2} implies that $\lambda_{2}(C_3(T_1,T_2,T_3))\geq\lambda_{2}(B_{27})=0.18405>1-\frac{\sqrt{6}}{3}$, where $B_{27}$ is shown in Figure \ref{Fig.31.}. Otherwise, if $24\geq l_0\geq19$, then Lemma \ref{le:2} implies that $\lambda_{2}(C_3(T_1,T_2,T_3))\leq\lambda_{2}(B_{28})\doteq0.18327<1-\frac{\sqrt{6}}{3}$, where $B_{28}$ is shown in Figure \ref{Fig.31.}.
\\{\bf Subcase 1.6.6.}  $l_1=2$, then $10\leq l_0\leq26$.
\\If $11\leq l_0\leq22$, then Lemma \ref{le:2} implies that $\lambda_{2}(C_3(T_1,T_2,T_3))\geq\lambda_{2}(B_{29})=0.18406>1-\frac{\sqrt{6}}{3}$, where $B_{29}$ is shown in Figure \ref{Fig.31.}. Otherwise, if $26\geq l_0\geq23$, then Lemma \ref{le:2} implies that $\lambda_{2}(C_3(T_1,T_2,T_3))\leq\lambda_{2}(B_{30})\doteq0.18327<1-\frac{\sqrt{6}}{3}$, where $B_{30}$ is shown in Figure \ref{Fig.31.}.
\\{\bf Subcase 1.6.7.}  $l_1=1$, then $12\leq l_0\leq28$.
\\If $12\leq l_0\leq26$, then Lemma \ref{le:2} implies that $\lambda_{2}(C_3(T_1,T_2,T_3))\geq\lambda_{2}(B_{31})=0.18408>1-\frac{\sqrt{6}}{3}$, where $B_{31}$ is shown in Figure \ref{Fig.31.}. Otherwise, if $28\geq l_0\geq27$, then Lemma \ref{le:2} implies that $\lambda_{2}(C_3(T_1,T_2,T_3))\leq\lambda_{2}(B_{32})\doteq0.18326<1-\frac{\sqrt{6}}{3}$, where $B_{32}$ is shown in Figure \ref{Fig.31.}.
\\{\bf Subcase 1.6.8.}  $l_1=0$, then $14\leq l_0\leq30$.
\\Then Lemma \ref{le:2} implies that $\lambda_{2}(C_3(T_1,T_2,T_3))\geq\lambda_{2}(S_3(3,3,31))=0.18409>1-\frac{\sqrt{6}}{3}$.
\\{\bf Subcase 1.7.} If $n_1=3$, $n_2=4$, then $n_3\geq14$ since $n\geq21$. If $n_3\geq17$, by Lemmas \ref{le:5} and \ref{le:2}, we have $\lambda_{2}(C_3(T_1,T_2,T_3))\leq\lambda_{2}(S_3(3,4,17))\doteq0.18308<1-\frac{\sqrt{6}}{3}$. Therefore, in order to make
$\lambda_{2}(C_3(T_1,T_2,T_3))\geq1-\frac{\sqrt{6}}{3}$, we have $14\leq n_3\leq16$.
Let $d$ be the length of the longest path from $v_3$ to the pendent vertex in $T_3$. In order to make
$\lambda_{2}(C_3(T_1,T_2,T_3))\geq1-\frac{\sqrt{6}}{3}$, then $d\leq2$. If $d\geq3$, then Lemma \ref{le:2} implies that
$\lambda_{2}(C_3(T_1,T_2,T_3))\leq\lambda_{2}(B_{1})\doteq0.16989<1-\frac{\sqrt{6}}{3}$, where $B_{1}$ is shown in Figure \ref{Fig.29.}.
Similarly, we have the length of the longest path from $v_2$ to the pendent vertex in $T_2$ is also at most 2. Since $\lambda_{2}(B_{8})\doteq0.17839<1-\frac{\sqrt{6}}{3}$ and $\lambda_{2}(B_{9})\doteq0.17595<1-\frac{\sqrt{6}}{3}$, where
$B_{8}$ and $B_{9}$ are shown in Figure \ref{Fig.29.}. Thus by Lemmas \ref{le:3}
and \ref{le:2}, we have $C_3(T_1,T_2,T_3)$ is the unicyclic graph $B_{33}$ as shown in Figure \ref{Fig.32.}, where $l_i$ ($0\leq i\leq 1$) are nonnegative integers and $21\leq l_0+2l_1+8=n\leq23$. If $l_1\geq2$, then Lemmas \ref{le:3} and \ref{le:2} imply that
$\lambda_{2}(C_3(T_1,T_2,T_3))\leq\lambda_{2}(B_{34})\doteq0.18082<1-\frac{\sqrt{6}}{3}$,
where $B_{34}$ is shown in Figure \ref{Fig.32.}. Thus $l_1\leq1$.
\\{\bf Subcase 1.7.1.} $l_1=1$, then $11\leq l_0\leq13$.
\\If $l_0=11$, then Lemma \ref{le:2} implies that $\lambda_{2}(C_3(T_1,T_2,T_3))\geq\lambda_{2}(B_{35})=0.18582>1-\frac{\sqrt{6}}{3}$, where $B_{35}$ is shown in Figure \ref{Fig.32.}. Otherwise, if $13\geq l_0\geq12$, then Lemma \ref{le:2} implies that $\lambda_{2}(C_3(T_1,T_2,T_3))\leq\lambda_{2}(B_{36})\doteq0.18311<1-\frac{\sqrt{6}}{3}$, where $B_{36}$ is shown in Figure \ref{Fig.32.}.
\\{\bf Subcase 1.7.2.}  $l_1=0$, then $13\leq l_0\leq15$.
\\Then Lemma \ref{le:2} implies that $\lambda_{2}(C_3(T_1,T_2,T_3))\geq\lambda_{2}(S_3(3,4,16))=0.18604>1-\frac{\sqrt{6}}{3}$.
\begin{figure}[htbp]
  \centering
  \includegraphics[scale=0.55]{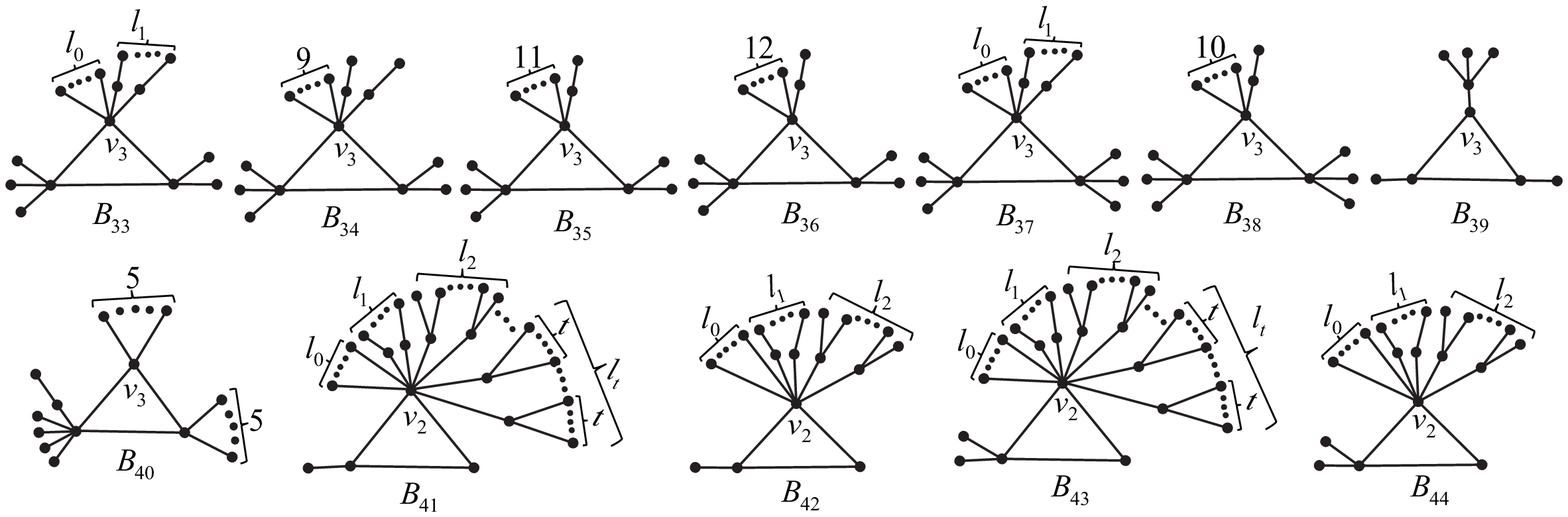}
  \caption{Unicyclic graphs $B_i$, $33\leq i\leq44$.}\label{Fig.32.}
\end{figure}
\\{\bf Subcase 1.8.} If $n_1=3$, $5\leq n_2\leq6$, then $n_3\geq12$ since $n\geq21$. Hence Lemmas \ref{le:5} and \ref{le:2} imply that $\lambda_{2}(C_3(T_1,T_2,T_3))\leq\lambda_{2}(S_3(3,5,12))\doteq0.18296<1-\frac{\sqrt{6}}{3}$.
\\{\bf Subcase 1.9.} If $n_1=3$, $7\leq n_2\leq8$, then $n_3\geq10$ since $n\geq21$. Hence Lemmas \ref{le:5} and \ref{le:2} imply that $\lambda_{2}(C_3(T_1,T_2,T_3))\leq\lambda_{2}(S_3(3,6,10))\doteq0.17916<1-\frac{\sqrt{6}}{3}$.
\\{\bf Subcase 1.10.} If $n_1=3$, $n_3\geq n_2\geq9$. Hence Lemmas \ref{le:5} and \ref{le:2} imply that $\lambda_{2}(C_3(T_1,T_2,$ $T_3))\leq\lambda_{2}(S_3(3,9,9))\doteq0.15418<1-\frac{\sqrt{6}}{3}$.
\\{\bf Subcase 1.11.} If $n_1=n_2=4$, then $n_3\geq13$ since $n\geq21$. If $n_3\geq14$, by Lemmas \ref{le:5} and \ref{le:2}, we have $\lambda_{2}(C_3(T_1,T_2,T_3))\leq\lambda_{2}(S_3(4,4,14))\doteq0.18029<1-\frac{\sqrt{6}}{3}$. Therefore, in order to make
$\lambda_{2}(C_3(T_1,T_2,T_3))\geq1-\frac{\sqrt{6}}{3}$, we have $n_3=13$. Similar as the Subcase 1.7, we have $C_3(T_1,T_2,T_3)$ is the unicyclic graph $B_{37}$ as shown in Figure \ref{Fig.32.}, where $l_i$ ($0\leq i\leq 1$) are nonnegative integers and $l_0+2l_1+9=21$. If $l_1\geq1$, then Lemmas \ref{le:3} and \ref{le:2} imply that
$\lambda_{2}(C_3(T_1,T_2,T_3))\leq\lambda_{2}(B_{38})\doteq0.17700<1-\frac{\sqrt{6}}{3}$,
where $B_{38}$ is shown in Figure \ref{Fig.32.}. Thus $l_1=0$, then $l_0=12$, $\lambda_{2}(C_3(T_1,T_2,T_3))=\lambda_{2}(S_3(4,4,13))\doteq0.18484>1-\frac{\sqrt{6}}{3}$.
\\{\bf Subcase 1.12.} If $n_1=4$, $5\leq n_2\leq6$, then $n_3\geq11$ since $n\geq21$. Hence Lemmas \ref{le:5} and \ref{le:2} imply that $\lambda_{2}(C_3(T_1,T_2,T_3))\leq\lambda_{2}(S_3(4,5,11))\doteq0.18282<1-\frac{\sqrt{6}}{3}$.
\\{\bf Subcase 1.13.} If $n_1=4$, $n_3\geq n_2\geq7$, then $n_3\geq9$ since $n\geq21$. Hence Lemmas \ref{le:5} and \ref{le:2} imply that $\lambda_{2}(C_3(T_1,T_2,T_3))\leq\lambda_{2}(S_3(4,7,9))\doteq0.17316<1-\frac{\sqrt{6}}{3}$.
\\{\bf Subcase 1.14.} If $n_1=5$, $5\leq n_2\leq6$, then $n_3\geq10$ since $n\geq21$. Hence Lemmas \ref{le:5} and \ref{le:2} imply that $\lambda_{2}(C_3(T_1,T_2,T_3))\leq\lambda_{2}(S_3(5,5,10))\doteq0.18267<1-\frac{\sqrt{6}}{3}$.
\\{\bf Subcase 1.15.} If $n_1=5$, $n_3\geq n_2\geq7$, then $n_3\geq8$ since $n\geq21$. Hence Lemmas \ref{le:5} and \ref{le:2} imply that $\lambda_{2}(C_3(T_1,T_2,T_3))\leq\lambda_{2}(S_3(5,7,8))\doteq0.18239<1-\frac{\sqrt{6}}{3}$.
\\{\bf Subcase 1.16.} If $n_1=n_2=6$, then $n_3\geq9$ since $n\geq21$. Hence Lemmas \ref{le:5} and \ref{le:2} imply that $\lambda_{2}(C_3(T_1,T_2,T_3))\leq\lambda_{2}(S_3(6,6,9))\doteq0.17781<1-\frac{\sqrt{6}}{3}$.
\\{\bf Subcase 1.17.} If $n_1=6$, $n_3\geq n_2\geq7$, then $n_3\geq8$ since $n\geq21$. Hence Lemmas \ref{le:5} and \ref{le:2} imply that $\lambda_{2}(C_3(T_1,T_2,T_3))\leq\lambda_{2}(S_3(6,7,8))\doteq0.18168<1-\frac{\sqrt{6}}{3}$.
\\{\bf Subcase 1.18.} If $n_3\geq n_2\geq n_1\geq7$. If $n_3\geq8$, by Lemmas \ref{le:5} and \ref{le:2}, we have $\lambda_{2}(C_3(T_1,T_2,T_3))$ $\leq\lambda_{2}(S_3(7,7,8))\doteq0.17940<1-\frac{\sqrt{6}}{3}$. Therefore, in order to make
$\lambda_{2}(C_3(T_1,T_2,T_3))\geq1-\frac{\sqrt{6}}{3}$, we have $n_3=7$. Since
$\lambda_{2}(B_{1})\doteq0.16989<1-\frac{\sqrt{6}}{3}$, $\lambda_{2}(C_{1})\doteq0.17464<1-\frac{\sqrt{6}}{3}$,
$\lambda_{2}(B_{39})\doteq0.16480<1-\frac{\sqrt{6}}{3}$ and $\lambda_{2}(B_{40})\doteq0.17911<1-\frac{\sqrt{6}}{3}$,
where $B_{1}$, $C_{1}$, $B_{39}$ and $B_{40}$ are shown in Figures \ref{Fig.29.}, \ref{Fig.30.} and \ref{Fig.32.}, respectively, thus by
Lemmas \ref{le:3} and \ref{le:2}, in order to make
$\lambda_{2}(C_3(T_1,T_2,T_3))\geq1-\frac{\sqrt{6}}{3}$,  we have $C_3(T_1,T_2,T_3)\cong S_3(7,7,7)$. Then
$\lambda_{2}(C_3(T_1,T_2,T_3))=\lambda_{2}(S_3(7,7,7)\doteq0.19422>1-\frac{\sqrt{6}}{3}$.
\\{\bf Case 2.} $|N|=2$. We assume that $n_1,n_2\in\{n_1,n_2,n_3\}$, and $2\leq n_1\leq n_2$.
\\{\bf Subcase 2.1.} If $n_1=2$, then $n_2\geq18$ since $n\geq21$.
Let $d$ be the length of the longest path from $v_2$ to the pendent vertex in $T_2$. In order to make
$\lambda_{2}(C_3(T_1,T_2,S_1))\geq1-\frac{\sqrt{6}}{3}$, then $d\leq2$. If $d\geq3$, then Lemma \ref{le:2} implies that
$\lambda_{2}(C_3(T_1,T_2,S_1))\leq\lambda_{2}(B_{1})\doteq0.16989<1-\frac{\sqrt{6}}{3}$, where $B_{1}$ is shown in Figure \ref{Fig.29.}.
Thus $C_3(T_1,T_2,T_3)$ is the unicyclic graph $B_{41}$ as shown in Figure \ref{Fig.32.}, where $l_i$ ($0\leq i\leq t$) are nonnegative integers
and $l_0+2l_1+3l_2+\cdots+(t+1)l_t+4=n\geq21$. If $l_3+l_4+\cdots+l_t\geq1$, then Lemma \ref{le:2} implies that
$\lambda_{2}(C_3(T_1,T_2,S_1))\leq\lambda_{2}(B_{3})\doteq0.17985<1-\frac{\sqrt{6}}{3}$, where $B_{3}$ is shown in Figure \ref{Fig.29.}.
Thus $l_3+l_4+\cdots+l_t=0$. Therefore $C_3(T_1,T_2,S_1)$ is the unicyclic graph $B_{42}$ as shown in Figure \ref{Fig.32.}, here $l_i$ ($0\leq i\leq 2$) are nonnegative integers and $l_0+2l_1+3l_2+4=n\geq21$. Recall that $v_2$ is a vertex degree $l_0+l_1+l_2+2$ in $B_{42}$. Note that the eigenvalues of
$\mathcal{L}_{v_2}(B_{42})$ are $\underbrace{1-\frac{\sqrt{6}}{3},\cdots,1-\frac{\sqrt{6}}{3}}_{l_{2}}$, $\underbrace{1-\frac{\sqrt{2}}{2},\cdots,1-\frac{\sqrt{2}}{2}}_{l_{1}+1}$, $\underbrace{1,\cdots,1}_{l_{0}+l_{2}+1}$, $\underbrace{1+\frac{\sqrt{2}}{2},\cdots,1+\frac{\sqrt{2}}{2}}_{l_{1}+1}$,  $\underbrace{1+\frac{\sqrt{6}}{3},\cdots,1+\frac{\sqrt{6}}{3}}_{l_{2}}$.
\\If $l_2\geq2$, then Lemma \ref{le:1} implies that $\lambda_{2}(B_{42})=1-\frac{\sqrt{6}}{3}$.
Analogously, when $l_2=1$, $\lambda_{2}(B_{42})\geq1-\frac{\sqrt{6}}{3}$, and when $l_2=0$, $\lambda_{2}(B_{42})>1-\frac{\sqrt{6}}{3}$.
\\{\bf Subcase 2.2.} If $n_1=3$, then $n_2\geq17$ since $n\geq21$.
Let $d$ be the length of the longest path from $v_2$ to the pendent vertex in $T_3$. In order to make
$\lambda_{2}(C_3(T_1,T_2,S_1))\geq1-\frac{\sqrt{6}}{3}$, then $d\leq2$. If $d\geq3$, then Lemma \ref{le:2} implies that
$\lambda_{2}(C_3(T_1,T_2,S_1))\leq\lambda_{2}(B_{1})\doteq0.16989<1-\frac{\sqrt{6}}{3}$, where $B_{1}$ is shown in Figure \ref{Fig.29.}.
 If $T_1$ is a path of length 2, then Lemmas \ref{le:3} and \ref{le:2} imply that
$\lambda_{2}(C_3(T_1,T_2,S_1))\leq\lambda_{2}(B_{5})\doteq0.18144<1-\frac{\sqrt{6}}{3}$, where $B_{5}$ is shown in Figure \ref{Fig.29.}.
Thus $C_3(T_1,T_2,S_1)$ is the unicyclic graph $B_{43}$ as shown in Figure \ref{Fig.32.}, where $l_i$ ($0\leq i\leq t$) are nonnegative integers
and $l_0+2l_1+3l_2+\cdots+(t+1)l_t+5=n\geq21$. If $l_3+l_4+\cdots+l_t\geq1$, then Lemma \ref{le:2} implies that
$\lambda_{2}(C_3(T_1,T_2,S_1))\leq\lambda_{2}(B_{3})\doteq0.17985<1-\frac{\sqrt{6}}{3}$, where $B_{3}$ is shown in Figure \ref{Fig.29.}.
Thus $l_3+l_4+\cdots+l_t=0$. Therefore $C_3(T_1,T_2,S_1)$ is the unicyclic graph $B_{44}$ as shown in Figure \ref{Fig.32.}, here $l_i$ ($0\leq i\leq 2$) are nonnegative integers and $l_0+2l_1+3l_2+5=n\geq21$. Recall that $v_2$ is a vertex degree $l_0+l_1+l_2+2$ in $B_{44}$. Note that the eigenvalues of
$\mathcal{L}_{v_2}(B_{44})$ are $\underbrace{1-\frac{\sqrt{6}}{3},\cdots,1-\frac{\sqrt{6}}{3}}_{l_{2}}$, 0.20943 $\underbrace{1-\frac{\sqrt{2}}{2},\cdots,1-\frac{\sqrt{2}}{2}}_{l_{1}}$, $\underbrace{1,\cdots,1}_{l_{0}+l_{2}+2}$, $\underbrace{1+\frac{\sqrt{2}}{2},\cdots,1+\frac{\sqrt{2}}{2}}_{l_{1}}$, 1.79057, $\underbrace{1+\frac{\sqrt{6}}{3},\cdots,1+\frac{\sqrt{6}}{3}}_{l_{2}}$.
\\If $l_2\geq2$, then Lemma \ref{le:1} implies that $\lambda_{2}(B_{44})=1-\frac{\sqrt{6}}{3}$.
Analogously, when $l_2=1$, $\lambda_{2}(B_{44})\geq1-\frac{\sqrt{6}}{3}$, and when $l_2=0$, $\lambda_{2}(B_{44})>1-\frac{\sqrt{6}}{3}$.
\\{\bf Subcase 2.3.} If $n_1=4$, then $n_2\geq16$ since $n\geq21$.
If $n_2\geq35$, by Lemmas \ref{le:5} and \ref{le:2}, we have $\lambda_{2}(C_3(T_1,T_2,S_1))\leq\lambda_{2}(S_3(4,35,1))\doteq0.18333<1-\frac{\sqrt{6}}{3}$. Therefore, in order to make
$\lambda_{2}(C_3(T_1,T_2,S_1))\geq1-\frac{\sqrt{6}}{3}$, we have $16\leq n_2\leq34$.
Let $d$ be the length of the longest path from $v_2$ to the pendent vertex in $T_2$. In order to make
$\lambda_{2}(C_3(T_1,T_2,S_1))\geq1-\frac{\sqrt{6}}{3}$, then $d\leq2$. If $d\geq3$, then Lemma \ref{le:2} implies that
$\lambda_{2}(C_3(T_1,T_2,S_1))\leq\lambda_{2}(B_{1})\doteq0.16989<1-\frac{\sqrt{6}}{3}$, where $B_{1}$ is shown in Figure \ref{Fig.29.}.
Similarly, we have the length of the longest path from $v_1$ to the pendent vertex in $T_1$ is also at most 2. Since $\lambda_{2}(B_{8})\doteq0.17839<1-\frac{\sqrt{6}}{3}$ and $\lambda_{2}(B_{9})\doteq0.17595<1-\frac{\sqrt{6}}{3}$, where
$B_{8}$ and $B_{9}$ are shown in Figure \ref{Fig.29.}, thus by Lemmas \ref{le:3}
and \ref{le:2}, we have $C_3(T_1,T_2,S_1)$ is the unicyclic graph $B_{45}$ as shown in Figure \ref{Fig.34.}, where $l_i$ ($0\leq i\leq 1$) are nonnegative integers
and $21\leq l_0+2l_1+6=n\leq39$. If $l_1\geq9$, then Lemmas \ref{le:3} and \ref{le:2} imply that
$\lambda_{2}(C_3(T_1,T_2,S_1))\leq\lambda_{2}(B_{46})\doteq0.18241<1-\frac{\sqrt{6}}{3}$,
where $B_{46}$ is shown in Figure \ref{Fig.34.}. Thus $l_1\leq8$.
\begin{figure}[htbp]
  \centering
  \includegraphics[scale=0.55]{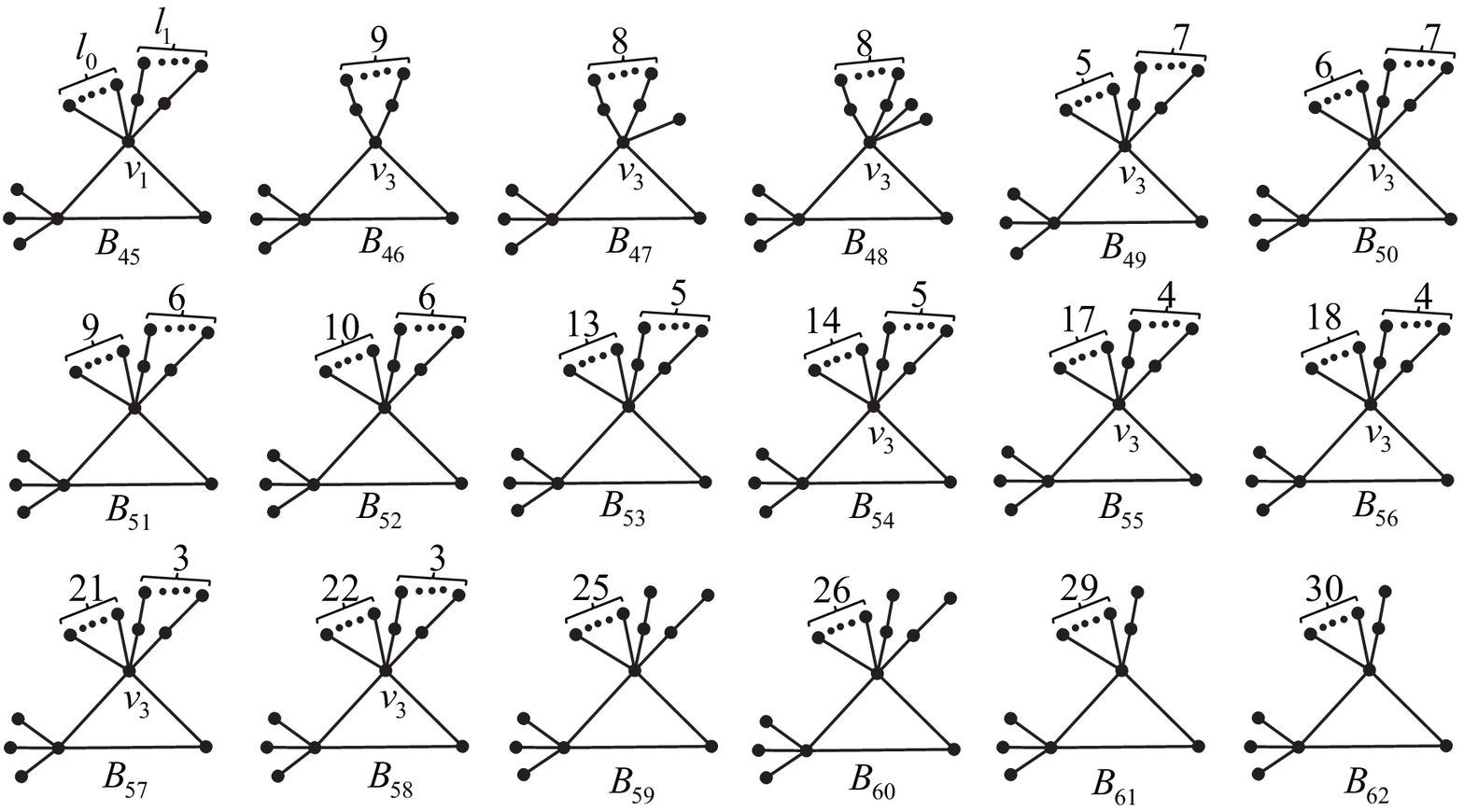}
  \caption{Unicyclic graphs $B_i$, $45\leq i\leq62$.}\label{Fig.34.}
\end{figure}
\\{\bf Subcase 2.3.1.} $l_1=8$, then $0\leq l_0\leq17$.
\\If $0\leq l_0\leq1$, then Lemma \ref{le:2} implies that $\lambda_{2}(C_3(T_1,T_2,S_1))\geq\lambda_{2}(B_{47})=0.18386>1-\frac{\sqrt{6}}{3}$, where $B_{47}$ is shown in Figure \ref{Fig.34.}. Otherwise, if $2\leq l_0\leq17$, then Lemma \ref{le:2} implies that $\lambda_{2}(C_3(T_1,T_2,S_1))\leq\lambda_{2}(B_{48})\doteq0.18335<1-\frac{\sqrt{6}}{3}$, where $B_{48}$ is shown in Figure \ref{Fig.34.}.
\\{\bf Subcase 2.3.2.} $l_1=7$, then $1\leq l_0\leq19$.
\\If $1\leq l_0\leq5$, then Lemma \ref{le:2} implies that $\lambda_{2}(C_3(T_1,T_2,S_1))\geq\lambda_{2}(B_{49})=0.18387>1-\frac{\sqrt{6}}{3}$, where $B_{49}$ is shown in Figure \ref{Fig.34.}. Otherwise, if $6\leq l_0\leq19$, then Lemma \ref{le:2} implies that $\lambda_{2}(C_3(T_1,T_2,S_1))\leq\lambda_{2}(B_{50})\doteq0.18335<1-\frac{\sqrt{6}}{3}$, where $B_{50}$ is shown in Figure \ref{Fig.34.}.
\\{\bf Subcase 2.3.3.} $l_1=6$, then $3\leq l_0\leq21$.
\\If $3\leq l_0\leq9$, then Lemma \ref{le:2} implies that $\lambda_{2}(C_3(T_1,T_2,S_1))\geq\lambda_{2}(B_{51})=0.18387>1-\frac{\sqrt{6}}{3}$, where $B_{51}$ is shown in Figure \ref{Fig.34.}. Otherwise, if $10\leq l_0\leq21$, then Lemma \ref{le:2} implies that $\lambda_{2}(C_3(T_1,T_2,S_1))\leq\lambda_{2}(B_{52})\doteq0.18335<1-\frac{\sqrt{6}}{3}$, where $B_{52}$ is shown in Figure \ref{Fig.34.}.
\\{\bf Subcase 2.3.4.} $l_1=5$, then $5\leq l_0\leq23$.
\\If $5\leq l_0\leq13$, then Lemma \ref{le:2} implies that $\lambda_{2}(C_3(T_1,T_2,S_1))\geq\lambda_{2}(B_{53})=0.18388>1-\frac{\sqrt{6}}{3}$, where $B_{53}$ is shown in Figure \ref{Fig.34.}. Otherwise, if $14\leq l_0\leq23$, then Lemma \ref{le:2} implies that $\lambda_{2}(C_3(T_1,T_2,S_1))\leq\lambda_{2}(B_{54})\doteq0.18334<1-\frac{\sqrt{6}}{3}$, where $B_{54}$ is shown in Figure \ref{Fig.34.}.
\\{\bf Subcase 2.3.5.} $l_1=4$, then $7\leq l_0\leq25$.
\\If $7\leq l_0\leq17$, then Lemma \ref{le:2} implies that $\lambda_{2}(C_3(T_1,T_2,S_1))\geq\lambda_{2}(B_{55})=0.18388>1-\frac{\sqrt{6}}{3}$, where $B_{55}$ is shown in Figure \ref{Fig.34.}. Otherwise, if $18\leq l_0\leq25$, then Lemma \ref{le:2} implies that $\lambda_{2}(C_3(T_1,T_2,S_1))\leq\lambda_{2}(B_{56})\doteq0.18334<1-\frac{\sqrt{6}}{3}$, where $B_{56}$ is shown in Figure \ref{Fig.34.}.
\\{\bf Subcase 2.3.6.} $l_1=3$, then $9\leq l_0\leq27$.
\\If $9\leq l_0\leq21$, then Lemma \ref{le:2} implies that $\lambda_{2}(C_3(T_1,T_2,S_1))\geq\lambda_{2}(B_{57})=0.18389>1-\frac{\sqrt{6}}{3}$, where $B_{57}$ is shown in Figure \ref{Fig.34.}. Otherwise, if $22\leq l_0\leq27$, then Lemma \ref{le:2} implies that $\lambda_{2}(C_3(T_1,T_2,S_1))\leq\lambda_{2}(B_{58})\doteq0.18334<1-\frac{\sqrt{6}}{3}$, where $B_{58}$ is shown in Figure \ref{Fig.34.}.
\\{\bf Subcase 2.3.7.} $l_1=2$, then $11\leq l_0\leq29$.
\\If $11\leq l_0\leq25$, then Lemma \ref{le:2} implies that $\lambda_{2}(C_3(T_1,T_2,S_1))\geq\lambda_{2}(B_{59})=0.18390>1-\frac{\sqrt{6}}{3}$, where $B_{59}$ is shown in Figure \ref{Fig.34.}. Otherwise, if $26\leq l_0\leq29$, then Lemma \ref{le:2} implies that $\lambda_{2}(C_3(T_1,T_2,S_1))\leq\lambda_{2}(B_{60})\doteq0.18334<1-\frac{\sqrt{6}}{3}$, where $B_{62}$ is shown in Figure \ref{Fig.34.}.
\\{\bf Subcase 2.3.8.} $l_1=1$, then $13\leq l_0\leq31$.
\\If $13\leq l_0\leq29$, then Lemma \ref{le:2} implies that $\lambda_{2}(C_3(T_1,T_2,S_1))\geq\lambda_{2}(B_{61})=0.18390>1-\frac{\sqrt{6}}{3}$, where $B_{61}$ is shown in Figure \ref{Fig.34.}. Otherwise, if $30\leq l_0\leq31$, then Lemma \ref{le:2} implies that $\lambda_{2}(C_3(T_1,T_2,S_1))\leq\lambda_{2}(B_{62})\doteq0.18333<1-\frac{\sqrt{6}}{3}$, where $B_{62}$ is shown in Figure \ref{Fig.34.}.
\\{\bf Subcase 2.3.9.}  $l_1=0$, then $15\leq l_0\leq33$.
\\Then Lemma \ref{le:2} implies that $\lambda_{2}(C_3(T_1,T_2,S_1))\geq\lambda_{2}(S_3(4,34,1))=0.18391>1-\frac{\sqrt{6}}{3}$.
\\{\bf Subcase 2.4.} If $5\leq n_1\leq6$, then $n_2\geq14$ since $n\geq21$. Hence Lemmas \ref{le:5} and \ref{le:2} imply that $\lambda_{2}(C_3(T_1,T_2,T_3))\leq\lambda_{2}(S_3(5,14,1))\doteq0.18316<1-\frac{\sqrt{6}}{3}$.
\\{\bf Subcase 2.5.} If $n_2\geq n_1\geq7$, then $n_2\geq10$ since $n\geq21$. Hence Lemmas \ref{le:5} and \ref{le:2} imply that $\lambda_{2}(C_3(T_1,T_2,T_3))\leq\lambda_{2}(S_3(7,10,1))\doteq0.16711<1-\frac{\sqrt{6}}{3}$.
\\{\bf Case 3.} $|N|=1$. Without loss of generality, we assume that $n_1\geq2$, then $n_1\geq19$ since $n\geq21$. That is $C_3(T_1,T_2,T_3)=C_3(T_1,S_1,S_1)$.
\begin{figure}[htbp]
  \centering
  \includegraphics[scale=0.55]{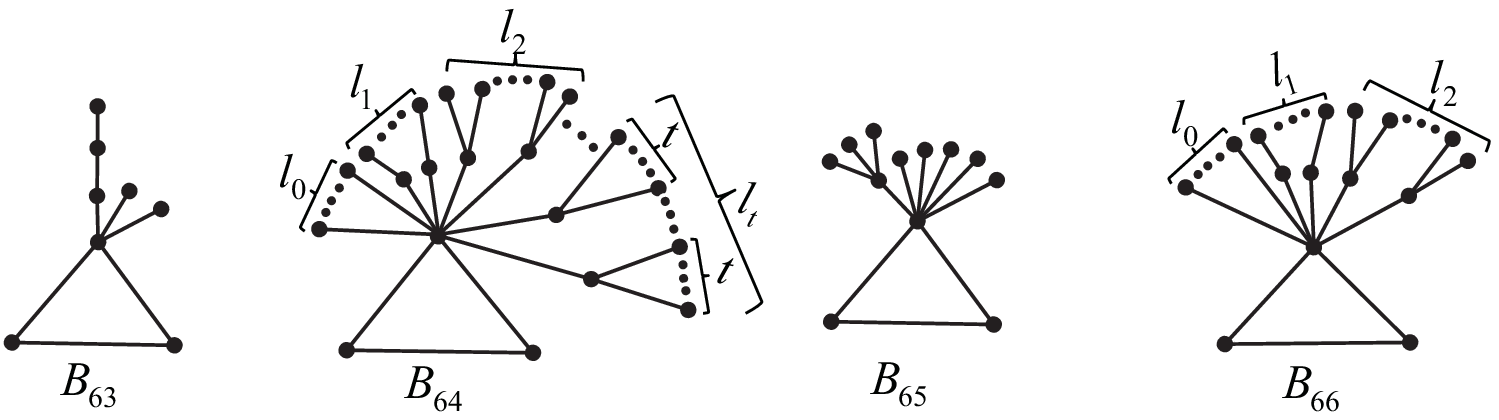}
  \caption{Unicyclic graphs $B_i$, $45\leq i\leq62$.}\label{Fig.35.}
\end{figure}let $d$ be the length of the longest
path from $v_1$ to the pendent vertex in $T_1$. In order to make
$\lambda_{2}(C_3(T_1,S_1,S_1))\geq1-\frac{\sqrt{6}}{3}$, then $d\leq2$. If $d\geq3$, then  Lemmas \ref{le:3} and \ref{le:2} imply that
$\lambda_{2}(C_3(T_1,S_1,S_1))\leq\lambda_{2}(B_{63})\doteq0.17815<1-\frac{\sqrt{6}}{3}$, where $B_{63}$ is shown in Figure \ref{Fig.35.}.
Thus $C_3(T_1,S_1,S_1)$ is the unicyclic graph $B_{64}$ as shown in Figure \ref{Fig.35.}, where $l_i$ ($0\leq i\leq t$) are nonnegative integers
and $l_0+2l_1+3l_2+\cdots+(t+1)l_t+3=n\geq21$. If $l_3+l_4+\cdots+l_t\geq1$, then  Lemmas \ref{le:3} and \ref{le:2} imply that
$\lambda_{2}(C_3(T_1,S_1,S_1))\leq\lambda_{2}(B_{65})\doteq0.17975<1-\frac{\sqrt{6}}{3}$, where $B_{65}$ is shown in Figure \ref{Fig.35.}. Thus $l_3+l_4+\cdots+l_t=0$. Therefore $C_3(T_1,S_1,S_1)$ is the unicyclic graph $B_{66}$ as shown in Figure \ref{Fig.35.}, here $l_i$ ($0\leq i\leq 2$) are nonnegative integers and $l_0+2l_1+3l_2+3=n\geq21$. Recall that $v_1$ is a vertex degree $l_0+l_1+l_2+2$ in $B_{66}$. Note that the eigenvalues of
$\mathcal{L}_{v_1}(B_{66})$ are $\underbrace{1-\frac{\sqrt{6}}{3},\cdots,1-\frac{\sqrt{6}}{3}}_{l_{2}}$, $\underbrace{1-\frac{\sqrt{2}}{2},\cdots,1-\frac{\sqrt{2}}{2}}_{l_{1}}$, 0.5 $\underbrace{1,\cdots,1}_{l_{0}+l_{2}}$, 1.5, $\underbrace{1+\frac{\sqrt{2}}{2},\cdots,1+\frac{\sqrt{2}}{2}}_{l_{1}}$, $\underbrace{1+\frac{\sqrt{6}}{3},\cdots,1+\frac{\sqrt{6}}{3}}_{l_{2}}$.
\\If $l_2\geq2$, then Lemma \ref{le:1} implies that $\lambda_{2}(B_{66})=1-\frac{\sqrt{6}}{3}$.
Analogously, when $l_2=1$, $\lambda_{2}(B_{66})\geq1-\frac{\sqrt{6}}{3}$, and
when $l_2=0$, $\lambda_{2}(B_{66})>1-\frac{\sqrt{6}}{3}$.
\end{proof}

%\section*{Acknowledgements}

%The authors are very grateful to the referees for many detailed
%comments and suggestions, which are very helpful for improving the
%presentation of the manuscript. Lemma 2.8 and Theorem 3.9 are due
%to an anonymous referee, to whom we would like to express our
%special thanks.


\begin{thebibliography}{99}

\bibitem{Butler} S. Butler, Eigenvalues and structures of graphs, Ph.D. dissertation, University of California, San Diego, 2008.

\bibitem{Chung} F.R.K.Chung, Spectral Graph Theory, American Math. Soc. Providence, 1997.

\bibitem{Cvetk} D. Cvetkovi$\acute{c}$, M. Doob, H. Sachs, Spectra of Graphs, Academic Press, New York, 1980.

\bibitem{HHYZ} H.H. Li, J.S. Li, Y.Z.Fan, The effect on the second smallest eigenvalue of the normalized Laplacian of a graph by grafting edges, Linear Multilinear Algebra, 56 (2008), 627-638.

\bibitem{HH} H.H. Li, J.S. Li, A note on the normalized Laplacian spectra, Taiwanese J. Math., 15 (2011), 129-139.

\bibitem{Jian} J.X. Li, J.M. Guo, W.C. Shiu, A. Chang, An edge-separating theorem on the second smallest normalized Laplacian eigenvalue of a graph and its applications. Discrete Appl. Math., 171 (2014), 104-115.

\bibitem{Jianxi} J.X. Li, J.M. Guo, W.C. Shiu, A. Chang, Six classes of trees with largest normalized algebraic connectivity. Linear Algebra Appl., 452 (2014), 318-327.

\bibitem{LG} J.X. Li, J.M. Guo, W.C. Shiu, A note on Randi\'{c} energy, MATCH Commun. Math. Comput. Chem., 74 (2015), 389-398.

\bibitem{LW} X.L. Li, J.F. Wang, Randi\'{c} energy and Randi\'{c} eigenvalues, MATCH Commun. Math. Comput. Chem., 73 (2015), 73-80.

\bibitem{Tian} X.G. Tian, L.G. Wang, The trees with the second smallest normalized Laplacian
eigenvalue at least $1-\frac{\sqrt{3}}{2}$, Discrete Appl. Math., 220 (2017), 118-133.






\end{thebibliography}
\end{document}